\documentclass[11pt, oneside]{article}

\usepackage{lipsum}
\usepackage{amsfonts, amssymb, amsthm, dsfont}
\usepackage{graphicx}
\usepackage{epstopdf}
\usepackage{algorithmic}
\usepackage{color}
\ifpdf
  \DeclareGraphicsExtensions{.eps,.pdf,.png,.jpg}
\else
  \DeclareGraphicsExtensions{.eps}
\fi

\usepackage[colorlinks=true, allcolors=blue]{hyperref}
\hypersetup{
 pdfauthor={H.~Do, J.~Feydy, O.~Mula},
 pdftitle={},
 pdfsubject={},
 pdfkeywords={}
}

\newtheorem{theorem}{Theorem}[section]
\newtheorem{lemma}[theorem]{Lemma}

\newtheorem{assumption}[theorem]{Assumption}

\newtheorem{remark}[theorem]{Remark}

\usepackage{amsopn}
\usepackage{mathtools}
\mathtoolsset{showonlyrefs=true}
\usepackage{caption}
\usepackage{subcaption}
\usepackage[left=0.9in,right=0.9in,top=1in,bottom=1in,includefoot,includehead,headheight=13.6pt]{geometry}

\usepackage[breakable]{tcolorbox}
\newtcolorbox{boxblue}[1]{breakable,colback=white,colframe=BlueMAP,fonttitle=\bfseries,title=#1,left=4pt,right=4pt,top=4pt,bottom=6pt}
\newtcolorbox{boxgreen}[1]{breakable,colback=white,colframe=red,fonttitle=\bfseries,title=#1,left=4pt,right=4pt,top=4pt,bottom=6pt}
\newtcolorbox{partie}{colback=white,colframe=BlueMAP,fonttitle=\bfseries}

\newcommand{\cA}{\ensuremath{\mathcal{A}}}
\newcommand{\cB}{\ensuremath{\mathcal{B}}}

\newcommand{\cG}{\ensuremath{\mathcal{G}}}
\newcommand{\cH}{\ensuremath{\mathcal{H}}}
\newcommand{\cI}{\ensuremath{\mathcal{I}}}

\newcommand{\cL}{\ensuremath{\mathcal{L}}}
\newcommand{\cM}{\ensuremath{\mathcal{M}}}

\newcommand{\cP}{\ensuremath{\mathcal{P}}}

\newcommand{\cX}{\ensuremath{\mathcal{X}}}
\newcommand{\cY}{\ensuremath{\mathcal{Y}}}
\newcommand{\cZ}{\ensuremath{\mathcal{Z}}}

\newcommand{\bN}{\ensuremath{\mathbb{N}}}

\newcommand{\bR}{\ensuremath{\mathbb{R}}}

\newcommand{\rL}{\ensuremath{\mathrm{L}}}

\newcommand{\rY}{\ensuremath{\mathrm{Y}}}

\newcommand{\pbM}{\ensuremath{\mathbf{M}}}

\newcommand{\x}{\ensuremath{\mathbf{x}}}
\newcommand{\y}{\ensuremath{\mathbf{y}}}
\newcommand{\z}{\ensuremath{\mathbf{z}}}
\newcommand{\e}{\ensuremath{\mathbf{e}}}

\newcommand{\balpha}{{\ensuremath{\boldsymbol \alpha}}}
\newcommand{\bbeta}{{\ensuremath{\boldsymbol \beta}}}
\newcommand{\bgamma}{{\ensuremath{\boldsymbol \gamma}}}

\def\[{\left[}
\def\]{\right]}
\def\<{\langle}
\def\>{\rangle}
\def\({\left(}
\def\){\right)}
\def\[{\left [}
\def\]{\right]}
\def\({\left(}
\def\){\right)}

\newcommand{\charFun}{\ensuremath{\mathds{1}}}

\newcommand{\rd}{\ensuremath{\mathrm d}}
\newcommand{\dx}{\ensuremath{\mathrm dx}}

\newcommand{\bary}{\ensuremath{\mathrm{Bar}}}

\newcommand{\cond}{\; :\;}

\usepackage{soul}

\title{Moment-SoS Methods for Optimal Transport Problems}
\author{Olga Mula, Anthony Nouy}

\begin{document}
\maketitle
\abstract{
Most common Optimal Transport (OT) solvers are currently based on an approximation of underlying measures by discrete measures.
However, it is sometimes relevant to work only with moments of measures instead of the measure itself, and many common OT problems can be formulated as moment problems (the most relevant examples being $L^p$-Wasserstein distances, barycenters, and Gromov-Wasserstein discrepancies on Euclidean spaces). We leverage this fact to develop a generalized moment formulation that covers these classes of OT problems. The transport plan is represented through its moments on a given basis, and the marginal constraints are expressed in terms of moment constraints. A practical computation then consists in considering a truncation of the involved moment sequences up to a certain order, and using the polynomial sums-of-squares hierarchy for measures supported on semi-algebraic sets. We prove that the strategy converges to the solution of the OT problem as the order increases. We also show how to approximate linear quantities of interest, and how to estimate the support of the optimal transport map from the computed moments using Christoffel-Darboux kernels. Numerical experiments illustrate the good behavior of the approach.}

\section{Introduction}

Optimal Transport provides a principled and versatile approach to work with probability distributions. In recent years, an increasing amount of theoretical results are being leveraged to build numerical solvers which are by now playing a fundamental role in numerous applications ranging from economy \cite{Carlier2021, Galichon2021}, quantum chemistry \cite{CFP2015, BCN2016}, gradient flow modeling \cite{CDPS2017}, and machine learning \cite{PC2019}. 

The prototypical example is the two-marginal Monge-Kantorovich problem: given  two Borel sets $\cX_1\subset \bR^{n_1}$ and $\cX_2\subset \bR^{n_2}$ and given two probability measures $\mu$ and $\nu$, solve 
\begin{equation}
\label{eq:MK}
\inf \biggl\{ \int_{\cX\times \cY} c(\x_1,\x_2) \rd\pi(\x_1,\x_2) \cond \pi \in \Pi(\cX_1 \times \cX_2; \mu , \nu) \biggr\}
\end{equation}
where $c:\cX_1\times \cX_2\mapsto \bR^+$ is a lower semi-continuous cost function. The infimum runs over the set $ \Pi(\cX_1 \times \cX_2; \mu , \nu) $ of coupling measures (usually called transport plans) on $\cX_1\times \cX_2$ with marginal distributions equal to $\mu$ and $\nu$ respectively. More generally, the problem can be posed on general Polish spaces. It can also be multi-marginal as we introduce later on in Section \ref{sec:OT}.

Numerous methods have been introduced for solving such problems in practice. Many algorithms {rely on an approximation of measures by discrete measures (by sampling or discretization on discrete grids)}, whose numerical cost quickly becomes prohibitive when the number of discretization points increases. 
Among the existing strategies to mitigate this effect, probably the most popular one is based on adding an entropic regularization to the loss function which is then solved with a Sinkhorn algorithm (\cite{Cuturi2013, CPSV2018, PC2019}). The algorithm is however still posed on a grid. Other approaches involving discrete grids are the auction algorithm \cite{BC1989}, numerical methods based on Laguerre cells \cite{GM2018}, multiscale algorithms \cite{Merigot2011, Schmitzer2016} and methods based on dynamic formulations \cite{BB2000}. Recently, an approach that dynamically discovers sampling points where the support of the solution measure lies has been introduced in \cite{FSV2022}. It provides promising results to address the curse of dimensionality when the optimal transport plans have sparse support. 

The present paper considers an entirely different avenue where the spatial discretization is replaced by a spectral discretization where the transport plan is represented through its moments on a given basis, and the marginal constraints are expressed in terms of moment constraints. A practical computation then consists in considering a truncation of the involved moment sequences up to a certain order. The procedure needs to be well posed in the sense that one needs to guarantee convergence to the original problem as the truncation order increases.

This type of approach is not entirely novel but it has only been explored in a few prior works despite that it is often relevant to work only moments of a measure instead of the measure itself. The idea was originally mentioned in \cite{Lasserre2008} for the case of a polynomial basis and polynomial cost function. The moment approach was also recently explored in \cite{Catala2020} for applications related to image processing involving trigonometric polynomial bases. It has also been recently used for solving the Monge problem \cite{henrion2022graph}, where an approximation of the transport map is constructed by solving a moment matrix completion problem and an approximation method based on Christoffel-Darboux kernel \cite{marx2021semi}. A relatively different contribution can be found in \cite{ACEL2020}, where the authors relax marginal constraints into a set of moment constraints but here moments do not come from a prescribed basis. Instead they are selected from a given dictionary involving potentially general test functions. Last but not least, the idea of leveraging moment formulations and sums-of-squares has also been used for the dual formulation of problem \eqref{eq:MK} in order to derive statistical estimation bounds of high dimensional OT problems (see \cite{VMRB2021, MVBVR2021}).

In view of the present state of the art, the main contribution of this paper
is to provide a general moment problem formulation of 
most common OT problems with polynomial or piecewise polynomial costs.
In particular, the problem of estimating $L^p$-Wasserstein distances for $p\geq 1$, barycenters, and Gromov-Wasserstein discrepancies on Euclidean spaces will be covered by our framework. We prove that the resulting sequence of optimal solutions converges to the whole moment sequence of the original OT measure as the polynomial order increases. The case of piecewise polynomial costs is addressed by a reformulation in terms of conditional measures.

For practical computations, we can directly apply the SoS-moment hierarchy similarly as in \cite{Lasserre2008, Catala2020}, which eventually boils down to solving semidefinite programming optimization problems. It is worth emphasizing that by switching the point of view to a moment problem,  we do not recover the measure itself. Instead, the resulting outputs will be moments of the optimal transport plan. Depending on the application, this may of course be a limitation. However, we show that it is possible to recover linear quantities of interest, and also the support of the measure by a post-processing algorithm based on Christoffel-Darboux kernels. Our numerical examples show that even the support of concentrated measures can efficiently be estimated with relatively low polynomial order. This feature seems particularly appealing. It could for instance be leveraged to recover optimal transport plans from high-dimensional OT problems: we could use the support estimation to provide well-chosen sampling points to grid-based approaches. A full development of these ideas will be presented in a forthcoming work.

The paper is organized as follows. After introducing some basic notation in Section \ref{sec:notation}, we define optimal optimal transport problems in Section \ref{sec:OT}. We prove that when the involved cost function is a polynomial or a piecewise polynomial, the problem can be interpreted as a generalized moment problem. This section will allow us to introduce the  basic principles of our approach to OT problems and  important results from real algebraic geometry.    
In Sections \ref{sec:wasserstein} and \ref{sec:gromov}, we consider
 optimal transport problems that are playing a crucial role in numerous applicative areas, and prove that they can be expressed as generalized moment problems. More precisely, we consider in section \ref{sec:wasserstein} the problems of computing $L^p$-Wasserstein distances and barycenters for $p\geq 1$, and in section \ref{sec:gromov} the problems of computing Gromov-Wasserstein discrepancies and corresponding barycenters. 
In Section \ref{sec:hierarchy}, we formulate a generalized moment problem that includes all previous OT problems. We derive a solution strategy based on the SoS-moment (or Lasserre's) hierarchy, and we prove its convergence to the solution of the OT problem. Section \ref{sec:postproc} explains how to postprocess the moments to estimate linear quantities of interest and the support of the measure. Section \ref{sec:numerics} illustrates the potential of the approach by giving numerical results on estimating the $L^1$ and $L^2$ Wasserstein distances, barycenters, and the $L^2$ Gromov-Wasserstein discrepancy.

\section{Some elements of notation}
\label{sec:notation}
In the following, $\bN$ should be understood as the set of non-negative integers (including zero). Vectors $\x$ from the Euclidean space $\bR^n$ will be denoted with bold notation. The coordinates $\x=(x_1,\dots, x_n)^T$ will be written with plain text. The canonical vectors will be denoted as $\e_i = (0,\dots, 0, 1,0,\dots, 0)^T$ for $i\in\{1,\dots,n\}$. For any index $p\in \bN^*$,
$
\Vert \x \Vert_p \coloneqq \left(\sum_{i=1}^n |x_i|^p \right)^{1/p}
$
denotes the $\ell^p(\bR^n)$ norm of $\x$. 
We let $\bR[\x]$ be the space of real polynomials over $\bR^n$. 
For any multi-index $\boldsymbol{\alpha}=(\alpha_1,\dots, \alpha_n)^T\in \bN^n$ with length  $\vert \balpha \vert = \sum_{i=1}^n \alpha_i$,  we define the associated monomial $
\x^{\balpha} = \prod_{i=1}^n x_i^{\alpha_i}
$ of degree $\vert \balpha \vert $. 
We let $\bN_{r}^n = \{\balpha \in \bN^n : \vert \balpha \vert \le r\},$ and 
 $\bR[\x]_r$ be the space of real polynomials of degree less than $r$ that can be written $\sum_{\balpha \in \bN_r^n} c_\balpha \x^\balpha$ for some real coefficients $c_\balpha$. 

For any Borel set $\cX$ in $\bR^n$, we denote $\cM(\cX)$ the space of finite signed Borel measures supported on $\cX$,
$$
\cM(\cX)_+ \coloneqq \{ \mu \in \cM(\cX) \cond \mu \geq 0 \}
$$
its positive cone of finite Borel measures supported on $\cX$, and
$$\cP(\cX) \coloneqq \{ \mu \in \cM(\cX)_+ \cond \mu(\cX)=1 \}$$
the set of probability measures supported on $\cX$. The indicator function of a subset $A\subset \bR^n$ is denoted as $\charFun_A(\x)$.
For any Borel set $\cX$   in $\bR^n$ and any measure $\mu \in \cM(\cX)$,  
$$
m_{\balpha}(\mu) \coloneqq \int_{\cX} \x^{\balpha} \rd\mu(\x)
$$
is the moment of $\mu$ associated to the multi-index $\balpha\in \bN^n$, and 
$$
m(\mu) \coloneqq \left(m_{\balpha}(\mu)\right)_{\balpha\in \bN^n}
$$
is the sequence of moments of $\mu$. The mass of $\mu$ is denoted $\mathrm{mass}(\mu) = m_0(\mu).$ A measure $\mu$ is said \emph{determinate} if it is uniquely determined by its moment sequence $m(\mu)$.

Finally, for a given sequence $y = (y_\balpha)_{\balpha \in \bN^n}$, we introduce the \emph{Riesz functional} $\ell_y : \bR[\x] \mapsto \bR$ which associates to a real polynomial $g(\x) = \sum_{\balpha \in \bN^n} a_\balpha \x^\balpha $ the value 
 $\ell_y(g) = \sum_{\balpha \in \bN^n} a_\balpha y_\balpha$. For any measure $\mu$, we thus have
 \begin{equation}
  \label{eq:riesz}
 \ell_{m(\mu)}(g) = \sum_{\balpha \in \bN^n} a_\balpha m_\balpha(\mu) = \int_{\bR^n} g(\x) d\mu(\x) ,\quad \forall g \in \bR[\x].
 \end{equation}

\section{Optimal transport problems with polynomial costs}\label{sec:OT}

\subsection{Formulation}
To guide the subsequent discussion, we consider the multi-marginal version of problem \eqref{eq:MK} as a prototypical example of an OT problem. This problem consists in considering $K$ probability measures $\mu_i\in \cP(\cX_i)$ defined on Borel sets $\cX_i\subset \bR^{n_i}$ for all $i\in \{1, \dots, K\}$, and solving  
\begin{equation}
\label{eq:MK-multimarg}
\rho \coloneqq \inf_{\pi \in \Pi} \cL(\pi).
\end{equation}
The loss function is of the form
\begin{equation}
\label{eq:loss-multimarg}
\cL(\pi) \coloneqq \int_\cX c(\x) \rd \pi(\x),\quad \forall \pi \in \cM(\cX)_+,
\end{equation}
and the set $\cX$ is  defined as the product set 
$$
\cX \coloneqq \cX_1\times\dots\times \cX_K.
$$
Note that $\cX$ can be identified with a subset of $\bR^n$, with
$$
n\coloneqq n_1 + \hdots + n_K.
$$
The function $c: \cX \to \bR$ is a given cost function, and the constraint $\Pi$ is a shorthand notation for the set of coupling measures having $\mu_i$ as marginals, namely
\begin{equation}
\label{eq:couplings}
\Pi(\cX;\mu_1,\dots, \mu_K) 
\coloneqq
\{ \pi \in \cP(\cX) \cond  \mathrm{proj}_i \# \pi  = \mu_i,\quad \forall i \in \{1,\dots, K\} \},
\end{equation}
where $\mathrm{proj}_i: \cX\to \cX_i$ denotes the canonical projection, and the push-forward measure $\mathrm{proj}_i \# \pi$ is the $i$-th marginal of $\pi$.

The existence of a minimizer for \eqref{eq:MK-multimarg} is standard in OT theory. Indeed, $\Pi $ is trivially non empty since the coupling $\otimes_{i=1}^K \mu_i$ belongs to this set. The set $\Pi$ is convex and compact for the $weak-*$ topology thanks to the imposed marginals, and if the cost function $c$ is lower semi-continuous (l.s.c.), then  the loss function $\cL:\pi \mapsto \int c \, \rd\pi$ is l.s.c. with respect to the weak-* topology.
Hence we can guarantee the existence of a minimizer by imposing a very weak hypothesis on the cost function $c$ such as  lower-semicontinuity.

We next show that when the loss function is of polynomial nature, problem \eqref{eq:MK-multimarg} is equivalent to a moment problem under some conditions, as initially observed in \cite{Lasserre2008} (see also \cite[Section 7.3]{lasserre2009moments}). To see this, note first of all that if $c$ is a polynomial from $\bR[\x]$, it follows from \eqref{eq:riesz} that the loss function \eqref{eq:loss-multimarg} satisfies
\begin{equation}
\label{eq:loss-pi-to-mom}
\cL(\pi) = \ell_{m(\pi)}(c) \coloneqq \rL(m(\pi)).
\end{equation}

In addition, the marginal constraints on 
$\pi \in \Pi$ in problem \eqref{eq:MK-multimarg} imply constraints on the moments of $\pi$, 
\begin{equation}
\label{eq:mom-marginals}
m_{(0,\dots, 0, \bbeta, 0,\dots,0)}(\pi) = m_{\bbeta}(\mu_i),\quad \forall \bbeta \in \bN^{n_i}, \text{ and } \forall i\in \{1,\dots, K\}.
\end{equation}

As a result of \eqref{eq:loss-pi-to-mom} and \eqref{eq:mom-marginals}, instead of considering the OT problem \eqref{eq:MK-multimarg} where we search for an unknown measure $\pi \in \Pi$, one can alternatively consider the moment problem of searching for the optimal sequence $y = (y_\balpha)_{\balpha \in \bN^n}$ solving
\begin{equation}
\label{eq:mom-pb}
\rho_{mom} \coloneqq \inf_{y \in \Pi_{mom}} L(y).
 \end{equation}
 where $\Pi_{mom} := \Pi_{mom}(\cX ; m(\mu_1), \hdots, m(\mu_K))$ is the set of sequences in $\bR^{\bN^{n}}$ satisfying the following constraints:
\begin{itemize}
\item[(i)] \textit{Marginal conditions}: the sequence $y$ should satisfy 
\begin{equation}
    \label{eq:seq-mom-marginals}
    y_{(0,\dots, 0, \bbeta, 0,\dots,0)} = m_{\bbeta}(\mu_i),\quad \forall \bbeta \in \bN^{n_i}, \text{ and } \forall i\in \{1,\dots, K\}
    \end{equation}
\item[(ii)] \textit{Moment sequence condition:} the sequence $y$ must have a representing measure supported on $\cX$, that is, there must exist a measure $\pi \in \cM(\cX)_+$ such that $y = m(\pi)$. We write this condition as $y\in \mathrm{MS(\cX)}.$
 \end{itemize}

The equivalence between problems \eqref{eq:MK-multimarg} and \eqref{eq:mom-pb} is closely related to the determinacy of measures $\mu_i$, for which a sufficient condition is that the sets $\cX_i$ are compact.
We summarize these facts in the theorem below.

\begin{theorem}[Polynomial cost]\label{th:equivalence-multimarg-poly}
If the sets $\{\cX_i\}_{i=1}^K$ are compact, then the OT problem  \eqref{eq:MK-multimarg} with polynomial cost is equivalent to the generalized moment problem \eqref{eq:mom-pb}:
\begin{itemize}
  \item A minimizer $\pi^*$ of problem \eqref{eq:MK-multimarg} is such that $m(\pi^*)$ is a minimizer of problem \eqref{eq:mom-pb}.
  \item A minimizer $y^*$ of problem \eqref{eq:mom-pb} has a representing measure which is solution of  \eqref{eq:MK-multimarg}.
\end{itemize}
In addition, if the solution $\pi^*$ of the OT problem \eqref{eq:MK-multimarg} is unique, then the solution $y^*$ of $ \eqref{eq:mom-pb}$ is unique and $y^* = m(\pi^*).$
\end{theorem}

\begin{proof}
  Suppose $\pi^*$ is a solution of problem \eqref{eq:MK-multimarg}. Since it is a Borel measure supported on the compact set $\cX$, it is determinate\footnote{This results from the density of the space of polynomials in the space $C(\cX)$ of continuous functions over a compact set $\cX$ in $\bR^{n}$ (Weierstrass approximation theorem), and the fact that the topological dual of $C(\cX)$ is equal to $\cM(\cX)$ (Riesz representation theorem).} so there exists a unique sequence $y  = m(\pi^*)$. This sequence $y$ is clearly in the set $ \Pi_{mom}$. Therefore, as a  feasible solution for \eqref{eq:mom-pb}, it satisfies  $ \rho_{mom} \le  L(y) =  \cL(\pi^*) = \rho  $. Conversely, let $y^*$ be a solution to problem \eqref{eq:mom-pb}. Since $y\in \textrm{MS}(\cX)$, there is a representing measure $\pi$ such that $y^* = m(\pi)$. For $\pi$ to belong to the feasible set $\Pi$ of problem  \eqref{eq:MK-multimarg}, the marginal  conditions on $m(\pi)$ should imply that the marginals of $\pi$ are the $\mu_i$. This is satisfied given that the marginal measures $\mu_i$ are determinate because $\cX_i$ is compact. Therefore $\pi \in \Pi$ and $ \rho \le  \cL(\pi)  = L(m(\pi))  =  L(y^*) =   \rho_{mom}$. This proves that $\rho =  \rho_{mom}$, and that $m(\pi^*)$ is a solution of problem  \eqref{eq:mom-pb} if and only if $\pi^*$ is a solution of \eqref{eq:MK-multimarg}.  
\end{proof}

\begin{remark}
Since the $\mu_i$ are probability measures, $y \in \Pi_{mom}$ is such that $y_{(0,\hdots,0)} = 1$ and a representing measure $\pi \in \cM(\cX)_+$ such that $y = m(\pi)$ has mass $1$, i.e. $\pi \in \cP(\cX).$
\end{remark}

\subsection{The moment sequence condition}\label{sec:moment-sequence-condition} 
In this section, we discuss how the moment sequence condition $y \in MS(\cX)$ is translated into mathematical terms. This question is in fact directly related to the so-called moment problem which studies the following question: Given a Borel subset $\cX\subseteq \bR^n$ and a sequence of real numbers $y=(y_{\balpha})_{\balpha\in \bN^n}$, what are the conditions on $y$ under which we can guarantee that $y = m(\pi)$ for some positive measure $\pi \in \cM(\cX)_+$? For the one dimensional case ($n=1$), this classical problem is well understood and dates back to contributions by Markov, Stieltjes, Hausdorff, and Hamburger. Explicit conditions on $y$ exist, and they are all stated in terms of positive semi-definiteness of certain Hankel matrices. Much less is known for the multidimensional case ($n>1$). A general result is given by the Riesz-Haviland theorem, which states that a moment sequence $y$ has an associated Borel measure  $\pi$ such that $y=m(\pi)$ if and only if $\ell_y(f)\ge 0$ for all polynomials $f \in \bR[\x] $ nonnegative on $\cX$. This theorem is not really useful if we do not have 
an explicit  characterization of polynomials that are nonnegative on $\cX$ (so called \emph{positivstellensatz}). Such a characterization has been provided by Schm\"ugden in \cite{schmudgen2017moment} when the ambient space $\cX$ is a compact basic semi-algebraic set of the form 
\begin{equation}
\label{eq:semialgebraic}
\cX = \{ \x\in \bR^n \cond g_j(\x)\geq 0,\; j=1,\dots, J \}
\end{equation}
for some polynomials $g_j \in \bR[\x]$. 
It is proven is  \cite{schmudgen2017moment} that a sequence $y$ has a representing Borel measure supported on $\cX$ (i.e. satisfies the moment sequence condition)
if and only if it satisfies
\begin{equation}
\ell_{y}(g_I f^2) \ge 0 , \quad  \forall I \subset \{0,\hdots,J\}, \quad \forall f\in \bR[\x], \label{eq:constraints-spd-schmudgen-Riesz}
\end{equation}
where $g_I = \prod_{j\in I} g_j$ and where we have used the convention $g_\emptyset=1$. 
 For a polynomial $g(\x) = \sum_{\bgamma \in \bN^n} c_\bgamma \x^\bgamma \in \bR[\x]$ and $r\in \bN$, we let $\pbM_r(g y)   \in \bR^{\bN^n_r\times \bN^n_r}$ be the matrix with entries 
 $$\pbM_r(g y)_{\balpha , \bbeta} = \ell_y(g(\x) \x^{\balpha} \x^{\bbeta} ) = \sum_{\bgamma \in \bN^n} c_\bgamma y_{\balpha + \bbeta + \bgamma}, \quad \balpha,\bbeta \in \bN^n_r,$$
 which is such that for any 
 polynomial $f  \in \bR[\x]_r$ of degree less than $r$, we have 
 $$ \ell_y(g f^2) = \sum_{\balpha,\bbeta \in \bN^n_r}\pbM_r( g y)_{\balpha,\bbeta} a_\balpha a_\bbeta, \quad \text{for } f(\x) = \sum_{\bgamma \in \bN^n_r} a_\bgamma \x^\bgamma.$$
Therefore, the moment sequence condition \eqref{eq:constraints-spd-schmudgen-Riesz}
is equivalent to 
\begin{align}
\pbM_r(g_I y) \succcurlyeq 0 , \quad  \forall I \subset \{0,\hdots,J\},  \quad \forall r\in \bN.\label{eq:constraints-spd-schmudgen}
\end{align}
A simpler characterization has been given by Putinar  in  \cite{Putinar1993} under 
 the following additional assumption.
\begin{assumption}\label{ass-sos}
There exists a polynomial $u$ of the form $u = u_0 + \sum_{j=1}^J u_j g_j$, where the $u_j$ are sums of squares (SoS) polynomials, and such that $\{\x \in \bR^n : u(\x) \ge 0\}$ is  compact.
\end{assumption}
Under Assumption \ref{ass-sos}, it is proven in \cite{Putinar1993} that $y$ has a representing Borel measure supported on $\cX$ if and only if
\begin{align}
\pbM_r(g_j y) \succcurlyeq 0  \quad \forall j\in \{0,\hdots,J\}, \quad \forall r\in \bN, \label{eq:constraints-spd}
\end{align}
where we have used the convention $g_0 = 1$. 

The linear positive semidefinite constraints \eqref{eq:constraints-spd-schmudgen} or \eqref{eq:constraints-spd} are exactly the moment sequence condition. We summarize the above results in the following theorem.  

\begin{theorem}[Th. 3.8 in \cite{lasserre2009moments}]\label{th:mom-seq-semidefinite}
Let $\cX$ be a basic semi-algebraic set  as in  \eqref{eq:semialgebraic}. A sequence $y \in \bR^{\bN^n}$ satisfies $y\in MS(\cX)$ (i.e. satisfies the  moment sequence condition on $\cX$) if and only if it satisfies the positive semidefinite constraints \eqref{eq:constraints-spd-schmudgen}, or the positive semidefinite constraints \eqref{eq:constraints-spd} under the additional Assumption \ref{ass-sos}.
\end{theorem}

In our context, since $\cX $ is the product set $\cX_1\times \dots\times \cX_K$, we assume that each $\cX_i$ is a compact basic semi-algebraic set defined as
\begin{equation}
\cX_i = \{ \x_i\in \bR^{n_i} \cond g^{(i)}_j(\x_i)\geq 0,\; j=1,\dots, J_i \},\quad \forall i \in \{1,\dots,K\}.
\end{equation}
Then $\cX$ is a basic semi-algebraic set defined as in \eqref{eq:semialgebraic}
with $J=\sum_{i=1}^K J_i$ and functions $\{g_j\}_{j=1}^J$ defined by
$$
g_j(\x) = g_{l}^{(i)}(\x_i)
\quad \text{for } j = \sum_{k=1}^{i-1} J_k + l, \quad 1\le l \le J_i, \quad 1\le i \le K.
$$


\begin{remark}[About Assumption \ref{ass-sos}]
Assumption \ref{ass-sos} is trivially satisfied if $g_1(\mathbf{x}) = R - \Vert \mathbf{x} \Vert_2^2 $ for some positive $R$. Since for any compact semi-algebraic set $\cX$, there exists a sufficiently large $R$ such that $\cX \subset \{\mathbf{x} : \Vert \mathbf{x}\Vert_2 < R\}$, the condition $R - \Vert \mathbf{x}\Vert_2^2 \ge 0$ is redundant and can be systematically added to the definition of $\cX$. 
A stronger condition for Assumption  \ref{ass-sos} to hold for a product set $\cX$ is that the description of each set $\cX_i$ contains a function $g^{(i)}_1(\mathbf{x}_i) = R_i - \Vert \mathbf{x}_i \Vert_2^2$. If this is not the case, we should prefer to add a single function $g_1(\mathbf{x}) = R - \Vert \mathbf{x}\Vert_2^2 $ to the description of $\cX$ in order to reduce the number of positive semidefinite constraints. 
\end{remark}


\subsection{Piecewise polynomial costs}\label{sec:piecewise-polynomial-cost}
Some OT problems (such as, e.g., the $L^1$-Wasserstein distance), involve a continuous or l.s.c. piecewise polynomial cost. These problems can also be formulated as generalized moment problems, up to the introduction of new unknown measures. 

\paragraph{Piecewise polynomial costs.}
Let us assume that 
$\cX = \cA_1 \cup  \hdots \cup \cA_m$, where the $\cA_i$ are pairwise disjoint Borel sets and $$c_{\mid \cA_i} = c_i \in \bR[\x], \quad 1\le i \le m.$$ 
For a measure $\pi \in \cP(\cX)$, we introduce the measures
$$
\pi_i \coloneqq  \charFun_{\cA_i} \pi \in \cM(\bar \cA_i)_+,
$$
where $\bar \cA_i$ is the closure of $\cA_i$.
Since $\charFun_{\cA_1} + \hdots + \charFun_{\cA_m} = \charFun_\cX$, we have  $\pi = \pi_1 + \hdots + \pi_m,$
and 
\begin{align*}
 \cL(\pi) = \int_\cX  c(\x) d\pi(\x)  = \sum_{i=1}^m \int_\cX c_i(\x) d\pi_i(\x)  =\sum_{i=1}^m \ell_{m(\pi_i)}(c_i) := \tilde \cL(\pi_1 , \hdots , \pi_m).
\end{align*}
We claim that the OT problem \eqref{eq:MK-multimarg} is equivalent to  
\begin{align}
\inf_{ \substack{\pi_1 \in \cM(\bar \cA_1)_+, \hdots , \pi_m \in \cM(\bar \cA_m)_+ , \\ \pi_1 + \hdots + \pi_m  \in \Pi}}  \tilde \cL(\pi_1, \hdots , \pi_m)  := \tilde \rho,\label{OT-problem-pwpoly}
\end{align}
Indeed, if $\pi$ is solution of problem \eqref{eq:MK-multimarg}, then the measures $\pi_i = \charFun_{\cA_i} \pi$, $1\le i\le m$, 
 satisfy the constraints of \eqref{OT-problem-pwpoly} and $\tilde \rho \le \tilde \cL(\pi_1,\hdots,\pi_m) = \cL(\pi) = \rho$.
Conversely,   
the set of measures $(\pi_1,\hdots,\pi_m)$ satisfying the constraints of problem \eqref{OT-problem-pwpoly} is compact in the weak-$^*$ topology (since the $\pi_i \in \cM(\bar \cA_i)_+$ and $\mathrm{mass(\pi_i)}\le 1$), and therefore implies the existence of solutions. Moreover,  if    
 $(\pi_1,\hdots ,\pi_m)$ is a solution of \eqref{OT-problem-pwpoly}, then $\pi = \pi_1 + \hdots +  \pi_m   \in \Pi$ and $\tilde \rho =  \tilde \cL(\pi_1, \hdots, \pi_m) = \sum_{i=1}^m \int_\cX c(\x) (\rd\pi_1 + \hdots + \rd \pi_m) = \cL(\pi ) \ge \rho$. 
Finally, denoting $\tilde \cL(\pi_1 , \hdots , \pi_m) = \tilde L(m(\pi_1), \hdots,m(\pi_m))$, 
we have that the initial OT problem \eqref{eq:MK-multimarg}  is equivalent to the optimization problem
\begin{align}
\inf_{ \substack{ y_1 \in MS(\bar \cA_1)_+, \hdots , y_m \in MS(\bar \cA_m)_+ \\ y_1 + \hdots + y_m  \in \Pi_{mom}}}  \tilde L(y_1, \hdots , y_m) \label{OT-problem-pwpoly-mom}
\end{align}
over $m$ sequences $(y_i)_{1\le i\le m}$ that satisfy moment sequence conditions and whose sum satisfies marginal conditions.
We summarize these facts in the theorem below.

\begin{theorem}[Piecewise polynomial cost]\label{th:equivalence-multimarg-pwpoly}
If $ \cX_1\times\dots\times\cX_K$ is compact, then the OT problem  \eqref{eq:MK-multimarg} with l.s.c. piecewise polynomial cost over a partition $(\cA_i)_{1\le i \le m}$ of $\cX$ is equivalent to the generalized moment problem \eqref{OT-problem-pwpoly-mom}: 
 a minimizer $\pi^*$ of problem \eqref{eq:MK-multimarg} is such that $(m(\pi_i^*))_{1\le i\le m}$, with $\pi^*_i = \charFun_{\cA_i} \pi^*$, is a minimizer of problem \eqref{OT-problem-pwpoly-mom}, and conversely,  a minimizer $(y^*_i)_{1\le i \le m}$ of problem 
 \eqref{OT-problem-pwpoly-mom} is such that each $y^*_i$ has a representing measure $\pi_i$ supported on $\bar \cA_i$, and the sum $\pi = \pi_1 + \hdots + \pi_m$ is solution of  \eqref{eq:MK-multimarg}. In addition, if the solution $\pi^*$ of the OT problem \eqref{eq:MK-multimarg} is unique, then \eqref{OT-problem-pwpoly-mom} has infinitely many solutions $(y^*_1,\hdots,y_m^*)$ but the  sum $y^*_1 + \hdots + y_m^* := y^*$ is unique and such that 
 $y^* = m(\pi^*).$
\end{theorem}

\begin{remark}
For having a practical characterizing of the set $MS(\bar \cA_i)$ of sequences that satisfy the moment sequence condition on $\bar \cA_i$, the partition should be such that the $\bar \cA_i$ are basic semi-algebraic compact sets. If $\cX$ is a basic semi-algebraic compact set, that means that $\cA_i$ should be defined as the set of points in $\cX$ satisfying a finite set of additional polynomial inequalities. 
\end{remark}

\paragraph{Sum of piecewise polynomial costs.}
In the case where 
\begin{align}
c(\x) = \sum_{k=1}^s c_k(\x), \label{eq:cost-sum-pwpoly}
\end{align}
where each $c_k$ is a l.s.c. piecewise polynomial associated with a particular partition $(\cA_{k,i})_{1\le i \le m_k }$, i.e. $c_{k \mid \cA_{k,i}} := c_{k,i} \in \bR[\x]$, we could introduce a finer partition of $\cX$ composed by sets 
$\cA_{1,i_1}\cap \hdots \cap \cA_{s,i_s}= \cA_{\mathbf{i}} $, with $1 \le i_k \le m_k$. The function $c$ being polynomial on each set $ \cA_{\mathbf{i}}$, the problem can be reformulated as a generalized moment problem involving  $m_1 \hdots m_s$ measures $\pi_{\mathbf{i}}$ supported on the sets $\bar \cA_{\mathbf{i}}$, for $\mathbf{i} \in \{1,\hdots,m_1\} \times \hdots \times \{1,\hdots,m_s\}$. However, the resulting number of unknown measures is exponential in $s$.

An alternative approach, that will be used later in this paper, is to introduce for each $1 \le k \le s$ a collection of measures $(\pi_{k,i})_{1\le i \le m_i}$ and consider the  problem 
\begin{align}\label{eq:OT-sum-pwpoly}
\inf_{\pi\in \Pi,(\pi_{k,i})}\sum_{k=1}^s \sum_{i=1}^{m_k} \int_{\cX} c_{k,i}(\x) d\pi_{k,i}
\end{align}  
over  measures $\pi \in \Pi$ and $\pi_{k,i} \in \cM(\bar \cA_{k,i})_+$, $1\le i \le m_k, 1\le k\le s$,  satisfying  $$\sum_{i=1}^{m_k} \pi_{k,i} = \pi, \quad \text{for all } 1\le k\le m.$$ 
This results in a problem with $m_1 + \hdots + m_s +1$ unknown measures, that can be equivalently written as the  problem 
\begin{align}\label{eq:OT-sum-pwpoly-mom}
\inf_{y\in \Pi_{mom},(y_{k,i})}\sum_{k=1}^s \sum_{i=1}^{m_k}  \ell_{y_{k,i}}(c_{k,i})
\end{align}  
with measures $y\in \Pi_{mom}$ and $y_{k,i} \in MS(\cA_{k,i})$, $1\le i \le m_k, 1\le k\le s$, satisfying the additional constraints 
$$
\sum_{i=1}^{m_k} y_{k,i} = y, \quad \text{for all } 1\le k\le s.
$$
Note that  the measure $\pi$ (resp. the sequence $y$) can be eliminated from problem \eqref{eq:OT-sum-pwpoly}
 (resp. \eqref{eq:OT-sum-pwpoly-mom}).
We summarize the above results in the next theorem. 

\begin{theorem}[Sum of piecewise polynomial costs]\label{th:equivalence-multimarg-sum-pwpoly}
Assume $ \cX_1\times\dots\times\cX_K$ is compact, and consider a l.s.c. piecewise polynomial cost of the form 
\eqref{eq:cost-sum-pwpoly}, where each $c_k$ is a l.s.c. piecewise polynomial over a  partition $(\cA_{k,i})_{1\le i \le m_k}$ of $\cX$, $1\le k\le s$. 
Then the OT problem  \eqref{eq:MK-multimarg}  is equivalent to the problem \eqref{eq:OT-sum-pwpoly-mom}: 
 a minimizer $\pi^*$ of problem \eqref{eq:MK-multimarg} is such that $(m(\pi_{k,i}^*))$, with $\pi^*_{k,i} = \charFun_{\cA_{k,i}} \pi^*$, is a minimizer of problem \eqref{eq:OT-sum-pwpoly-mom}, and conversely,  a minimizer $(y^*_{k,i})$ of problem 
  \eqref{eq:OT-sum-pwpoly-mom} is such that each $y^*_{k,i}$ has a representing measure $\pi_{k,i}$ supported on $\bar \cA_{k,i}$, and for each $k$, the sum $ \pi_{k,1} + \hdots + \pi_{k,m_k} = \pi$ is solution of  \eqref{eq:MK-multimarg}. In addition, if the solution $\pi^*$ of the OT problem \eqref{eq:MK-multimarg} is unique, then  \eqref{eq:OT-sum-pwpoly-mom} have infinitely many solutions but the  sum $y^*_{k,1} + \hdots + y_{k,m_k}^* := y^*$ is unique and such that 
 $y^* = m(\pi^*).$
\end{theorem}

\section{Wasserstein distances and barycenters}\label{sec:wasserstein}

In this section, we consider the problems of computing distances and barycenters in Wasserstein spaces and show that they can be expressed as generalized moment problems. Throughout the section, 
$\cX$ denotes a compact basic semi-algebraic set in the normed vector space $(\bR^d,\Vert \cdot \Vert_p)$, with $p \in \bN^*$.

\subsection{Wasserstein distances}
The Wasserstein space $\cP_p(\cX)$ is defined as the set of probability measures $\mu\in \cP(\cX)$ with finite moments up to order $p$, namely
$$
\cP_p(\cX) \coloneqq \biggl\{ \mu \in \cP(\cX) \cond \int_\cX \Vert \x \Vert^p_p \,\rd\mu(\x) \;< +\infty \biggr\}.
$$
Let $\mu$ and $\nu$ be two probability measures in $\cP_p(\cX)$. For any $p\in \bN^*$, the $L^p$-Wasserstein distance $W_p(\mu,\nu)$ between $\mu$ and $\nu$ is defined by
\begin{equation}
\label{eq:Wp}
W_p^p(\mu,\nu) \coloneqq \mathop{\inf}_{\pi \in \Pi(\cX\times\cX;\mu,\nu)} \int_{\cX \times \cX} \Vert \x-\y\Vert^p_p \,\rd\pi(\x, \y). 
\end{equation}
The space $\cP_p(\cX)$ endowed with the distance $W_p$ is a metric space, usually called $L^p$-Wasserstein space (see \cite{Villani2003} for more details). The $W_p$ distance defined through problem \eqref{eq:Wp} is an optimal transport problem of the form \eqref{eq:MK} with $K=2$ marginals, $\cX_1 = \cX_2 = \cX$ and a continuous  cost function
$$
c(\x,\y)= || \x-\y||^p_p = \sum_{i=1}^d |x_i-y_i|^p
$$
We claim that for any $p\in \bN^*$, this problem can be seen as a generalized moment problem. We distinguish the cases where $p$ is even and odd.  

\paragraph{Case p even.}
When $p$ is an even number, the cost $c$ is a polynomial and we simply use the binomial theorem to derive that the loss function in \eqref{eq:Wp}
can be expressed as 
\begin{align*}
\cL^{W_p}(\pi) := \int_{\cX \times \cX} \Vert \x-\y\Vert^p_p \,\rd\pi(\x, \y)  
&= \sum_{i=1}^d \sum_{k=0}^p {p\choose k} \int_{\cX\times\cX} (-1)^k x_i^k y_i^{p-k} \rd\pi(\x,\y)
\end{align*}
or in terms of the moments $m(\pi)$ of $\pi$, 
\begin{align*}
\cL^{W_p}(\pi) = \sum_{i=1}^d \sum_{k=0}^p {p\choose k} (-1)^k m_{k\e_i, (p-k)\e_i}(\pi) 
& := \sum_{i=1}^d L_i^{W_p}(m(\pi))  := L^{W_p}(m(\pi)),
\end{align*}
where we recall that $\e_i$ is the $i$-th canonical vector in $\bN^d$.
The marginal constraints $\pi \in  \Pi(\cX\times \cX; \mu, \nu)$ of problem \eqref{eq:Wp} can also be expressed in terms of moments. 
We derived their general form in equation \eqref{eq:mom-marginals}. In the present context, they read
\begin{equation*}
m_{(\balpha, 0)}(\pi) = m_{\balpha}(\mu),\quad  
m_{(0, \bbeta)}(\pi) = m_{\bbeta}(\nu), \quad \forall \balpha,\bbeta \in \bN^{d}.
\end{equation*}
The problem \eqref{eq:Wp} can then be expressed as the generalized moment problem 
\begin{align}
\inf_{y\in \Pi_{mom}} L^{W_p}(y), \label{eq:Wp-moment}
\end{align}
where $\Pi_{mom} := \Pi_{mom}(\cX \times \cX;m(\mu), m(\nu))$ is the set of sequences $y \in \bR^{\bN^{2d}}$ that satisfy the moment sequence condition and 
the marginal constraints  
$$
y_{(\balpha, 0)} = m_{\balpha}(\mu)\quad \text{and} \quad y_{(0, \bbeta)} = m_{\bbeta}(\nu), \quad \forall \balpha,\bbeta \in \bN^{d}.
$$
Here, Theorem \ref{th:equivalence-multimarg-poly} applies and proves the equivalence between problems \eqref{eq:Wp-moment} and \eqref{eq:Wp}.

\paragraph{Case $p$ odd.}
 When $p$ is odd, the presence of the absolute value in the cost function $c$ prevents from having a polynomial expression. We can nevertheless derive a moment formulation by exploiting the fact that the cost is piecewise polynomial on $\cX \times \cX$.
We first  introduce for all $i\in\{1,\dots,d\}$ the subsets 
\begin{equation*}
\cA_i^+ \coloneqq \{ (\x,\y)\in \cX\times\cX \cond x_i-y_i \geq 0\}, \quad \cA_i^- \coloneqq \{ (\x,\y)\in \cX\times\cX \cond x_i-y_i < 0\}, 
\end{equation*}
that form a partition of  $\cX\times \cX$, i.e.
$$
  \cA_i^+ \cup \cA_i^-= \cX\times \cX, \quad \cA_i^+ \cap \cA_i^- = \emptyset.
$$
If $\cX$ is compact semi-algebraic, then $\cA_i^+$ and $   \overline{ \cA_i ^-}$ are also compact semi-algebraic. 
For any $\pi \in \cP(\cX\times\cX)$, we can define measures $\pi^+_i  , \pi^-_i$  by
\begin{align}\label{eq:def-pipm}
 \pi_i^+  &= \charFun_{\cA_i^+}    \pi, \quad 
 \pi_i^-  = \charFun_{\cA_i^-}  \pi, 
 \end{align}
which are such that
\begin{align}
\pi = \pi_i^+  + \pi_i^-,\quad \forall i \in \{1,\dots, d\}. \label{eq:constraint-Wp-sum-measures}
\end{align}
When $p$ is odd, since $\charFun_{\cA_i^- } + \charFun_{\cA_i^+}  = \charFun_\cX$, 
we can write the Wasserstein loss function as
\begin{align*}
\cL^{W_p}(\pi) 
&= \sum_{i=1}^d \int_{\cX\times\cX}  \vert x_i-y_i \vert^p (\rd\pi_i^+(\x,\y) + \rd\pi_i^-(\x,\y)) 
\\&= \sum_{i=1}^d \int_{\cX\times\cX}  (x_i-y_i)^p (\rd\pi_i^+(\x,\y)-\rd\pi_i^-(\x,\y)) 
\\&= \sum_{i=1}^d \sum_{k=0}^p {p\choose k} (-1)^k m_{k\e_i, (p-k)\e_i}(\pi_i^+ - \pi_i^-)
\\
&= \sum_{i=1}^d L^{W_p}_i(m(\pi_i^+) - m(\pi_i^-)).
\end{align*}

From Theorem \eqref{th:equivalence-multimarg-sum-pwpoly}, we know that 
problem \eqref{eq:Wp} is equivalent to the following problem with $2d+1$ measures,
\begin{align}\label{eq:Wp-odd}
W_p^p(\mu,\nu) =
\inf_{\substack{\pi \in \Pi(\mu,\nu), \\ \pi_i^+ \in \cM(\cA_i^+)  ,\pi_i^- \in \times \cM( \overline{\cA_i^-}) }} \sum_{i=1}^d  L_i^{W_p}(m(\pi^+_i) - m(\pi_i^-)) 
\end{align}
which can be equivalently reformulated as a generalized moment problem
\begin{align}
W_p^p(\mu,\nu) = \inf_{y , y_1^+ , \hdots, y_d^- } \sum_{i=1}^d L^{W_p}_i \left( y_i^+ -   y_i^- \right)
\end{align}
over a set of $2d+1$ sequences satisfying moment sequence conditions $y_ i^+ \in MS(\cA_i^+)$ and $y_ i^- \in MS(\overline{\cA_i^-})$,  $1\le i \le d$, and the constraints $y\in \Pi_{mom}(\cX\times \cX ; m(\mu),m(\nu))$ and 
$$
 y= y_i^+ +  y_i^-, \quad \forall i \in \{1,\hdots,d\}.
$$
Note that the variable $y$ can be eliminated. 

 \subsection{Wasserstein barycenters}\label{sec:wasserstein-barycenters}
 A notion that is widely used to approximate measures in the Wasserstein spaces is the one of barycenters. To define it, let $N\in \bN^*$ and let
$$
\Sigma_N \coloneqq \Big\{ \lambda \in \bR^N\cond \lambda_i\geq 0,\, \sum_{i=1}^N \lambda_i = 1 \Big\}
$$
be the simplex in $\bR^N$. We say that $\bary(\rY_N, \Lambda_N) \in \cP_p(\cX)$ is a barycenter associated to a given set $\rY_N = (\mu_i)_{1\leq i\leq N}$ of $N$ probability measures from $\cP_p(\cX)$ and to a given set of weights $\Lambda_N = (\lambda_i)_{1\leq i\leq n} \in \Sigma_N$, if and only if $\bary(\Lambda_N, \rY_N)$ is a solution to
\begin{equation}
\label{eq:barywass}
\inf_{\nu \in \cP_p(\cX)} \sum_{i=1}^N \lambda_i W_p^p(\nu,\mu_i).
\end{equation}
Existence and uniqueness of minimizers of \eqref{eq:barywass} has been studied in depth in \cite{AC2011} for the case $p=2$. It is shown, in particular, that if one of the $\mu_i$ has a density, the barycenter is unique. In the following we assume existence of minimizers.
Problem \eqref{eq:barywass} can be written as an optimization problem  
$$
\inf_{\nu , \pi_1,\hdots , \pi_N} \sum_{i=1}^N \lambda_i \cL^{W_p}(\pi_i)
$$
over measures $\nu \in \cP_p(\cX)$, and $\pi_i \in \cM(\cX \times \cX)_+$, $1\le i \le N,$ satisfying the constraints $\pi_i\in \Pi(\cX\times\cX; \nu,\mu_i)$.

When $p$ is even, this can be equivalently written as a generalized moment problem
$$
\inf_{y , y_1,\hdots , y_N} \sum_{i=1}^N \lambda_i L^{W_p}(y_i)
$$
over sequences that satisfy the constraints $y_i \in \Pi_{mom}(\cX \times \cX ; y , m(\mu_i))$, $1\le i \le N,$
Note that the unknown $y$ can be eliminated by imposing that all $y_i$ have the same left marginal sequence.

When $p$ is odd, the problem \eqref{eq:barywass} is equivalent to a generalized moment problem
$$
\inf_{y , (y_{i,j}^+ , y_{i,j}^-)}\sum_{i=1}^N \sum_{j=1}^d \lambda_i L^{W_p}_j(y_{i,j}^+ - y_{i,j}^-) 
$$
over sequences satisfying moment sequence conditions $y \in MS(\cX)$ and  $y_{i,j}^\pm \in MS(\cX\times \cX)$, $1\le i \le N, 1\le j \le d$, and the  additional constraints  
$y_{i,j}^+ + y_{i,j}^- = y_{i,k}^+ + y_{i,k}^-$ for all $i \in \{1,\hdots,N\}$ and 
$ 1\le j<k\le d$, and $y_{i,j}^+ + y_{i,j}^- \in \Pi_{mom}(\cX\times \cX ; y , m(\mu_i))$ for all $i \in \{1,\hdots,N\}$ and 
$ 1\le j\le d$.
Again, 
 the unknown $y$ could be eliminated by imposing that all $y_i$ have the same left marginal sequences.

\section{Gromov-Wasserstein discrepancies and barycenters}\label{sec:gromov}
For some applications such as shape matching or word embedding, an important limitation of classic Wasserstein metrics lies in the fact that it is not invariant to rotations and translations and more generally to isometries. It is also defined for measures defined on the same ambient space $\cX$. To overcome these limitations, several extensions have been proposed (see, e.g., \cite{AJJ2019}). We focus here on the so-called Gromov-Wasserstein discrepancies in Euclidian spaces, originally introduced in \cite{Memoli2011}, and which has recently attracted a lot of attention from practitioners.

\subsection{Gromov-Wasserstein discrepancies}
Given two compact semi-algebraic Borel sets $\cX \in \bR^{d_\cX}$  and $\cY \in \bR^{d_\cY}$, two probability measures $\mu\in \cP(\cX)$ and $\nu \in \cP(\cY)$, and  two cost functions 
$c_\cX : \cX \times \cX \to \bR$ and $c_\cY : \cY \times \cY \to \bR$, we define for $p\in \bN^*$ a Gromov-Wasserstein discrepancy $GW_{p}$ between 
measures $\mu$ and $\nu$ as 
\begin{align}
GW_p^p(c_\cX,c_{\cY} , \mu,\nu) = \inf_{\pi \in \Pi(\cX\times\cY; \mu,\nu)}  \cL^{GW_p}(\pi) \label{eq:GWp}
\end{align}
where the loss function  $\cL^{GW_p} : \cM(\cX\times \cY)_+ \to \bR$ is  such that 
$$
\cL^{GW_p}(\pi) = \int_{\cX\times \cY} \int_{\cX\times \cY} \vert c_\cX( \x , \x' )   -c_\cY( \y , \y' ) \vert^p d\pi(\x,\y) d\pi(\x',\y').
$$
Note that this problem is quadratic in $\pi$. It can alternatively be expressed as a linear problem with a rank-one tensor constraint in the augmented space
$$
\cZ := \cX \times \cY \times \cX \times \cY
$$
which can be identified with a basic semi-algebraic set of $\bR^{2n}$ with 
 $n= d_\cX + d_\cY$. Using the space $\cZ$, we can write
\begin{align}
GW_p^p (c_\cX,c_{\cY} , \mu,\nu) =
\inf_{\substack{\gamma= \pi \otimes \pi \in \cM_+(\cZ) \\ \pi \in \Pi(\cX\times \cY; \mu, \nu)}} 
\cL_{\text{aug}}^{GW_p}(\gamma)
\label{eq:GWp_aug}
\end{align}
with
$$
\cL_{\text{aug}}^{GW_p}(\gamma) \coloneqq \int_\cZ \vert c_\cX( \x , \x' )   -c_\cY( \y , \y' ) \vert^p d\gamma(\x, \y, \x',\y'),
$$
and we have
$$
\cL_{\text{aug}}^{GW_p}(\pi\otimes\pi) = \cL^{GW_p}(\pi), \quad \forall \pi \in \cM_+(\cX\times \cY).
$$

In the particular case where the cost functions $c_\cX = \Vert \cdot - \cdot \Vert_q^q$  and $c_\cY = \Vert \cdot - \cdot \Vert_q^q$ 
are related to $\ell^q$ norms, for some $q \in \bN ^*$, 
we denote $GW_{p,q}(\mu,\nu)$ the corresponding Gromov-Wasserstein discrepancy 
and $\cL^{GW_{p,q}}$ the corresponding loss. Note that case $ GW_{2,2}$ is of particular practical interest.  
We now distinguish different cases  depending on whether the costs $c_\cX$ and $c_{\cY}$ are polynomials or not.

\subsubsection{Polynomial costs $c_\cX$ and $c_\cY$} \label{sec:gromov-poly}
\label{sec:GW-poly}
Here we consider polynomial costs $c_\cX$ and $c_\cY$. Two cases will be again distinguished.  

\paragraph{Case $p$ even.}
When $p$ is even, then the cost 
 $ \vert c_\cX( \x , \x' )   -c_\cY( \y , \y' ) \vert^p $ 
 is a polynomial on $\cX$. Given polynomial expansions of $c_\cX$ and $c_\cY$, we can deduce a polynomial expansion of their difference
 $$
 g(\z)
 = c_\cX( \x , \x' )   -c_\cY( \y , \y' )
 = \sum_{i=1}^N c_i  \z^{\bgamma_i}, \quad \forall \z =(\x,\y,\x',\y') \in \cZ,
 $$
 with $\bgamma_i \in \bN^{2n}$ and $c_i\in \bR$. Using the multinomial theorem,
 $$
 |g(\z)|^p = ( c_\cX( \x , \x' )   -c_\cY( \y , \y' ) )^p 
 =  \sum_{\mathbf{k} = (k_1,\hdots,k_N) \in \bN^N , \vert \mathbf{k}\vert = p  } {p \choose {\mathbf{k}}} \prod_{i=1}^N c_i^{k_i} \z^{k_i\bgamma_i} := \sum_{\mathbf{k}  \in \bN^N , \vert \mathbf{k} \vert = p  } a_{\mathbf{k}} \z^{\bgamma_{\mathbf{k}}},
 $$ 
 with $\bgamma_{\mathbf{k}}= \sum_{i=1}^N k_i \bgamma_i \in \bN^{2n}$ and $a_{\mathbf{k}}= {p \choose {\mathbf{k}}} \prod_{i=1}^N c_i^{k_i} $. 
For $\bgamma \in \bR^{2n}$, we denote $\bgamma^L,\bgamma^R \in \bR^{n}$ such that $\bgamma = (\bgamma^L, \bgamma^R)$. This  yields the following expression of the Gromov-Wasserstein loss function in terms of moments
  \begin{align} 
  \cL^{GW_p}_{{aug}}(\pi \otimes \pi) &= %
 \sum_{\mathbf{k}  \in \bN^N , \vert \mathbf{k} \vert = p  } a_{\mathbf{k}}  m_{\bgamma_{\mathbf{k}}}(\pi\otimes \pi):=  L^{GW_p}_{aug}(m(\pi \otimes \pi)),\label{eq:loss_GWp_even_linear}
 \end{align}
with $L^{GW_p}_{aug}:\bR^{\bN^{2n}} \to \bR$ a linear functional, or
  \begin{align} 
  \cL^{GW_p}(\pi) &= %
  \sum_{\mathbf{k}  \in \bN^N , \vert \mathbf{k} \vert = p  } a_{\mathbf{k}}  m_{\bgamma^L_{\mathbf{k}}}(\pi)m_{\bgamma^R_{\mathbf{k}}}(\pi) 
  := L^{GW_p}(m(\pi)),\label{eq:loss_GWp_even_quadratic}
 \end{align} 
with $L^{GW_p}:\bR^{\bN^{n}} \to \bR$ a quadratic functional.
 

When $p$ is even and the costs are polynomials, the Gromov-Wasserstein problem \eqref{eq:GWp} can therefore be expressed  as a generalized moment problem with quadratic objective function
$$
GW_p^p(c_\cX,c_\cY ; \mu,\nu) = \inf_{y \in \Pi_{mom}} L^{GW_p}(y) 
$$
 with $\Pi_{mom}=\Pi_{mom}( \cX\times \cY; m(\mu),m(\nu))$ the set of sequences satisfying the moment sequence condition  {$y\in MS(\cX\times \cY)$} and marginal conditions, 
and we easily prove the equivalence between the two problems, following the proof of Theorem \ref{th:equivalence-multimarg-poly}.

\paragraph{Case p odd.}  When $p$ is odd, we can use a similar strategy as for Wasserstein distances. We introduce two subsets $\cA^+$ and $\cA^-$ of 
$\cZ$ defined by 
$$
\cA^+ = \{\z \in\cZ : g(\z) \ge 0 \}, \quad  \cA^- = \{\z   \in \cZ : g(\z) < 0 \},
$$
with $g(\z) = c_\cX( \x , \x' )   -c_\cY( \y , \y' )$ for $\z = (\x,\y,\x',\y').$ The sets are such that $(\cA^+ , \cA^-)$ form a partition of $\cZ$.  
If $\cX$ and $\cY$ are basic semi-algebraic sets, then the sets $\cA^+$, $\overline{\cA^-}$ are also basic semi-algebraic sets.  
For any $\pi \in \cP(\cX\times\cY)$, we define two measures $\gamma^+ = \charFun_{\cA^+}\pi \otimes \pi $ and $ \gamma^- = \charFun_{\cA^- }   \pi \otimes \pi,$
 which are such that 
\begin{align}
\pi \otimes \pi = \gamma^+ + \gamma^-. \label{eq:constraint-GWp-sum-measures}
\end{align}
Since $\charFun_{\cA^-} + \charFun_{\cA^+} = \charFun_\cZ$, 
we can write the Gromov-Wasserstein loss function as
\begin{align}
\cL^{GW_p}(\pi \otimes \pi) &= \cL^{GW_p}(\gamma^+ + \gamma^-) = \int_{\cX} g(\z)^p (d\gamma^+(\z) - d\gamma^-(\z)).
\end{align} 
Therefore, from \eqref{th:equivalence-multimarg-pwpoly}, we know that the problem \eqref{eq:GWp} is equivalent to the following problem 
\begin{align}\label{eq:GWp-odd}
\inf_{\pi \in \Pi(\mu,\nu), \gamma^+  \in \cM(\cA^+) ,  \gamma^-  \in \cM(\overline{\cA^-}) } \cL^{GW_p}_{aug}(\gamma^+ + \gamma^-)
\end{align}
over three measures satisfying the constraint \eqref{eq:constraint-GWp-sum-measures}, or equivalently 
\begin{align}\label{eq:GWp-odd-mom}
GW_p^p(c_\cX,c_\cY ; \mu,\nu) = \inf_{y \in \Pi_{mom}, y^+  \in MS(\cA^+ ),  y^-   \in MS(\overline{\cA^-} ) } L^{GW_p}_{aug}(y^+) - L^{GW_p}_{aug}(y^-),
\end{align}
with $L^{GW_p}_{aug}$  defined by \eqref{eq:loss_GWp_even_linear}, and where $y \in \Pi_{mom}(m(\mu),m(\nu))$ satisfies marginal conditions and the moment sequence condition on $\cX\times \cY$, the sequences $y^+ \in MS(\cA^+ )$ and $ y^- \in MS(\overline{\cA^-} ) $ satisfy the moment sequence condition on $\cA^+$ and $\overline{\cA^-}$ respectively, and the three sequences  satisfy the 
additional quadratic constraint 
$y^+ + y^- = y \otimes y
$, 
or equivalently 
$$
y^+_{\balpha,\bbeta} + y^-_{\balpha,\bbeta}  = y_{\balpha} y_{\bbeta}, \quad \forall \balpha,\bbeta \in \bN^{n}.
$$

\subsubsection{Piecewise polynomial costs $c_\cX$ and $c_\cY$}
\label{sec:gromov-pwpoly}
The case where $c_\cX$ and $c_\cY$ are piecewise polynomial functions can be treated by following the general strategy presented in Section \ref{sec:piecewise-polynomial-cost}. 
Let us briefly discuss the case of $GW_{p,q}$ with $q$ odd, where the cost is 
$$
\left\vert  \Vert \x - \x' \Vert_q^q -  \Vert \y - \y' \Vert_q^q \right \vert^p = \left \vert \sum_{i=1}^d \vert x_i - x_i' \vert^q -  \vert y_i - y_i' \vert^q \right\vert^p := \vert g(\z) \vert ^p.
$$
For $p$ even and $q$ odd, a first strategy is to introduce a partition $\{\cA_{\boldsymbol{\alpha} } : \boldsymbol{\alpha}  \in \{-1,1\}^{2d}\}$ with $2^{2d}$ elements, where 
\begin{align*}
\cA_{\boldsymbol{\alpha}} = \{ &\z = (\x,\y,\x',\y') \in \cZ :  \text{ for all $1\le i\le d$,} \\
  &x_i - x_i' \ge 0 \; \text{if} \; \alpha_i = 1 \; \text{or} \;  x_i - x_i' < 0 \; \text{if} \; \alpha_i = -1, \\
&y_i - y_i' \ge 0 \; \text{if} \; \alpha_{i+d} = 1 \; \text{or} \;   y_i - y_i' < 0 \; \text{if} \; \alpha_{i+d} = -1 \}.
\end{align*}
On each element $\cA_{\boldsymbol{\alpha}}$, the cost $  g(\z)  ^p$ is a polynomial. Therefore, the problem on a single measure $\pi$ can be reformulated as a problem on $4^d$ measures $\pi_{\boldsymbol{\alpha}} = \charFun_{\cA_{\boldsymbol{\alpha}}} \pi$. 
For $p$ odd and $q$ odd, we can introduce a partition $\{\cA_{\boldsymbol{\alpha}}^\pm : \boldsymbol{\alpha}  \in \{-1,1\}^{2d}\}$ with $2^{2d+1}$ elements, where 
$
\cA_{\boldsymbol{\alpha}}^+ =  \cA_{\boldsymbol{\alpha}} \cap \cB_{\boldsymbol{\alpha}}^{+}
$ and $
\cA_{\boldsymbol{\alpha}}^- =  \cA_{\boldsymbol{\alpha}} \cap \cB_{\boldsymbol{\alpha}}^{-}
$, with 
\begin{align*}
&\cB_{\boldsymbol{\alpha}}^{+} = 
 \{ \z = (\x,\y,\x',\y') \in \cZ :  
 \sum_{i=1}^d \alpha_i (x_i - x_i')^q  -  \alpha_{i+d} (y_i - y_i' )^q \ge 0  \},\\
&\cB_{\boldsymbol{\alpha}}^{-} = 
 \{ \z = (\x,\y,\x',\y') \in \cZ :  
 \sum_{i=1}^d \alpha_i (x_i - x_i')^q  -  \alpha_{i+d} (y_i - y_i' )^q < 0  \}.
\end{align*}
The initial problem on a measure $\pi$ is then reformulated as a problem on $2^{2d+1}$ measures $\pi_{\boldsymbol{\alpha}}^\pm$, $\boldsymbol{\alpha}  \in \{-1,1\}^{2d}$. 
\\
\par 
With the approach above, the number of measures is exponential in $d$. 
For $p$ even, in order to reduce the number of measures,
an alternative approach is to write the cost as 
$$
  g(\z)^p = \sum_{\mathbf{k}  \in \bN^{2d} , \vert \mathbf{k}\vert = p  } {p \choose {\mathbf{k}}} \prod_{i=1}^d \vert x_i - x_i' \vert^{qk_i} \prod_{i=1}^{d} (-1)^{k_{i+d}}\vert y_i - y_i' \vert^{qk_{i+d}},
$$
and for each $\mathbf{k}  \in \bN^{2d} $, with $\vert \mathbf{k}\vert = p$, introduce a partition adapted to the piecewise polynomial 
$p_{ \mathbf{k}}(\z) := \prod_{i=1}^d \vert x_i - x_i' \vert^{qk_i} \prod_{i=1}^{d} \vert y_i - y_i' \vert^{qk_{i+d}}$, and as many measures as the number of elements in the partition.  To a polynomial $p_{ \mathbf{k}}(\z)$ is associated a partition 
composed by at most $2^{m_{\mathbf{k}}}$ elements, with $m_{\mathbf{k}} \le p$ the number of odd entries in $\mathbf{k}$. This yields a reformulation with a number of measures bounded by $2^p {2d + p  \choose {2d}} = O(d^p)$.
As an example, for $p=2$, 
$$
g(\z)^2 =   \sum_{i=1}^d\sum_{j=1}^d  \vert x_i - x_i'\vert \vert x_j - x_j'\vert  - \sum_{i=1}^d\sum_{j=1}^d  \vert x_i - x_i'\vert \vert y_j - y_j'\vert + \sum_{i=1}^d\sum_{j=1}^d  \vert y_i - y_i'\vert \vert y_j - y_j'\vert,
$$
which can be reduced to  a sum of $2 d^2 + d$ piecewise polynomials, each of these piecewise polynomials being associated with a partition composed by $2$ or $4$ elements. This yields a reformulation in $O(d^2)$ measures.

%
%
%

\subsection{Gromov-Wasserstein barycenters}\label{sec:gromov-barycenters}

Using the same notations as in Section \ref{sec:wasserstein-barycenters}, we say that $\bary(\rY_N, \Lambda_N) \in \cP(\cX)$ is a Gromov-Wasserstein barycenter associated to a given set $\rY_N = (\mu_i)_{1\leq i\leq N}$ of $N$ probability measures in $\cP(\cY)$ and to a given set of weights $\Lambda_N = (\lambda_i)_{1\leq i\leq n} $ in the simplex  $\Sigma_N$, if and only if $\bary(\Lambda_N, \rY_N)$ is a solution to
\begin{equation}
\label{eq:barygromov}
\inf_{\nu \in \cP(\cX)} \sum_{i=1}^N \lambda_i GW_{p}^p(c_\cX , c_\cY ; \nu,\mu_i).
\end{equation}
Existence and uniqueness of minimizers of \eqref{eq:barygromov} has been studied in depth in \cite{AC2011} for $ GW_{2,2}$. It is shown, in particular, that if one of the $\mu_i$ has a density, the barycenter is unique. In the following we assume existence of minimizers.
Problem \eqref{eq:barygromov} can be written as a quadratic optimization problem  
$$
\inf_{\nu , \pi_1,\hdots , \pi_N} \sum_{i=1}^N \lambda_i \cL^{GW_p}(\pi_i)
$$
over measures $\nu \in \cP(\cX)$ and $\pi_i \in \cM(\cX \times \cY)_+$, $1\le i \le N,$ satisfying the constraints $\pi_i\in \Pi(\cX\times\cY; \nu,\mu_i)$.

When $p$ is even and the costs $c_\cX$ and $c_\cY$ are polynomials, this can be equivalently written as a generalized moment problem
\begin{align}
\inf_{y , y_1,\hdots , y_N} \sum_{i=1}^N \lambda_i L^{GW_p}(y_i) \label{gromov-barycenter-Ly}
\end{align}
over sequences that satisfy moment sequence conditions  $y\in MS(\cX)$ and $y_i \in MS(\cX\times \cY)$, $1\le i \le N,$  and the additional constraints $y_i \in \Pi_{mom}(y, m(\mu_i))$ for $1\le i \le N$.

When $p$ is odd and the costs $c_\cX$ and $c_\cY$ are polynomials, using notations of Section \ref{sec:gromov-poly}, we can introduce additional measures $\gamma_i^-$ and $\gamma_i^+$ supported on $\cA^+$ and $\cA^-$ respectively, and the problem is reformulated as 
$$
\inf_{y , y_1,\hdots , y_N , y_1^+, \hdots, y_N^-} \sum_{i=1}^N \lambda_i L^{GW_p}_{aug}(y_i^+ - y^- _i ),
$$
with the same constraints as before for $y, y_1,\hdots,y_N$ and the additional constraints $\gamma_i ^\pm \in MS(\cA^\pm)$  and $y_i^+ + y_i^- = y_i \otimes y_i$, $1\le i\le N$.

When the costs $c_\cX$ and $c_\cY$ are piecewise polynomials, e.g. for $GW_{p,q}$ with odd $q$, the problem can still be reformulated as a generalization moment problem up to the introduction of new measures, following Section \ref{sec:gromov-pwpoly}. The derivation is rather  technical but straightforward.

\section{SoS-moment hierarchy}
\label{sec:hierarchy}

All OT problems considered in this paper are  of the form
\begin{align}
\rho := \inf_{\pi_1 \in \cM(\cX_1)_+, \hdots, \pi_M \in \cM(\cX_M)_+}  \cG(\pi_1,\hdots,\pi_M) \label{gen-OT}
\end{align}
under additional constraints
$$
\cH_j(\pi_1,\hdots,\pi_M) = b_j, \quad j\in \Gamma,
$$
where $\cG$ and $\cH_j$, $j\in \Gamma,$ are linear or quadratic functions of a finite set of moments of the measures $\pi_1, \hdots, \pi_M$, and $\Gamma$ is a countable set.
The constraints include that
$\mathrm{mass}({\pi_i}) = m_{0}(\pi_i) \le 1.$ 
Problem \eqref{gen-OT} can be equivalently formulated as a generalized moment problem 
\begin{align}
\rho := \inf_{(y_1,\hdots,y_M) \in K}  G(y_1,\hdots,y_M) \label{gen-OT-mom}
\end{align}
where $K$ is the set of sequences $y_1 \in MS(\cX_1), \hdots, y_M \in MS(\cX_M)$ that satisfy 
the  constraints 
$$
H_j(y_1 , \hdots,y_M) = b_j, \quad j\in \Gamma,
$$
and where the functions $G : \bR^{\bN^{n_1}} \times \hdots \times \bR^{\bN^{n_M}} \to \bR$ and $H_j : \bR^{\bN^{n_1}} \times \hdots \times \bR^{\bN^{n_M}} \to \bR$ are linear or quadratic functions involving only finitely many entries of the input sequences $y_1, \hdots,y_M$. The constraints include the conditions 
$(y_i)_{0} \le 1$ for all $1\le i \le M$. 

The $\cX_i$ are assumed to be compact semi-algebraic sets defined by 
 $$
 \cX_i= \{\x_i \in \bN^{n_i} : g_{i,j}(\x_i) \ge 0 , \; 0 \le j \le J_i \},
 $$
 for some polynomials $g_{i,j} $ over $\bR^{n_i}$, where $g_{i,0}(\x_i) = 1$ and $g_{i,1}(\x_i) = R^2 - \Vert \x_i\Vert_2^2 $ for $\x_i\in \bR^{n_i}$, $1\le i\le M$, where $R>0.$ From Theorem \eqref{th:mom-seq-semidefinite}, the moment sequence condition $y_i \in MS(\cX_i)$ is equivalent to 
 the following set of positive semidefinite constraints
  \begin{align}
  \pbM_{r}(g_{i,j} y_i) \succcurlyeq 0 , \quad \forall j \in \{0,\hdots,J_i\}, \quad i \in \{1,\hdots,M\}, \quad r\in \bN. \label{eq:gen-OT-semdefpos}
  \end{align}
  The matrix $\pbM_{r}(g_{i,j} y_i) $ depends linearly on the entries $(y_i)_{\balpha}$ of order  $\vert \balpha \vert \le r_{i,j} + 2r$ with $r_{i,j} = \lceil deg(g_{i,j})/2 \rceil$.
We assume that $G$ only involves 
moments of order up to $r_G$, and the function 
$H_j$ only involves moments of order up to $r_{H_j}$.

 The Lasserre's (or SoS-moment) approach for solving \eqref{gen-OT} consists in considering a hierarchy of problems
 \begin{align}
\rho_r := \inf_{(y_1, \hdots,y_M) \in K_r}  G(y_1,\hdots,y_M) \label{eq:gen-OT-mom-r}
\end{align}
where $K_r$ is the set of sequences $y_1 \in MS_r(\cX_1), \hdots, y_M \in MS_r(\cX_M)$ that satisfy 
 the constraints 
\begin{align}
&H_j(y_1 , \hdots,y_M) = b_j, \quad j\in \Gamma_r, \label{eq:gen-OT-constraints-r}
\end{align}
with $\Gamma_r = \{j\in \Gamma : r_{H_j }\le 2 r\}$, and where $MS_r(\cX_i)$ is the set of sequences $y_i$ that satisfy
\begin{align}
  \pbM_{r-r_{i,j}}(g_{i,j} y_i) \succcurlyeq 0 , \quad \forall j \in \{0,\hdots,J_i\}. \label{eq:gen-OT-semdefpos-r}
  \end{align}
  Problem \eqref{eq:gen-OT-mom-r} is called a relaxation of order $r$ of problem \eqref{gen-OT-mom}. 
  These problems are considered  for $r \ge r^* :=  \max\{ \lceil r_G/2\rceil ,\max_{i,j} r_{i,j} \}$.  
   They only involve the  entries of $y_1, \hdots,y_M$ of order less than $2r$, and can be formulated over $M$ finite dimensional vectors $y_i^r$ in $\bR^{\bN^{n_i}_{2r}}$, $1\le i\le M$.   Then $y_i^r$ can be considered again as an infinite sequence in $\bN^{n_i}$ by completion with zeros.  
 
 \begin{theorem}\label{th:conv-hierarchy}
 Problem \eqref{eq:gen-OT-mom-r} admits a solution for all $r \ge r^*$. 
The sequence $(\rho_r)_{r\ge r^*}$ is increasing and  $\rho_r \to \rho$ as $r\to \infty.$ Moreover, from a sequence of solutions $(y_1^r,\hdots,y_M^r)$ of problems  \eqref{eq:gen-OT-mom-r}, we can extract a subsequence $(y_1^{r_k},\hdots,y_M^{r_k})$  such that for each $1\le i\le N$ and $\balpha \in \bN^{n_i}$, 
 $$
 (y_{i}^{r_k})_{\balpha} \to (y_i)_{\balpha} \quad \text{as $k\to \infty$},
$$ 
where the set of sequences $(y_1, \hdots,y_M)$ is a solution of problem \eqref{gen-OT-mom} and admits a representing measure 
$(\pi_1,\hdots,\pi_M)$ solution of  \eqref{gen-OT}. If 
  \eqref{gen-OT-mom} (or equivalently \eqref{gen-OT}) admits a unique solution, then we have the convergence of the whole sequence 
 $(y_{i}^{r})_{\balpha}$ to $(y_i)_{\balpha}$ as $r\to \infty$, for all $\balpha \in \bN^{n_i},$ where $(y_1,\hdots,y_M)$ is the solution of \eqref{gen-OT-mom}.   
  \end{theorem}
 \begin{lemma}\label{lem:bound-sequence}
  Let $r\in \bN$ and consider a sequence $y \in \bN^{n}$. If $\pbM_r(y) \succcurlyeq 0$ and $\pbM_{r-1}(g y) \succcurlyeq 0$ with $g(\x) = R^2 - \Vert \x\Vert_2^2$ for $r\in \bN$, then for all $0\le k \le r$,
 $$
 \vert y_\balpha \vert \le y_0 \max\{1 , R^{2k}\}, \quad \forall \balpha \in \bN^n_{2k}.
 $$
 \end{lemma}
 \begin{proof}
 Since $\pbM_r(y) \succcurlyeq 0$ implies $\pbM_k(y) \succcurlyeq 0$ for all $0\le k \le r$, we deduce from  \cite[Prop. 3.6]{lasserre2009moments} that 
 \begin{align}
 \vert y_{\balpha} \vert \le \max\{y_{0} , \max_{1\le i \le n} \ell_y(x_i^{2k})\}, \quad \forall \balpha \in \bN^n_{2k}\label{eq:prop3.6}
\end{align}
for all $0 \le k\le r.$
 Moreover $\pbM_{r-1}( (R^2 - \Vert \x\Vert_2^2)y) \succcurlyeq 0$ is equivalent to $\ell_y(R^2 f^2) - \sum_{i=1}^n \ell_y(x_i^2 f^2) \ge 0$ for all $f \in \bR[\x]_{r-1}$. 
 Taking $f=1,$ we obtain $\ell_y(x_i^2) \le y_{0} R^2$. Then taking $f(\x) = x_i^{k-1}$ with $1\le k \le r$, we obtain $\ell_y(R^2 x_i^{2k-2}) - \sum_{j=1}^n \ell_y(x_j^2 x_i^{2k-2}) \ge 0$, which implies $\ell_y(x_i^{2k}) \le R^2 \ell_y(x_i^{2k-2}) \le R^{2 k} y_0$. Then $\max_{1\le i \le n} \ell_y(x_i^{2k}) \le y_0 R^{2k}$ and we conclude by using \eqref{eq:prop3.6}.
  \end{proof}

\begin{proof}[Proof of Theorem \ref{th:conv-hierarchy}]
The proof is adapted from the proof of  \cite[Theorem 4.3]{lasserre2009moments}. 
We detail it for the sake of completeness. Clearly, $K_r \supset K_{r+1} \supset \hdots \supset K$ for all $r\ge r^*$, so that $\rho_r$ is increasing with $r$ and $\rho_r \le \rho$. For all $1\le i\le M$, we have
  $(y_i^r)_0 \le 1$ and  $g_{i,0}=1$ and $g_{i,1}=R^2 - \Vert \cdot \Vert^2_2$. Then from the constraints  
\eqref{eq:gen-OT-semdefpos-r} and Lemma \ref{lem:bound-sequence}, we deduce  
$$
\vert (y^r_{i})_{\balpha} \vert \le    \tau_{\omega(\balpha)}, \quad \omega(\balpha) = \lceil \vert \balpha \vert/2 \rceil 
$$
with $\tau_k =  \max\{1,R^{2k}\}.$ 
We deduce that $K_r$ is a compact set of a finite dimensional space, and from the continuity of $G$ and $H_j$, $j\in \Gamma_r$, we deduce that 
\eqref{eq:gen-OT-mom-r} admits a solution $y^r = (y^r_1,\hdots,y^r_M)$. Now we identify each $y^r_i$ with a sequence in $\bN^{n_i}$ with components $(y^r_i)_{\balpha} = 0$ for $\vert \balpha \vert >2r.$ 
We introduce sequences $\hat y^r_i \in \bN^{n_i}$ defined by
$$
 (\hat y^r_i)_{\balpha} := (y^r_i)_{\balpha} / \tau_{\omega(\balpha)},
$$
which are such that $\Vert \hat y^r_i \Vert_{\ell^\infty} \le 1$ and $\hat y^r_i \in c_0 := \{ y \in \bN^{n_i}  : \lim_{\vert \balpha \vert \to 0} y_\balpha = 0\} \subset \ell^\infty$. Since $c_0$ is the topological dual of $\ell^1$, we have by the Banach-Alaoglu theorem that the unit ball $B_1(c_0)$ of $c_0$ is compact in the  weak-* topology $\sigma(c_0,\ell^1)$. Therefore, we can extract a subsequence $(\hat y_i^{r_k})_{k\ge 1}$ of $(\hat y_i^{r})_{r\ge r^*}$ which converges to some $\hat y_i \in B_1(c_0)$ in the weak-* topology. In particular, this implies that for all fixed $\balpha \in \bN^{n_i}$, $(\hat y^{r_k}_i)_{\balpha} \to (\hat y_i)_\balpha$ as $k \to \infty$ and therefore,  $(y^{r_k}_i)_{\balpha} \to (y_i)_{\balpha}$ as $k\to \infty$, where $y_i \in \bN^{n_i}$ is  defined by  $(y_i)_{\balpha} = (\hat y_i)_\balpha \tau_{\omega(\balpha)}$.
\\
Since the function $G(y_1,\hdots,y_M)$ only depends continuously on the finite set of variables $\{ (y_i)_{\balpha} : \vert \alpha \vert \le r_G, 1\le i \le N \}$, we deduce that 
$\rho_{r_k} = G(y^{r_k}_1,\hdots,y^{r_k}_M) \to G(y_1,\hdots,y_M) $ as $k\to \infty$. 
Also, for a fixed $j \in \Gamma$, since $H_j$ only depend continuously on the finite set of variables $\{ (y_i)_{\balpha} : \vert \balpha \vert \le r_{H_j,} 1\le i \le N \}$, we have that $H_j(y_1 , \hdots , y_M) = \lim_{k \to \infty} H_j(y_1^{r_k} , \hdots , y_M^{r_k}) = b_j. $ Also, for any $m \in \bN$, since $ \pbM_{m}(g_{i,j} y_i)$ only depends continuously in a finite set of variables $\{ (y_i)_{\balpha} : \vert \balpha \vert \le r_{i,j} + 2m, 1\le i \le N \}$, and from the closedness of the positive cone of symmetric positive semidefinite matrices, we deduce that $ \pbM_{m}(g_{i,j} y_i) = \lim_{k\to  \infty} \pbM_m(g_{i,j} y_i^{r_k}) \succcurlyeq 0$. Hence  
 $(y_1,\hdots,y_M) \in K$ and 
 $$
 \rho \le G(y_1,\hdots , y_M) = \lim_{k\to \infty}  \rho_{r_k} \le \rho,
 $$
 which proves that $(y_1,\hdots , y_M)$ is a solution of \eqref{eq:gen-OT-mom-r}. 
Since $\rho_r$ is increasing, this implies that the whole sequence $\rho_r $ converges to $\rho$ as $r\to \infty$. If the solution of \eqref{eq:gen-OT-mom-r} is unique, then from all subsequences of $((y_i^r)_{\balpha})_{r\ge r^*}$, we can extract a subsequence that converges to the same limit $((y_i)_{\balpha})_{r\ge r^*}$, which implies the convergence of the whole sequence.  
\end{proof}

\section{Post-processing}
\label{sec:postproc}
 
 Here we consider the post-processing of the solution of the SoS-moment approach. 
 From the solution of the problem \eqref{eq:gen-OT-mom-r}
of order $r$, we obtain an approximation $y^r$ of the moments $y = m(\mu)$ (up to order $2r$) of some probability measure of interest $\mu$ over a basic semi-algebraic set $\cX \subset \bR^n$, which is the target solution of the initial OT problem. 
We here assume that $\mu$ is the unique solution of the initial OT problem. By theorem \ref{th:conv-hierarchy}, we have that $y^r_\balpha$ converges to $m_\balpha(\mu)$ as $r\to \infty$, for each $\balpha \in \bN^n$.

\subsection{Approximation of linear quantities of interest}
From approximate moments, we directly obtain an estimation of the first statistics of $\mu$ and its marginals (mean, variance, covariance...) or more generally of any quantity of interest 
 $$I(g) = \int_\cX g(x) d\pi(\x) .$$
 For a polynomial   $g = \sum_{\vert \balpha \vert\le p}  c_\balpha  \x^\balpha \in \bR[\x]_p$, $I(g) =  \ell_{m(\pi)}(g)$ 
is estimated by $$I_r = \ell_{y^r}(g) =  \sum_{\vert \balpha \vert\le p}  c_\balpha  y^r_\balpha $$ and we have that $I_r \to I(g)$  as $r \to \infty$. For a function $g$ which is not a polynomial, the quantity $I$ can be approximated by $I_{r,p}  = \ell_{y^r}(g_{p})$ where $g_p =  \sum_{\vert \balpha \vert\le p}  c_\balpha  \x^\balpha \in \bR[\x]_p$ is a polynomial approximation of $g$, and 
$$
\vert I - I_{r,p} \vert = \vert \int_\cX (g - g_p) d\mu(\x) + \ell_{m(\mu) - y^r}(g_p) \vert \le  \Vert g - g_p \Vert_{L^\infty(\cX)} + \sum_{\vert \balpha \vert\le p} \vert c_\balpha  \vert \vert y^r_\balpha - m_\balpha(\mu) \vert,
$$
with $\Vert g - g_p \Vert_{L^\infty(\cX)}$ the error of approximation of $g$ by $g_p$.  
We have that  $I_{r,p}$ converges to $I $ as $r ,p \to\infty$.
Studying the rate of convergence of $I_{r,p}$ to $I$ requires some additional information on the convergence of $g_p$ and the convergence of the approximate moments.  

\subsection{Approximation of the support of $\mu$}
\label{sec:christoffel}
Here, we show how to estimate the support $S(\mu)$ of $\mu$ from an approximation of its moments, using the Christoffel function. 
Note that $S(\mu)$ is contained in the basic semi-algebraic set  $\cX$.
This methodology has been originally proposed in \cite{marx2021semi}. It is presented and analysed in \cite{pauwels2021data,vu2022rate} in a statistical setting.


For $r\in \bN$, we denote $\Pi_r^n = \bR[\x]_r$ the space of polynomials over $\bR^n$ with degree less than $r$. 
We let $\boldsymbol{\phi}_r(\x) = (\x^\balpha)_{\balpha \in \bN^n_r} \in \bR^{s(r)} $ be the vector of monomials of degree less than $r$, with $s(r) := {n + r \choose r} = \# \bN^n_r = \dim \Pi_r^n .$ 
For any $r \in \bN$, the moment matrix $\pbM_r(\mu) \in \bR^{s(r) \times s(r)}$ of $\mu$, with moments up to order $2r$, is given by
$$
\pbM_r(\mu) = \int_\cX \boldsymbol{\phi}_r(\x) \boldsymbol{\phi}_r(\x)^T d\mu(\x), 
$$
which is the Gram matrix in $L^2_\mu(\cX)$ of the canonical basis of $\Pi_r^n .$ For two polynomials $g(\x) =  \boldsymbol{\phi}_r(\x)^T \mathbf{a} $ and $h(\x) =  \boldsymbol{\phi}_r(\x)^T \mathbf{b} $ in $ \bR[\x]_r$ with coefficient $\mathbf{a}, \mathbf{b} \in \bR^{s(r)}$, $\mathbf{a}^T \pbM_r(\mu) \mathbf{b}  = \int_\cX h(\x) g(\x) d\mu(\x),$ that is the inner product in $L^2_\mu(\cX)$. In practice, an approximation of this moment matrix can be obtained from the solution $y^r$ of a relaxation of order $r$, or from a solution $y^{\tilde r}$ of higher order $\tilde r \ge r$ in order to get a better estimation. 

\paragraph{Non degenerate case:}
Let us first consider the case where $S(\mu)$ is not contained by a proper real algebraic subset of $\cX$. In other words, for any polynomial $p\in \bR[\x]$,  
$$
\int_\cX p(\x)^2 d\mu(\x) = 0 \quad \text{if and only if} \quad p=0.
$$ 
This is the case when $S(\mu)$ has nonzero Lebesgue measure. Hence, $\pbM_r(\mu)$ is invertible and 
the finite-dimensional space $\Pi^n_r$ of polynomials of degree less than $r$ is a reproducing kernel Hilbert space in $L^2_\mu$, whose kernel, called the Christoffel-Darboux kernel, is given for $\x,\y\in \bR^n$ by  (see \cite{dunkl2014orthogonal})
$$
\kappa_{\mu,r}(\x,\y) =  \sum_{i=1}^{s(r)} \varphi_i(\x) \varphi_i(\y)
$$
where $(\varphi_{1}, \hdots, \varphi_{s(r)})$ is some orthonormal basis of $\Pi^n_r$. It can be also written 
$$
\kappa_{\mu,r}(\x,\y) =  \boldsymbol{\phi}_r(\x)^T \pbM_r(\mu)^{-1} \boldsymbol{\phi}_r(\y).
$$ 
The Christoffel function $\Lambda_{\mu,r}$ is defined for $\y \in \bR^n$ by 
$$
\Lambda_{\mu,r}(\x)= \inf \{ \int_\cX p(\y)^2 d\mu(\y) : p \in \Pi_r^n  ,  \; p(\x)=1\}. 
$$
In the present regular case, we have for all $\x$, 
$$
\Lambda_{\mu,r} = \kappa_{\mu,r}(\x,\x)^{-1}.
$$
The support is then approximated by the set 
\begin{align}
S_r(\mu) = \{ \x\in \cX : \Lambda_{\mu,r}(\x) \ge \gamma_r\},\label{eq:Sr_chris}
\end{align}
for some suitably chosen $\gamma_r$. Since $\Lambda_{\mu,r}(\x)\le \gamma_r$ is equivalent to the polynomial inequality $\kappa_{\mu,r}(\x,\x)\le \gamma_r^{-1}$, $S_r$ is a polynomial sublevel set in $\cX$.  

From the Markov inequality, we have that 
$$ \mu(\cX \setminus S_r(\mu))  =  \mu(\{\x : \kappa_{\mu,r}(\x,\x) > \gamma_r^{-1}\}) \le \gamma_r \int_\cX \kappa_{\mu,r}(\x,\x) d\mu(\x) =  \gamma_r {s(r)}.$$
Therefore, by choosing $\gamma_r = \eta/s(r)$, we guarantee that $\mu(S_r(\mu)) \ge 1-\eta$,  that is $S_r(\mu)$ contains a fraction $1-\eta$ of the mass of $\mu.$

When the measure is absolutely continuous with respect to the Lebesgue measure $\lambda$, it is proven in \cite{lasserre2019empirical} that $S_r(\mu)$, with a suitable choice of the sequence $\gamma_r$, converges to $S(\mu)$ in the Haussdorff distance. Also, for a point $\x \notin S(\mu)$, $\Lambda_{\mu,r}(\x)^{-1}$ grows exponentially with $r$, while for a point $\x \in S(\mu)$, it  only grows polynomially. A heuristic approach then consists in estimating the rate of convergence from several values of $r$ in order to decide if $\x$ is in the support or not. 



\paragraph{Singular case:}
We now consider the case where the measure of $\mu$ is contained in an algebraic set, which results in a singular  
  moment matrix $\pbM_r(\mu)$, and we follow  \cite{pauwels2021data} for the definition of an approximate support. 

 We let $V$ be the Zariski closure of $S(\mu)$, which is the smallest  algebraic set containing $S(\mu).$
  We denote by $\cI_r$ the ideal of polynomials in $\Pi_r^n$ that vanish on $V$, which is the set of polynomials $p \in \Pi^n_r$ satisfying $\int p(\x)^2 d\mu = 0$. 
The quotient space $\Pi^n_r / \cI_r$ is a reproducing kernel Hilbert space in $L^2_\mu(\cX)$ with dimension $r'  = \mathrm{rank}(\pbM_r(\mu)) $, with kernel 
$$
\kappa_{\mu,r} (\x,\y)= \sum_{i=1}^{r'} \varphi_i(\x) \varphi_i(\y)
$$
where $\varphi_1,\hdots,\varphi_{r'}$ is an orthonormal basis of $\Pi^n_r / \cI_r$ in $L^2_\mu$. This kernel can be obtained by 
$$
\kappa_{\mu,r} (\x,\y) =  \boldsymbol{\phi}_r(\x)^T \pbM_r(\mu)^{\dagger} \boldsymbol{\phi}_r(\y),
$$
with $\pbM_r(\mu)^{\dagger} $ the Moore-Penrose pseudo-inverse of $\pbM_r(\mu)$ (of rank $r'$), which can be expressed $\pbM_r(\mu)^{\dagger} = \sum_{i=1}^{r'} \lambda_i^{-1} \mathbf{v}_i \mathbf{v}_i^T$ given a spectral decomposition $\pbM_r(\mu) = \sum_{i=1}^{r'} \lambda_i \mathbf{v}_i \mathbf{v}_i^T$ with orthonormal eigenvectors $\mathbf{v}_i$ and corresponding nonzero eigenvalues $\lambda_i$ of $\pbM_r(\mu)$.

A Christoffel function $\Lambda_{\mu,r}$ can still be defined through a variational formulation
$$
\Lambda_{\mu,r}(\x)= \inf \{ \int_\cX p(\y)^2 d\mu(\y) : p \in \Pi_r^n / \cI_r ,  \; p(\x)=1\}. 
$$
We still have $\Lambda_{\mu,r}(\x) = \kappa_{\mu,r}(\x,\x)^{-1}$ for all $\x\in V$, but for $\x\notin V$, the functions  $\Lambda_{\mu,r}(\x)$ and $\kappa_{\mu,r}^{-1}$ differ, which yields two possible definitions of an approximate support $S_r(\mu)$, using either $\Lambda_{\mu,r}(\x)$ or $\kappa_{\mu,r}(\x,\x)^{-1}$, that is either \eqref{eq:Sr_chris} or 
$$
S_r = \{\x\in \cX : \kappa_{\mu,r}(\x,\x)^{-1} \ge \gamma_r\}.
$$ 

\paragraph{Practical aspects:}
The functions $\kappa_{\mu,r}$ and $\Lambda_{\mu,r}$ are functions of the moment matrix $\pbM_r(\mu)$ of the true measure $\mu$. In practice, the measure $\mu$ is replaced by the approximation $\mu_{ r}$ of a relaxation of order $r $, that yields approximate functions $\kappa_{\mu_r,r}$ and $\Lambda_{\mu_r,r}$ and 
 corresponding approximate supports $S_r := S_r(\mu_r).$ Note that for fixed $r$, $\mu$ could be replaced by the solution $\mu_{\tilde r}$ of a relaxation of higher order $\tilde r$. A quantitative approach is still missing. 
   
  \subsection{Approximation of the density}
If the measure $\mu$ admits a density $f$ with respect to a known measure $\nu$ on $\cX$, i.e $d\mu(\x) = f(\x) d\nu(\x)$, then the Christoffel function could also be used to estimate the density on the support $S(\mu)$ (or its estimation), as suggested in \cite{lasserre2019empirical}.

Also, the values $y^r_\balpha$ provide approximations of the moments
$$
m_\balpha(\mu) = \int_\cX \x^\balpha d\mu(\x) =  \int_{\cX} \x^\balpha f(\x) d\nu(\x),
$$
that is the inner product of $f(\x)$ and $\x^\balpha$  in $L^2_\nu(\cX)$. Different types of approximations of $f$ can be obtained from  this information. 
In particular, a polynomial approximation $f_{r,p} = \sum_{\vert \bbeta \vert \le p} a_\bbeta \x^\bbeta$ of $f$, $p\le 2r$, can then be obtained by solving a weighted least-squares problem
$$
\min_{(a_\bbeta)_{\vert \bbeta \vert\le p}} \sum_{\vert \balpha \vert \le 2r} w_\balpha \Big\vert y^r_{\balpha} - \sum_{\vert \bbeta \vert \le p} (G_\nu)_{\balpha,\bbeta} a_\bbeta \Big\vert^2,
$$ 
where $G_\nu$ is a Gram matrix in $L^2_\mu$ with entries   $(G_\nu)_{\balpha,\bbeta} = \ell_m(\nu)(\x^\balpha \x^\bbeta)  = \int_\cX \x^\balpha \x^\bbeta d\nu(\x) $. 
From a computational point of view, the use of canonical polynomial basis may yield to numerical instabilities and high round-off errors. Therefore, the use of other polynomial basis should be preferred. 

Some reformulations of the initial OT problem yield approximations of the moments of the measures $1_{A_k} \mu$ where the $A_k$ form  a partition of $\cX$. 
In this case, a local polynomial approximation of the density on $A_k$ can be computed,  which results in a global piecewise polynomial approximation of $f$ over $\cX$.

 \section{Numerical illustrations}
\label{sec:numerics}
 The aim of this section is to illustrate how the method behaves for the computation of Wasserstein distances, barycenters, and Gromov-Wasserstein discrepancies. We also discuss some choices we make in our implementation. 
 
 \subsection{Wasserstein distances and barycenters}
 
 The code used to generate the examples shown here  is available at
\begin{center}
\href{https://gitlab.tue.nl/data-driven/sos-ot}{https://gitlab.tue.nl/data-driven/sos-ot}
\end{center}
For our numerical tests, we consider cartoon images as displayed in Figure \eqref{img:cartoon-imgages}. The images are $400\times 400$ pixels but we will see each of them as a uniform measure on a subset $S\subset [0, 1]^2$ which is defined as
$$
d\mu(x) = \frac{1}{|S|}  \charFun_{S}(x) \dx
$$
The shape and location of the support $S$ varies for each image. We consider three different types of shapes: smileys, stars, and pacmen.

\begin{figure}[h]
  \centering
  \includegraphics[scale=0.07]{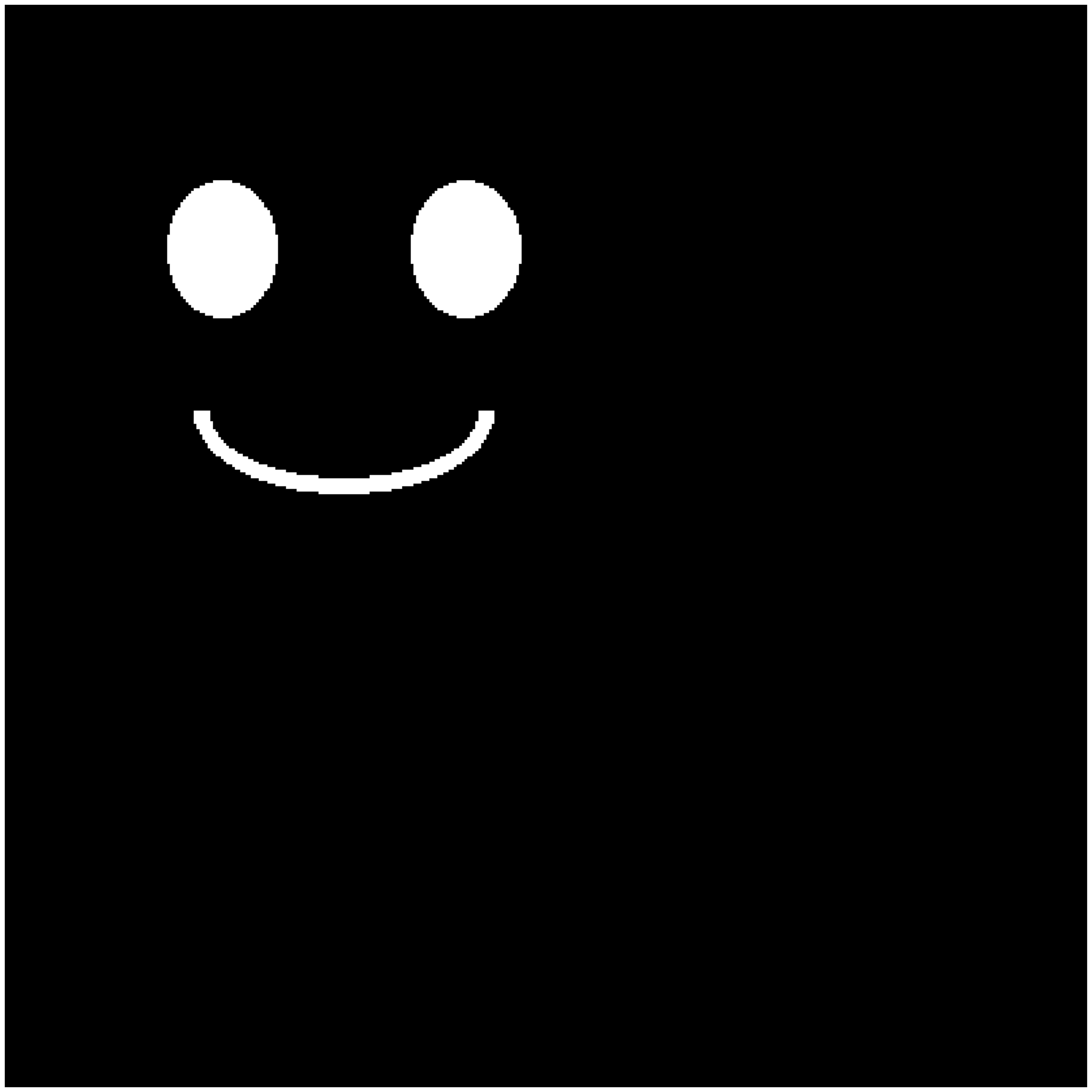}
  \quad
  \includegraphics[scale=0.07]{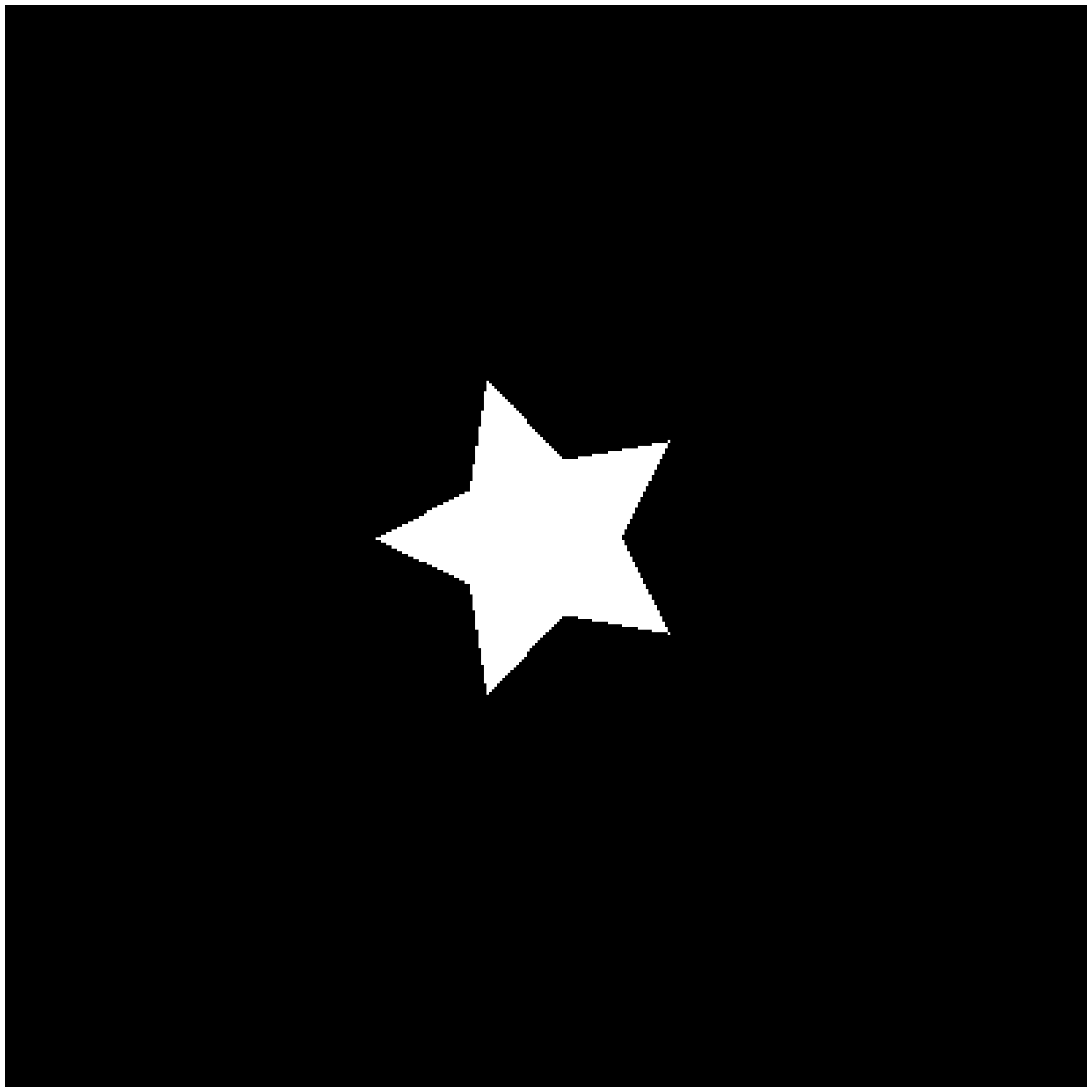}
  \quad
  \includegraphics[scale=0.07]{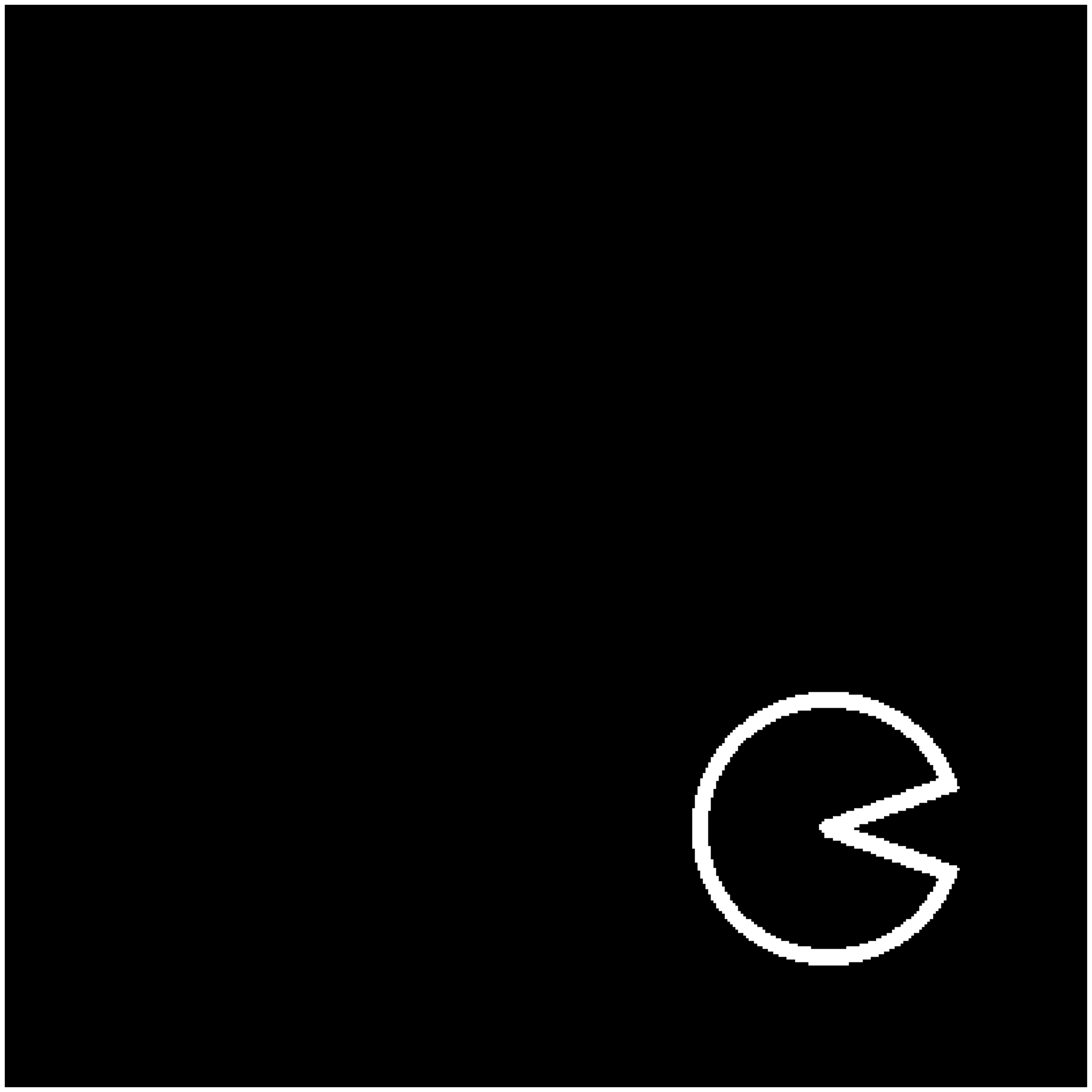}
  \caption{Sample images considered for the tests.}
  \label{img:cartoon-imgages}
\end{figure}

 \subsubsection{$W_2$, $W_1$ distance}
 We start by estimating the $W_1$ and $W_2$ distances of two translated smileys $\mu_1,\,\mu_2$ for which we know the exact translation vector $T=(t_1, t_2)\in \bR^2$ (see Figure \eqref{img:smileys-distances-img}). In this simple case, we know that the exact distance is given by
 $$
 W_p(\mu_1, \mu_2) = || T ||_{\ell_p(\bR^2)}, \quad p \in \{1, 2\},
 $$
 so we can validate the accuracy of our moment approach. Figure \eqref{img:smileys-distances-rel-err} shows the relative errors in the estimation as a function of the relaxation order $r$. We see that we obtain an extremely high accuracy for all relaxation orders. The high accuracy obtained with only the first order is particularly remarkable. Figure \eqref{img:smileys-distances-runtime} shows an exponential increase in the runtime as a function of $r$, which is to be expected given that the number of unknown moments to estimate grows exponentially with $r$. Repeating the same experiment with translated stars and translated pacmen yields similar results, with very high accuracy since the first relaxation order $r=1$.

\begin{figure}
  \centering
  \begin{subfigure}{\textwidth}
    \centering
    \includegraphics[scale=0.07]{{figures/data/smiley-happy00.pdf}}
    \includegraphics[scale=0.07]{{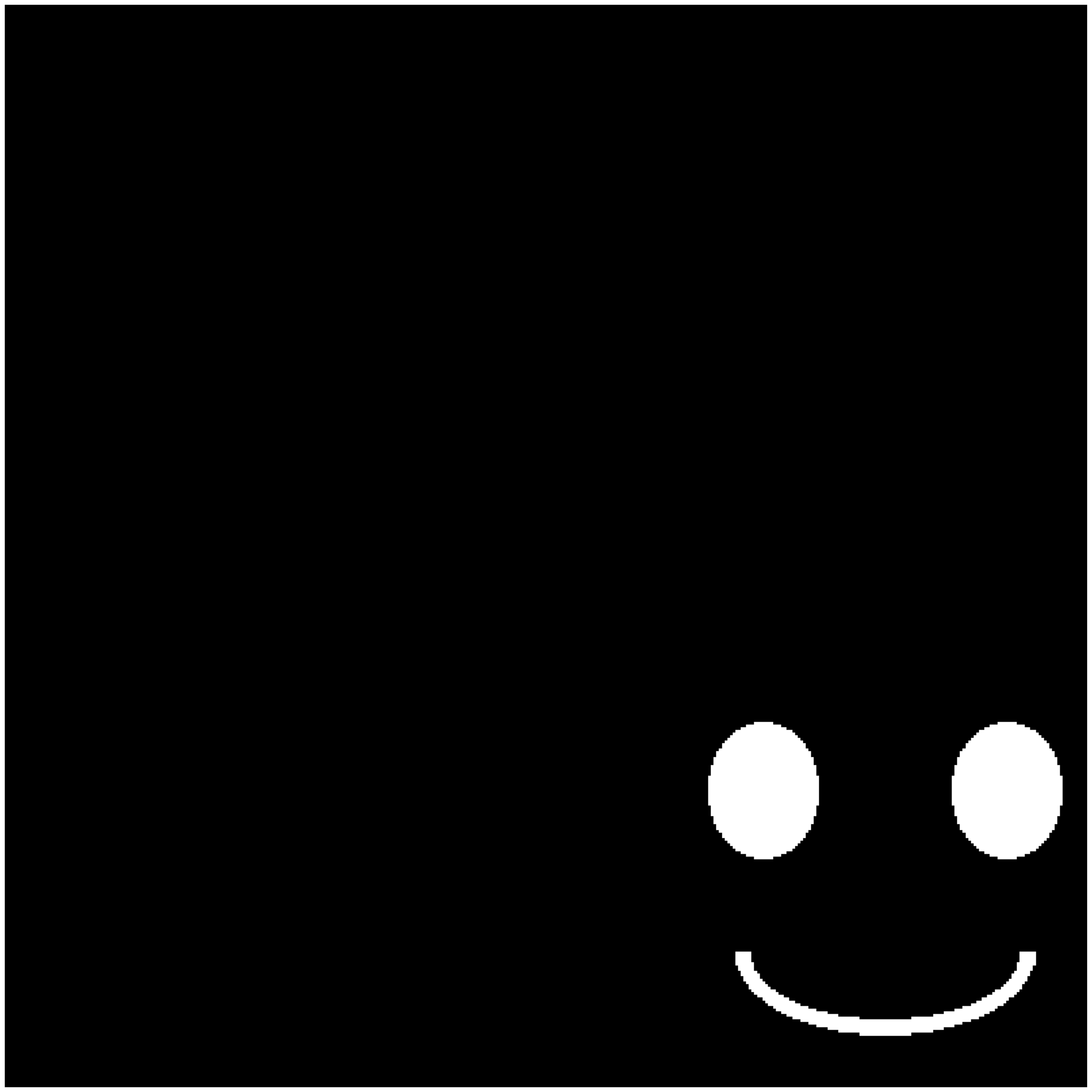}}
    \caption{Translated smileys.}
    \label{img:smileys-distances-img}
\end{subfigure}%
\\
  \begin{subfigure}{.47\textwidth}
      \centering
      \includegraphics[width=\textwidth]{{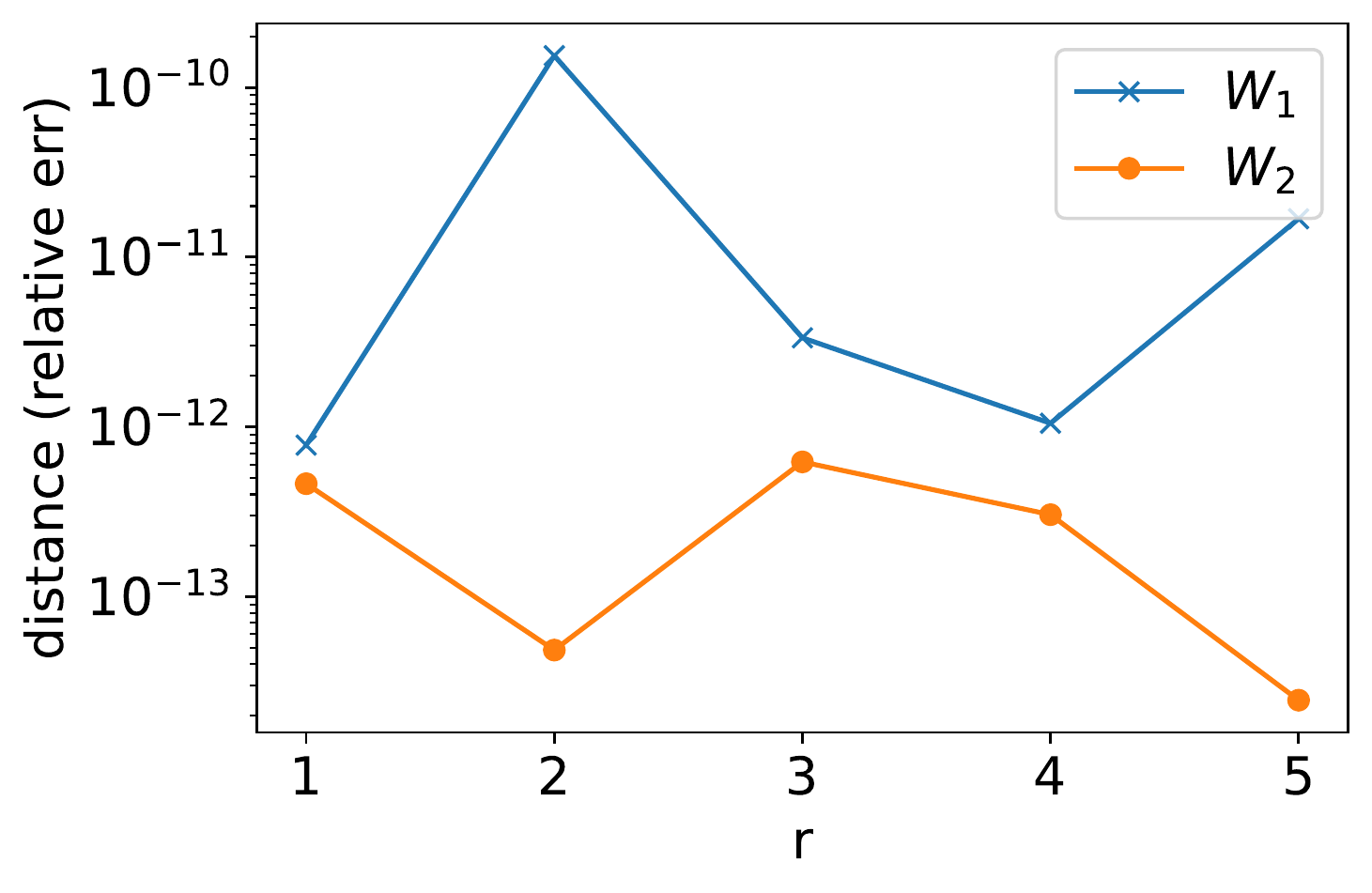}}
      \caption{Relative error for the distance.}
      \label{img:smileys-distances-rel-err}
  \end{subfigure}%
  \hspace{0.04\textwidth}
  \begin{subfigure}{.47\textwidth}
  \centering
  \includegraphics[width=\textwidth]{{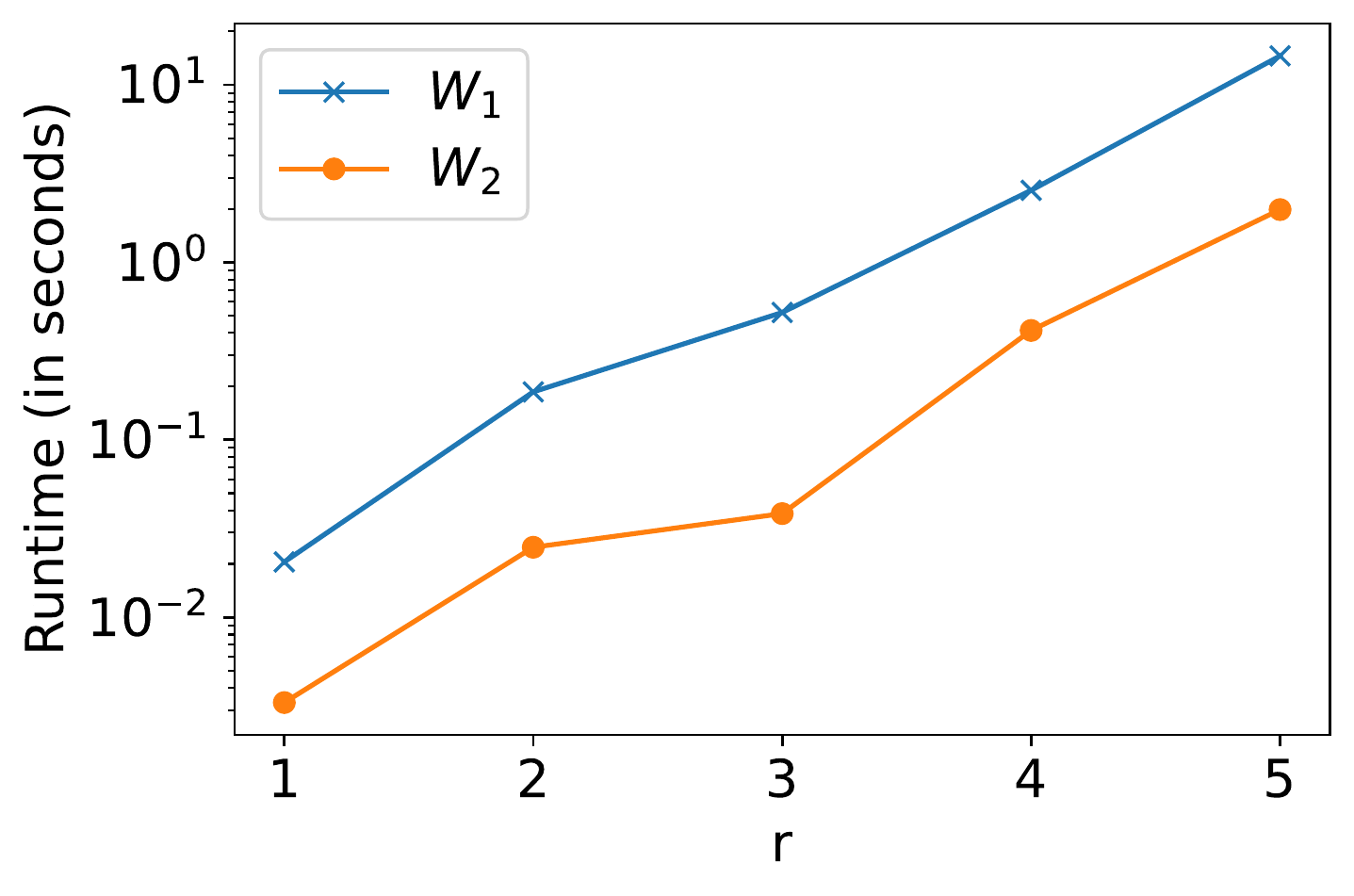}}
  \caption{Runtime.}
  \label{img:smileys-distances-runtime}
  \end{subfigure}
  \caption{$W_1$ and $W_2$ distance between translated smileys.}
  \label{img:smileys-distances}
  \end{figure}
 
 \subsubsection{$W_2$-Wasserstein barycenters}
 We now turn to the computation of barycenters. We present the following test for validation purposes: consider the four translated smileys of Figure \eqref{img:barycenter-measures}. We know that the $W_2$ barycenter of these measures with uniform weights $(0.25, 0.25, 0.25, 0.25)$ is equal to the smiley which is located at the center of the other images. Our goal is to study how accurately we can recover that target barycenter with our moment approach.
 
 Since we know the exact barycenter, we also know the exact moments so we start by examining how accurately they are estimated. Figure \eqref{img:barycenter-rel-err-moms} shows the relative error in the computation of the first moments as a function of $r$. Figure \eqref{img:barycenter-max-rel-err-moms} reports the maximum absolute error in the moment estimation for each order $r$. We observe that the absolute errors decays relatively quickly (we gain about a half order of magnitude per relaxation order). Similar observations hold for relative errors. It would be interesting to examine the trend for larger orders but this has not been possible with the current implementation due to conditionning issues, and also due to the use of Mosek as a black-box optimization solver (which prevented us from sparsifying certain variables and operations, which are critical to prevent memory overflows when the complexity grows). We  leave this implementation point for a future contribution  in which we will also explore strategies to solve optimal transport problems in high dimension.
 
 \begin{figure}
  \centering
  \begin{subfigure}{0.47\textwidth}
    \centering
    \includegraphics[scale=0.03]{{figures/data/smiley-happy00.pdf}}
    \qquad
    \includegraphics[scale=0.03]{{figures/data/smiley-happy22.pdf}}\\
    \includegraphics[scale=0.03]{{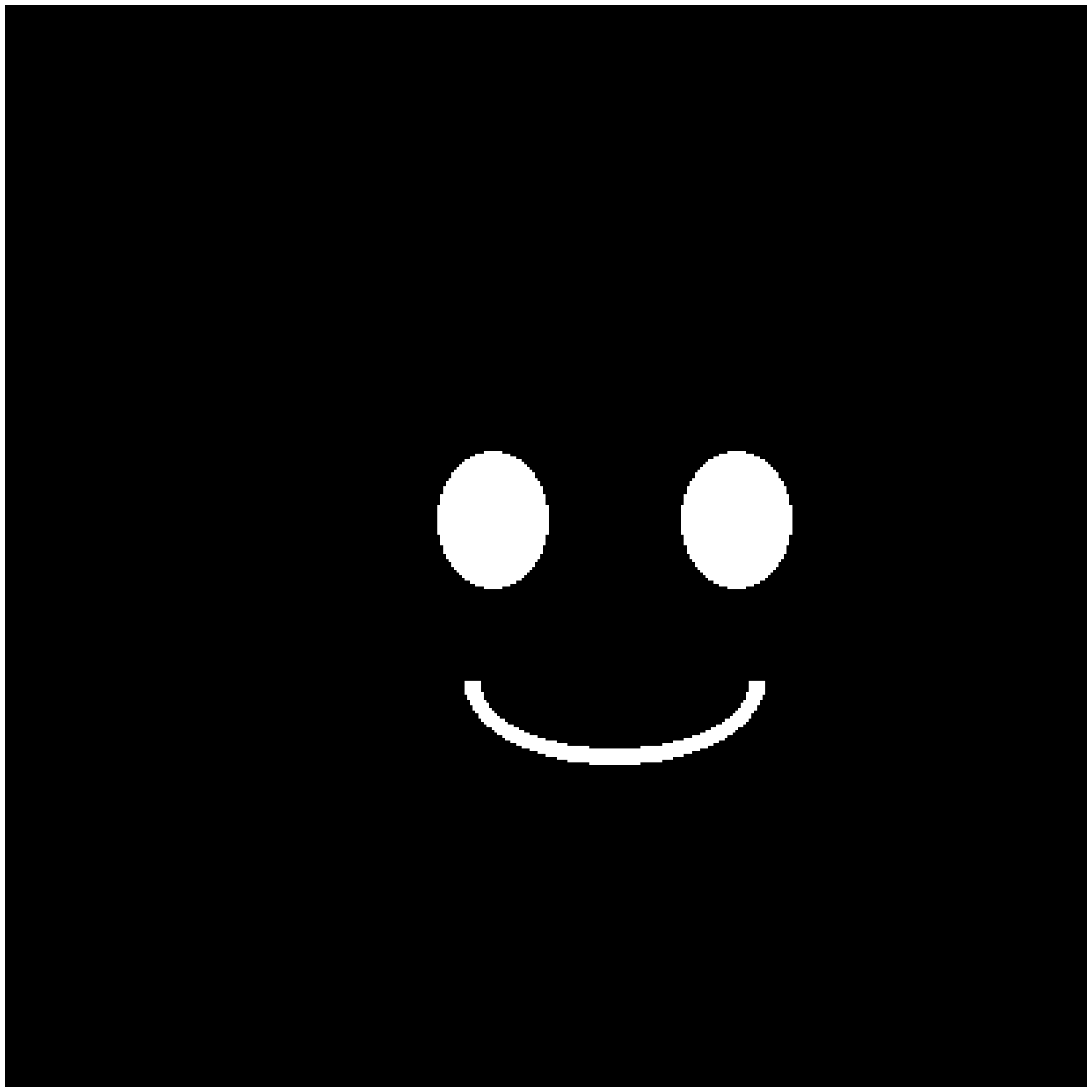}}\\
    \includegraphics[scale=0.03]{{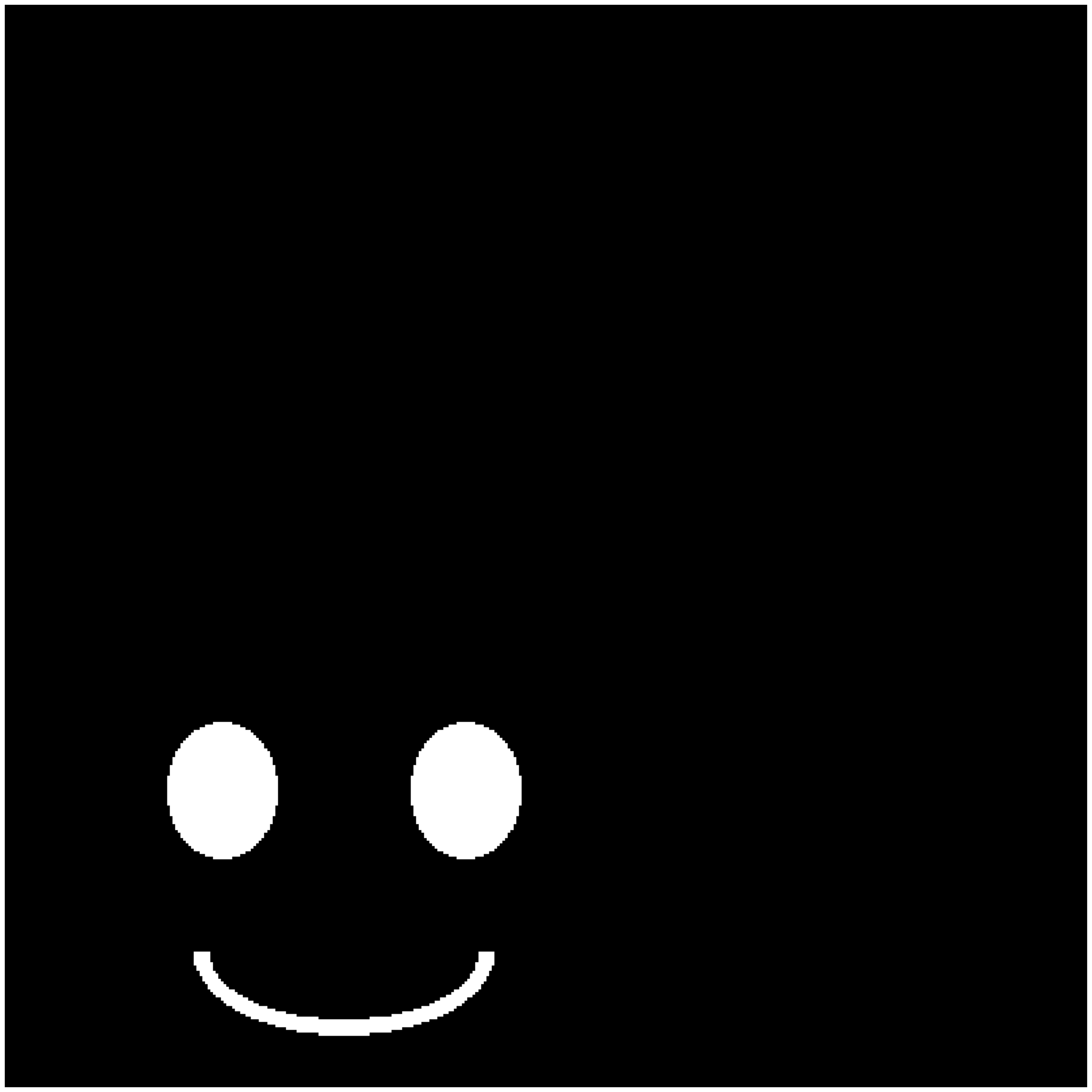}}
    \qquad
    \includegraphics[scale=0.03]{{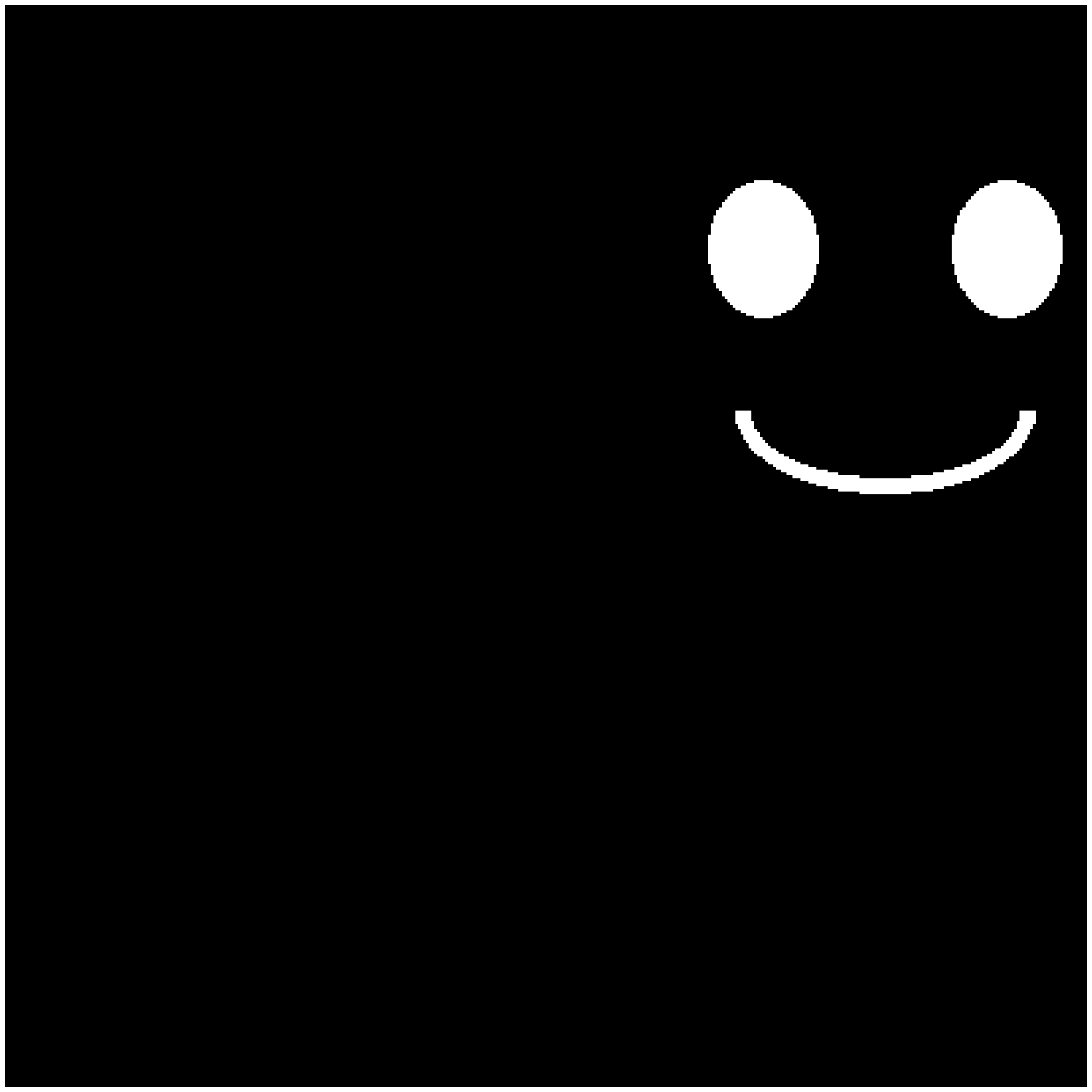}}
    \caption{The barycenter of the four smileys in the corners with uniform weights is equal to the smiley in the center.}
    \label{img:barycenter-measures}
\end{subfigure}%
~
\begin{subfigure}{0.47\textwidth}
  \centering
  \includegraphics[width=\textwidth]{{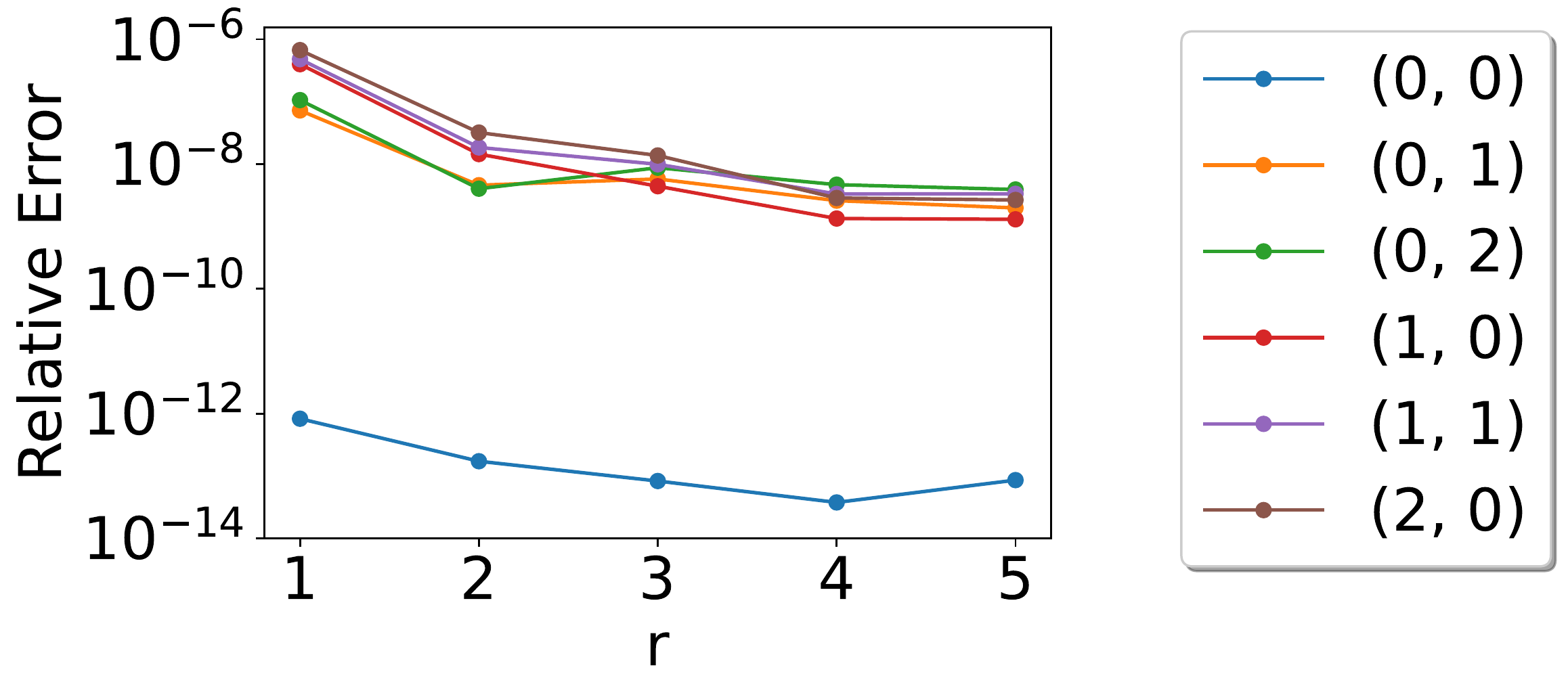}}
  \caption{Relative error in the estimation of the first moments. Absolute errors give a similar plot.}
  \label{img:barycenter-rel-err-moms}
\end{subfigure}%
\\
\begin{subfigure}{0.47\textwidth}
  \centering
  \includegraphics[width=0.9\textwidth]{{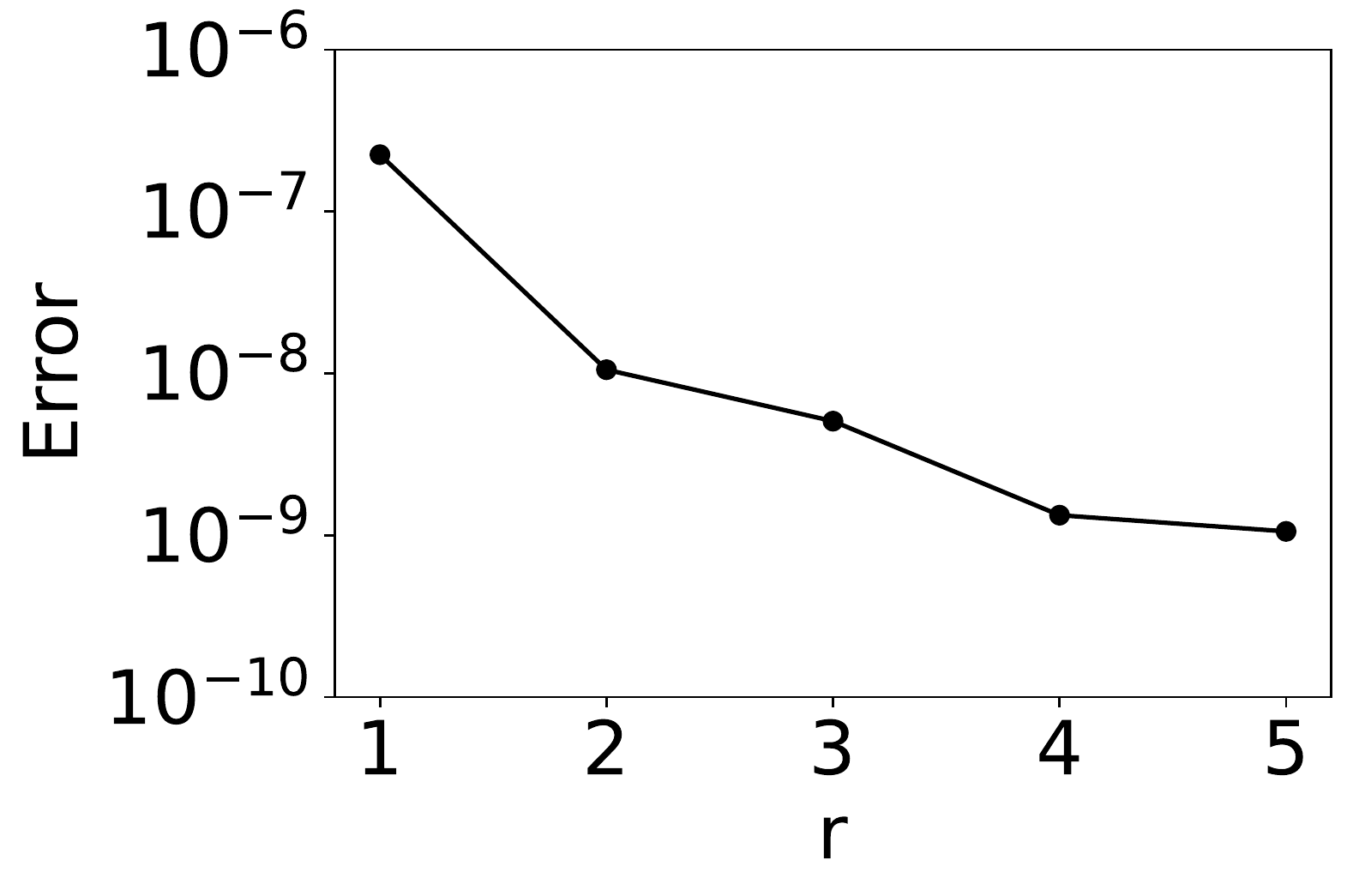}}
  \caption{Maximum absolute error in moment estimation. Relative errors give a similar plot.}
  \label{img:barycenter-max-rel-err-moms}
\end{subfigure}%
~
\begin{subfigure}{0.47\textwidth}
  \centering
  \includegraphics[width=0.9\textwidth]{{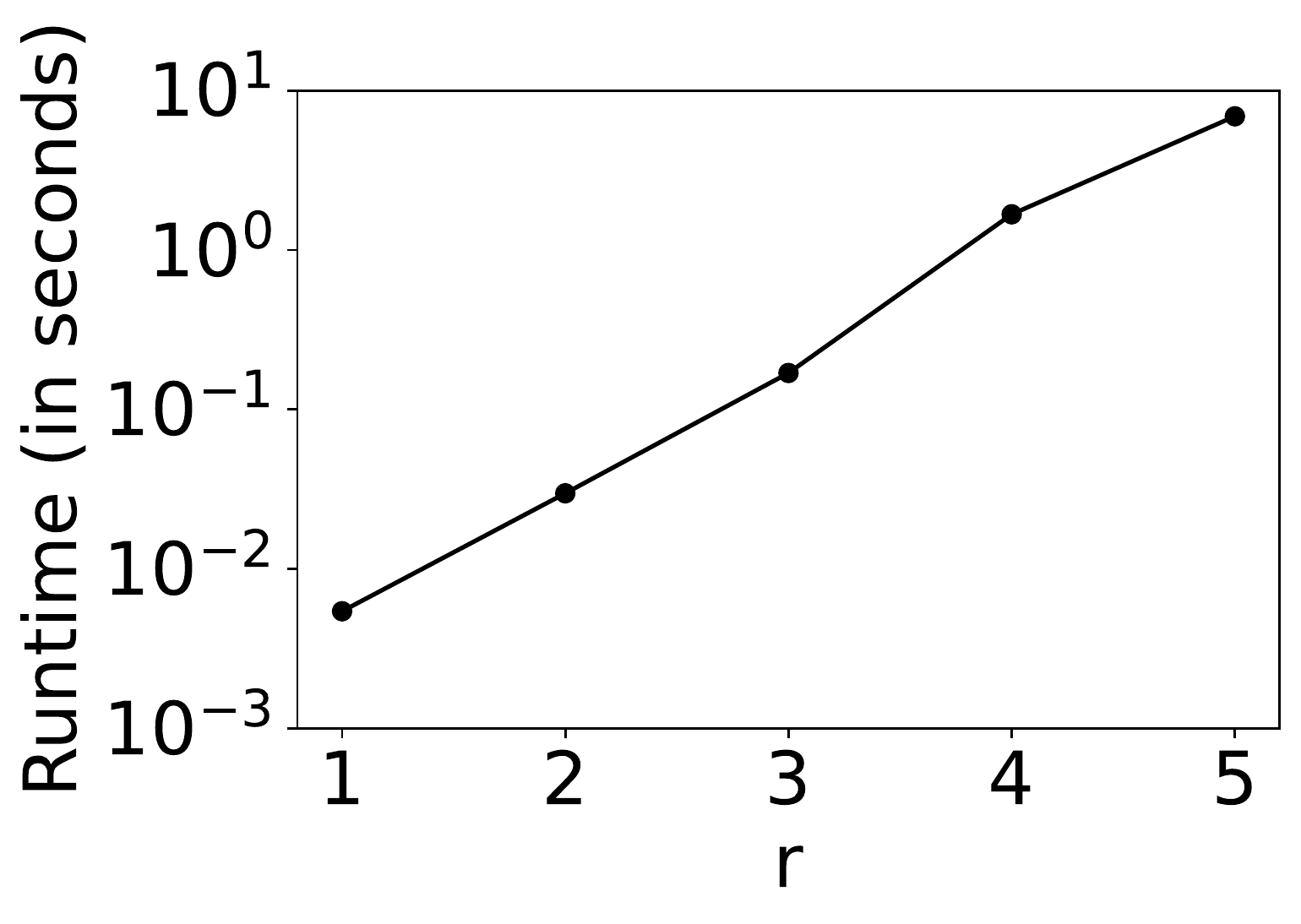}}
  \caption{Runtimes.}
  \label{img:barycenter-runtime}
\end{subfigure}%

  \caption{Barycenter of four smileys.}
  \label{img:barycenter-smileys}
  \end{figure}

  From the moments that our approach provides, we can reconstruct the support of the barycenter by computing the Christoffel function, and applying thresholding techniques discussed in Section \ref{sec:christoffel}. Figure \ref{img:barycenter-CD-smileys} shows the Christoffel function for increasing relaxation orders $r$. The figure also depicts the support that we obtain by thresholding this function with parameter $\gamma_r=0.3$ for all relaxation orders $r$. We may note how the estimation of the support improves as $r$ grows. For $r=4$ and $r=5$ it is possible to ``discover'' that the measure has several non-connected components such as the mouth and the eyes.

  We can repeat the same experiment by replacing the smileys with stars or pacmen. In this case, we obtain very similar results as the ones from Figure \eqref{img:barycenter-smileys} so we do not include them for the sake of brevity. We however plot the Christoffel function and the obtained support after thresholding (see Figures \ref{img:barycenter-CD-stars} and \ref{img:barycenter-CD-pacmans}). We observe that order $r=3$ is already enough to discover that the star has five corners. In the case of the pacman, the method is able to discover a very fair estimation of the support for $r=4$ and $r=5$ but it only gives a coarse approximation of the mouth (approximating this part better would have required higher relaxation orders).

\begin{figure}
  \begin{subfigure}{\textwidth}
    \centering
    \includegraphics[width=0.9\textwidth]{{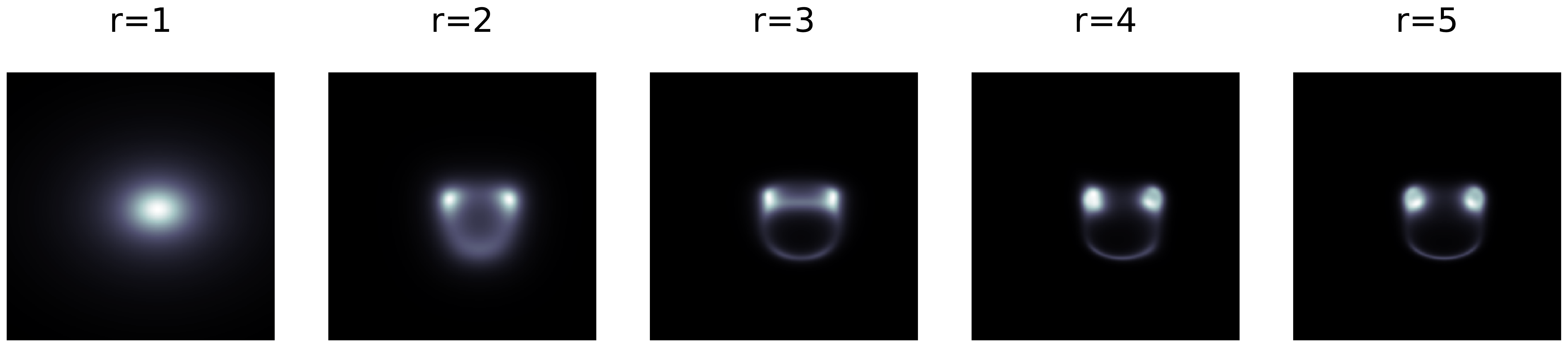}}
    \includegraphics[width=0.9\textwidth]{{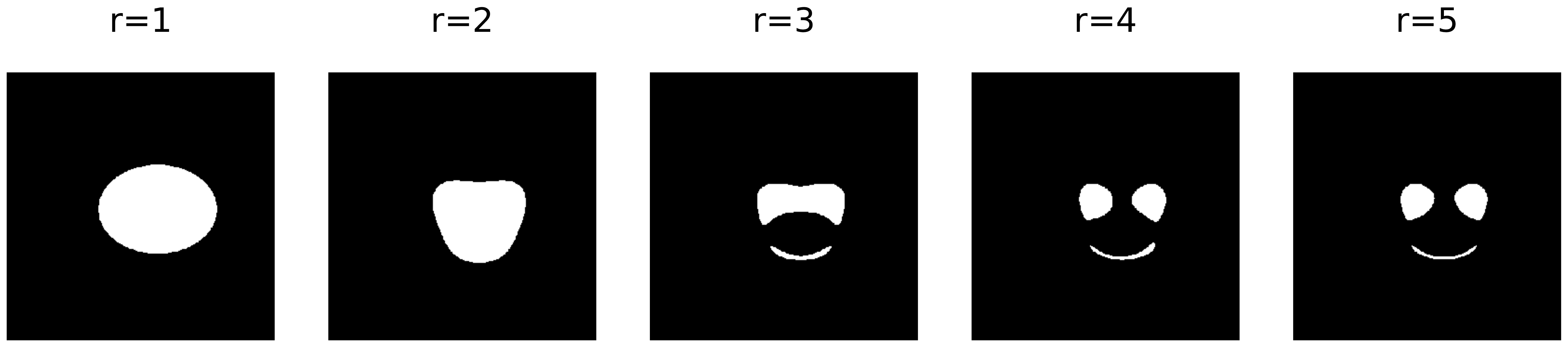}}
    \caption{Barycenter of smileys.}
    \label{img:barycenter-CD-smileys}
  \end{subfigure}%
  \\
  \begin{subfigure}{\textwidth}
    \centering
    \includegraphics[width=0.9\textwidth]{{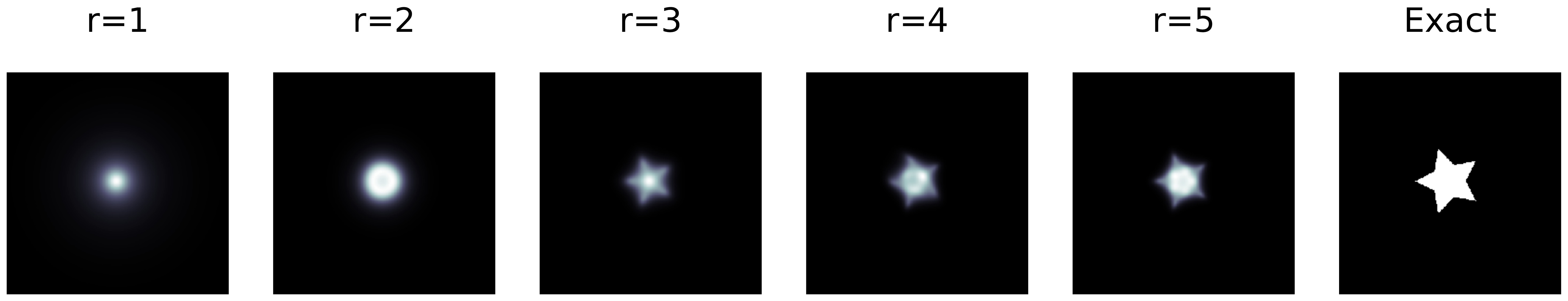}}
    \includegraphics[width=0.9\textwidth]{{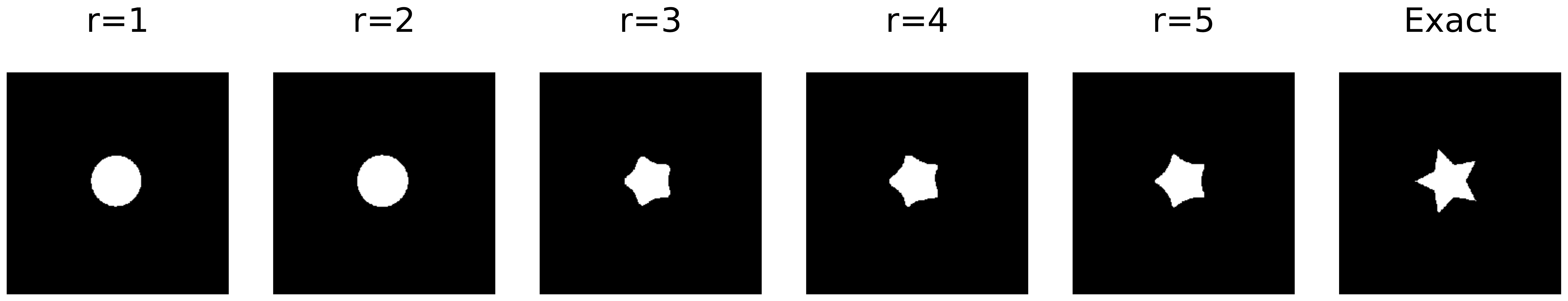}}
    \caption{Barycenter of stars.}
    \label{img:barycenter-CD-stars}
  \end{subfigure}%
  \\
  \begin{subfigure}{\textwidth}
    \centering
    \includegraphics[width=0.9\textwidth]{{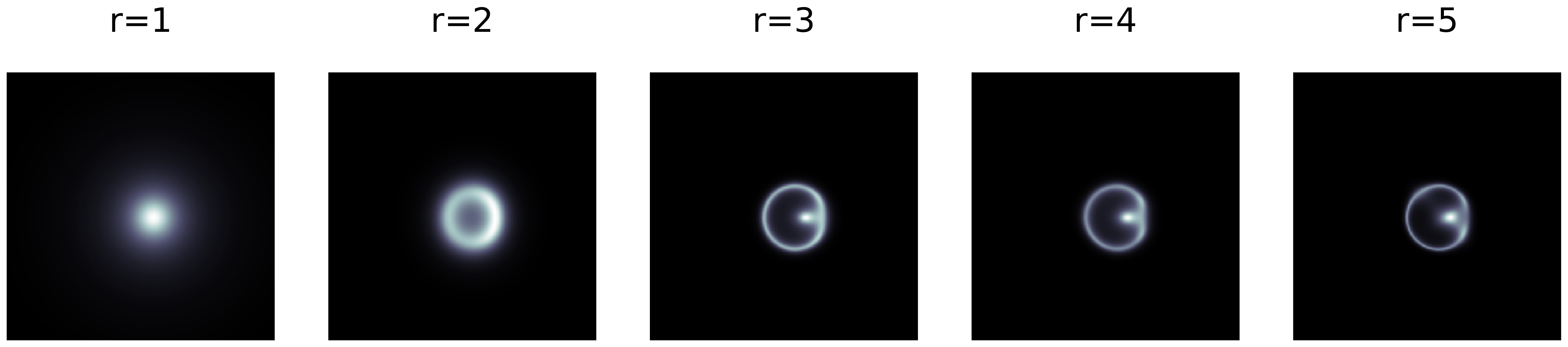}}
    \includegraphics[width=0.9\textwidth]{{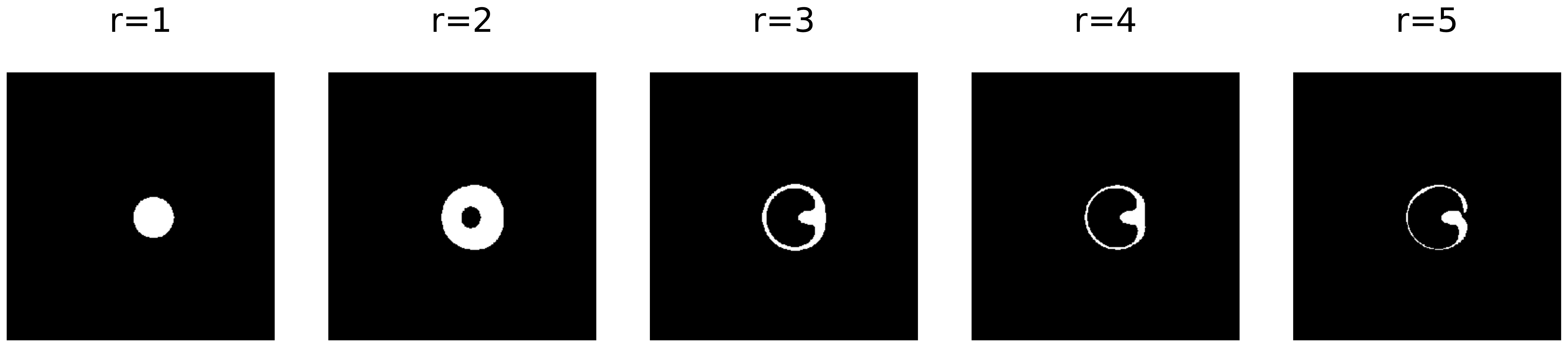}}
    \caption{Barycenter of pacmen.}
    \label{img:barycenter-CD-pacmans}
  \end{subfigure}%

    \caption{Estimation of the support of the barycenter. For each subfigure: Christoffel function (top), and support after thresholding (bottom)}
    \label{img:baryceter-CD}
    \end{figure}



\subsection{Gromov-Wasserstein discrepancy and barycenters}

Here we illustrate the computation of Gromov-Wasserstein discrepancies and barycenters. 
For our numerical tests, we consider empirical measures $\mu_1$ to $\mu_4$ associated with happy and sad smileys, see Figure \ref{img:smileys-GW22}.  Each measure corresponds to $1000$ independent samples from a mixture of three uniform measures with equal weights 1/3, the first two measures being supporting on the eyes, the third measure having the mouth as support. The mouth is here an algebraic set with zero Lebesgue measure. Measure $\mu_2$ (resp. $\mu_4$) is the push-forward of $\mu_1$ (resp. $\mu_3$) by an isometry, so that 
$GW_{2,2}^2(\mu_1,\mu_2) = 
GW_{2,2}^2(\mu_3,\mu_4) = 0$.  In this section, for the formulation of moment problems, we relied on Matlab libraries  tensap \cite{NGG2020} and GloptiPoly \cite{henrion2009gloptipoly}.

\begin{figure}[h]
    \centering
    \begin{subfigure}{.24\textwidth} \centering
    \includegraphics[scale=0.12]{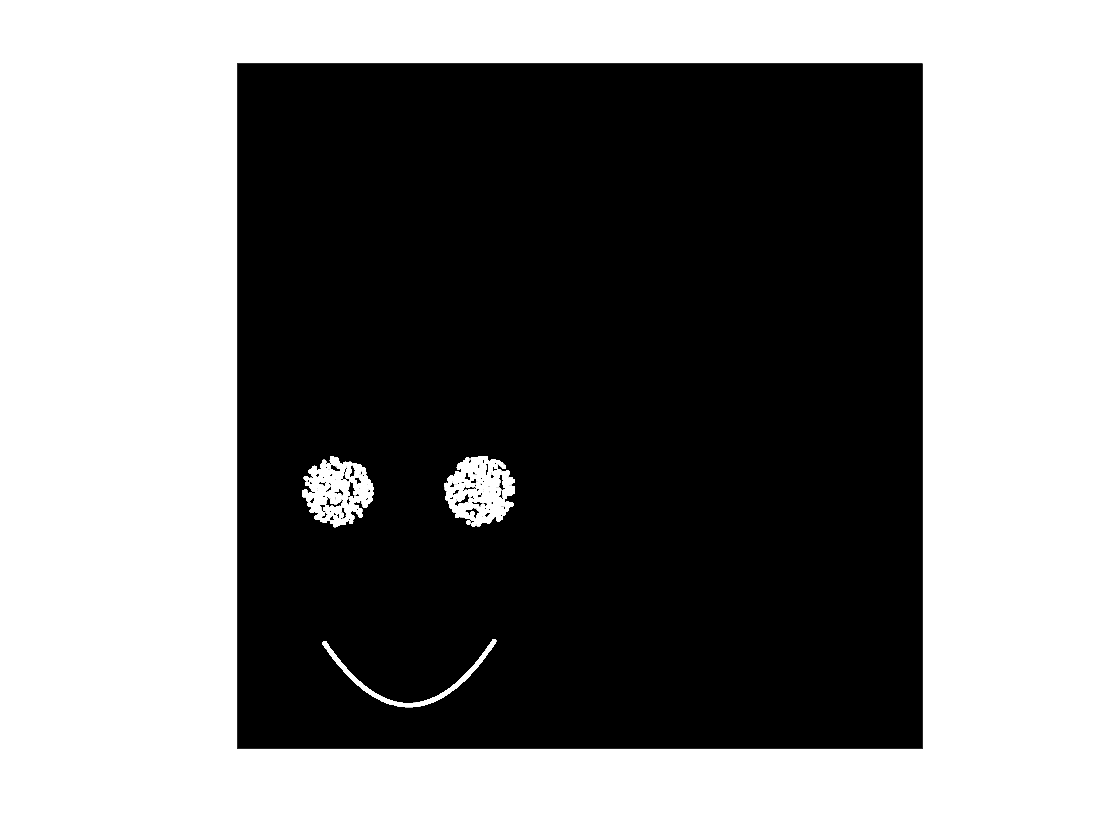}
    \caption{$\mu_1$}
    \end{subfigure}
    \begin{subfigure}{.24\textwidth}\centering
    \includegraphics[scale=0.12]{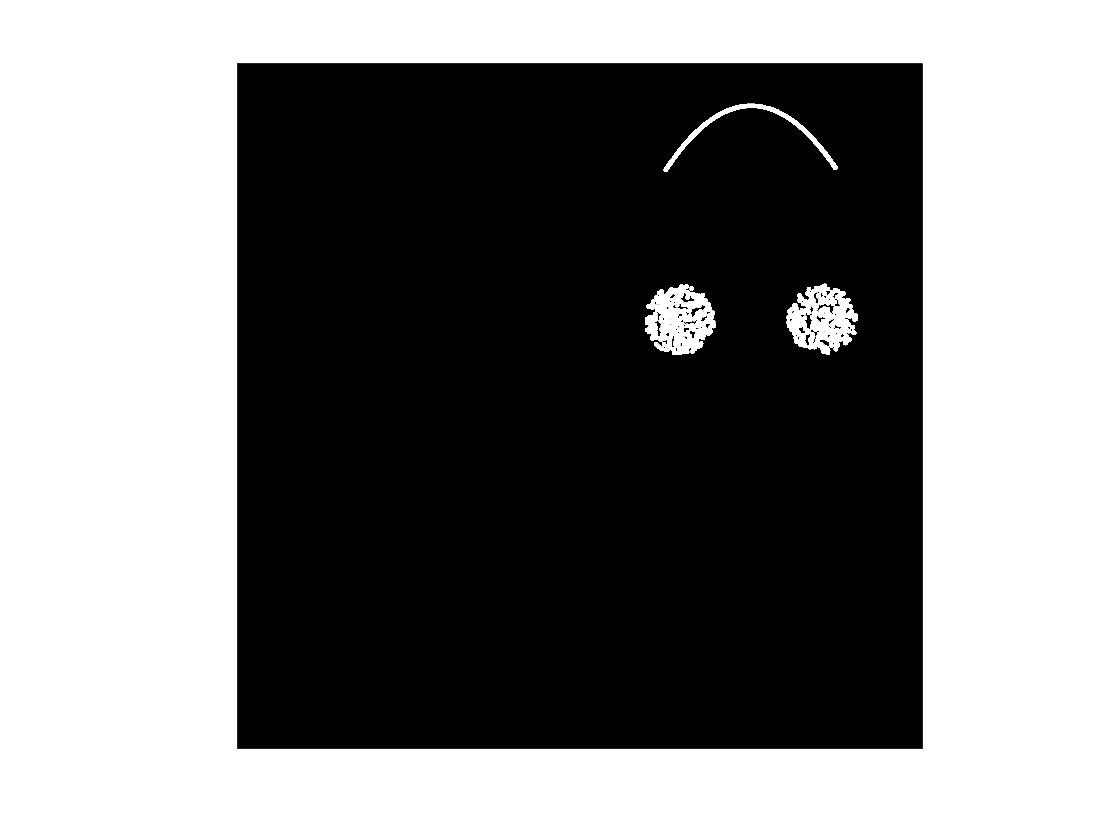}
    \caption{$\mu_2$}
    \end{subfigure}
    \begin{subfigure}{.24\textwidth}  \centering
     \includegraphics[scale=0.12]{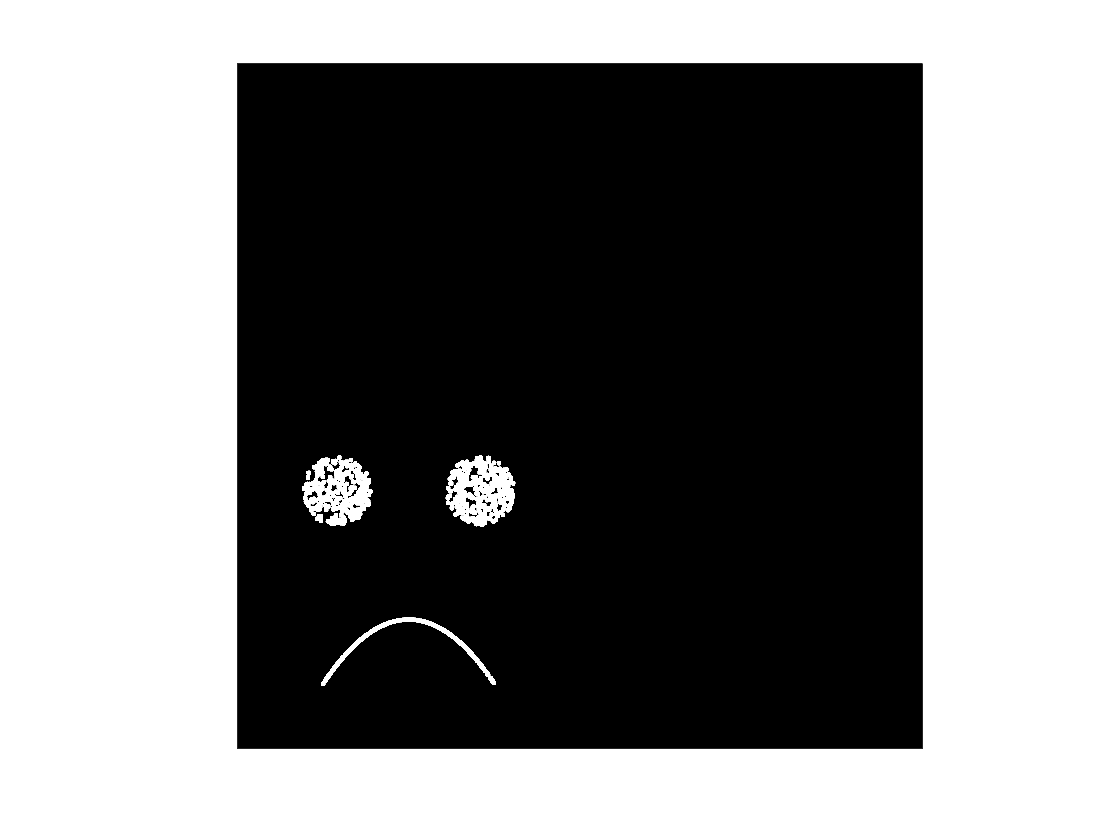}
    \caption{$\mu_3$}
    \end{subfigure}
    \begin{subfigure}{.24\textwidth}  \centering
    \includegraphics[scale=0.12]{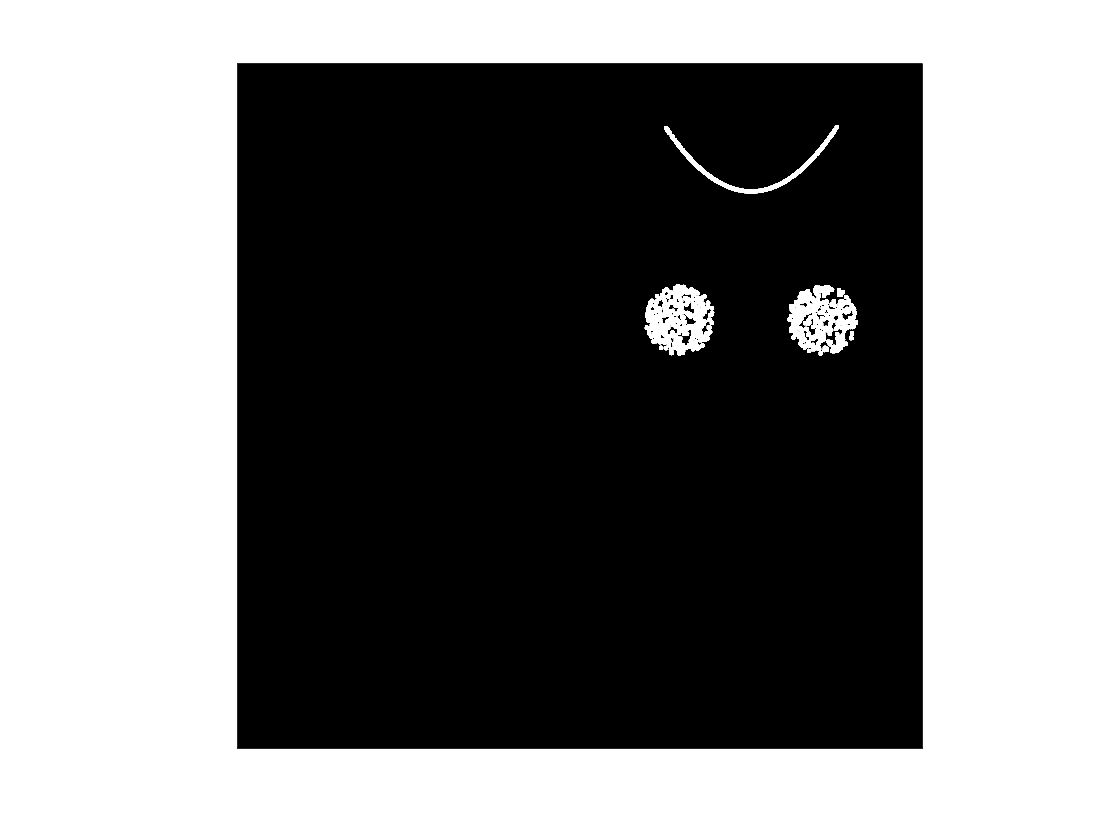}
   \caption{$\mu_4$}
    \end{subfigure}
    \caption{Empirical measures $\mu_1$ to $\mu_4$ for the examples with Gromov-Wasserstein discrepancies.}
    \label{img:smileys-GW22}
\end{figure}

\subsubsection{Gromov-Wasserstein discrepancy $GW_{2,2}$}
We here illustrate the estimation of the discrepancies $GW_{2,2}(\mu_i,\mu_j)$.
For a given relaxation order $r$, we initialize the truncated moment sequence $y^{(0)}$ with the truncated moments $m(\mu_i \otimes \mu_j)$  of the product measure $\mu_i \otimes \mu_j$. Then we construct a sequence of truncated moments $y^{(k)}$, $k\ge 1$, by a fixed point algorithm,  where $y^{(k)}$   minimizes   
$ y\mapsto L_{aug}^{GW_{2,2}}(y \otimes y^{(k-1)})$ over the truncated moment sequences $y$ satisfying the moment sequence condition and marginal constraints. As shown in Figure \ref{fig:GW22_convergence_fixed_point} for a given relaxation order, the fixed point algorithm converges rapidly, after roughly 4 to 5 iterations. Note that for $(i,j)$ equal to $ (1,2)$ and $(3,4)$, the objective function converges to a plateau of order $10^{-13}$, very close to zero in double precision.
\begin{figure}[h]
\centering
 \includegraphics[scale=0.2]{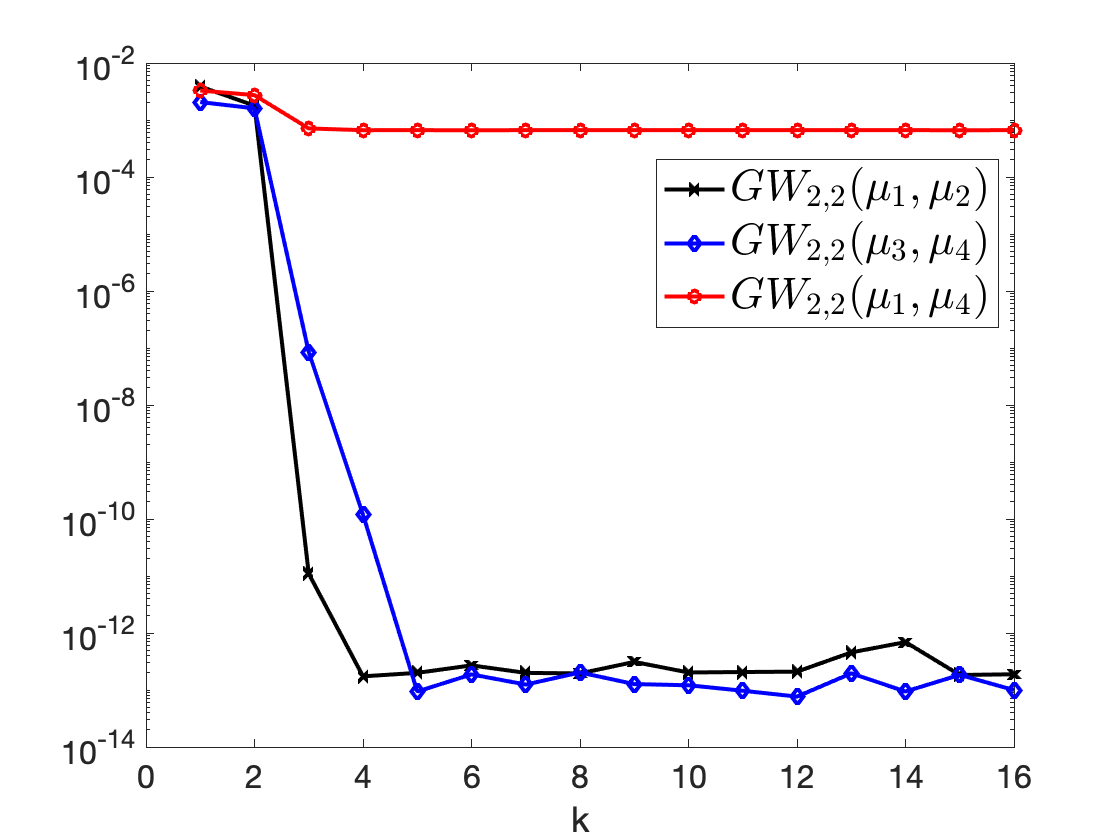} 
\caption{Convergence of $L_{aug}^{GW_{2,2}}(y^{(k)} \otimes y^{(k)}) = L^{GW_{2,2}}( y^{(k)}) $ for the estimation of 
$GW_{2,2}^2(\mu_1,\mu_2)$, $GW_{2,2}^2(\mu_3,\mu_4)$, $GW_{2,2}^2(\mu_1,\mu_4)$, for relaxation order $r=5$.}
\label{fig:GW22_convergence_fixed_point}
\end{figure}

Next, we provide estimations of discrepancies $GW_{2,2}(\mu_i,\mu_j)$ obtained at convergence of the fixed point algorithm. 
In Table \ref{tab:GW22_r}, we show the estimations of $GW_{2,2}(\mu_i,\mu_j)$ obtained for different relaxation orders. 
The obtained estimations for  $GW_{2,2}(\mu_1,\mu_2)$ and $GW_{2,2}(\mu_3,\mu_4)$ converge slowly with the relaxation order but are very small already for very small relaxation orders.  
The estimation of $GW_{2,2}(\mu_1,\mu_4)$ rapidly converges with the relaxation order. 
 
\begin{table}[h]
\centering
\begin{tabular}{|c|c|c|c|c|c|}
\hline
& $r=2$& $r =3$ & $r=4$ & $r=5$ & $r=6$\\
\hline 
$GW_{2,2}(\mu_1,\mu_2)$ &$1.3\, 10^{-6}$& $9.6\, 10^{-7}$ & $7.3\, 10^{-7}$ & $4.3\, 10^{-7}$ & $5.2\, 10^{-7} $
\\
\hline
$GW_{2,2}(\mu_3,\mu_4)$ &$2.3\, 10^{-6}$& $1.8\, 10^{-6}$ & $1.0\, 10^{-6}$ & $3.1\, 10^{-7}$ & $3.7\, 10^{-7} $
\\
\hline
$GW_{2,2}(\mu_1,\mu_4)$ &$2.21\, 10^{-2}$& $2.47\, 10^{-2}$ & $2.56\, 10^{-2}$ & $2.57\, 10^{-2}$ & $2.57\, 10^{-2} $
\\
\hline
\end{tabular}
\caption{Estimations of $GW_{2,2}(\mu_1,\mu_2)$, $GW_{2,2}(\mu_3,\mu_4)$ and  $GW_{2,2}(\mu_1,\mu_4)$ for relaxation orders $r=2$ to $6$.}
\label{tab:GW22_r}
\end{table}
%
%

\subsubsection{Gromov-Wasserstein barycenters $GW_{2,2}$}

 We now turn to the computation of Gromov-Wassertein barycenters, using discrepancy $GW_{2,2}$. 
 We consider the computation of the barycenters of the empirical measures $\mu_1$ and $\mu_2$ illustrated on Figure \ref{img:smileys-GW22}. The experiments are here for illustrative purpose.

For a given relaxation order $r$, we have to solve 
the optimization problem \eqref{gromov-barycenter-Ly} over truncated sequences $y$, $y_1$ and $y_2$, where $y_1$ has as  marginals $y$ and the truncated moments of $\mu_1$, and $y_2$ has as marginals $y$ and the truncated moments of $\mu_2$. The objective functional can be rewritten $\lambda L_{aug}^{GW_{2,2}}(y_1 \otimes y_1) + (1-\lambda) L_{aug}^{GW_{2,2}}(y_2 \otimes y_2)$, with $\lambda \in [0,1]$. For the solution of the optimization problem, we rely on a fixed point algorithm which constructs sequences  of truncated moment $y^{(k)}$, $y^{(k)}_1$ and $y^{(k)}_2$, $k\ge 1$,  such that  $(y^{(k)},y^{(k)}_1,y^{(k)}_2)$ minimizes 
$(y,y_1,y_2) \mapsto \lambda L_{aug}^{GW_{2,2}}(y_1 \otimes y_1^{(k)}) + (1-\lambda) L_{aug}^{GW_{2,2}}(y_2 \otimes y_2^{(k)})$ over truncated sequences satisfying marginal constraints and moment sequence conditions. For the initialization of $y^{(0)}$, we take the truncated moments of either $\mu_1$ or $\mu_2$ (depending on the value of $\lambda$) and for 
$y^{(0)}_1$ (resp. $y^{(0)}_2$), we take the tensor product of $y^{(0)}$ and the truncated moments of $\mu_1$ (resp. $\mu_2$). This algorithm converges rather slowly and should clearly be improved. However, it allows us to illustrate the  potential of the proposed  approach. 
The results are given at iteration $100$.
 
%
 Figures \ref{fig:GWbarycenter-christoffel} and \ref{fig:GWbarycenter-support} illustrate respectively the estimated supports and Christoffel functions of barycenters $\bary((\mu_1,\mu_2),\lambda)$ for different values of $\lambda$ and orders of relaxation $r$. We observe a rather fast convergence with $r$, at least for $\lambda \notin \{0,1\}$. The obtained barycenters for $\lambda \notin \{0,1\}$.

  \begin{figure}
    \centering
    \setlength{\tabcolsep}{0pt}
 \begin{tabular}{ccccc}
 &   $r=2$ & $r=3$ & $r=4$ & $r=5$ \\
\raisebox{35pt}{$\lambda=0$}
 &  \includegraphics[scale=0.24]{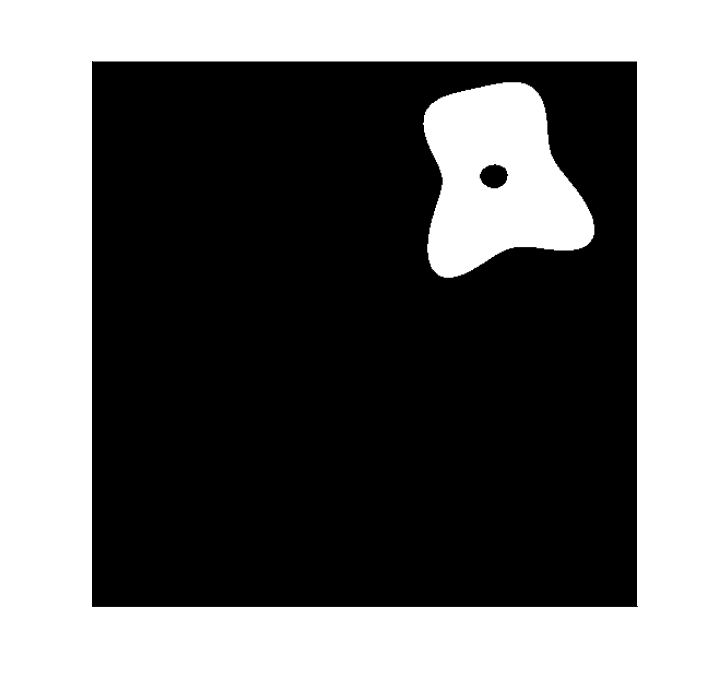}
 & \includegraphics[scale=0.24]{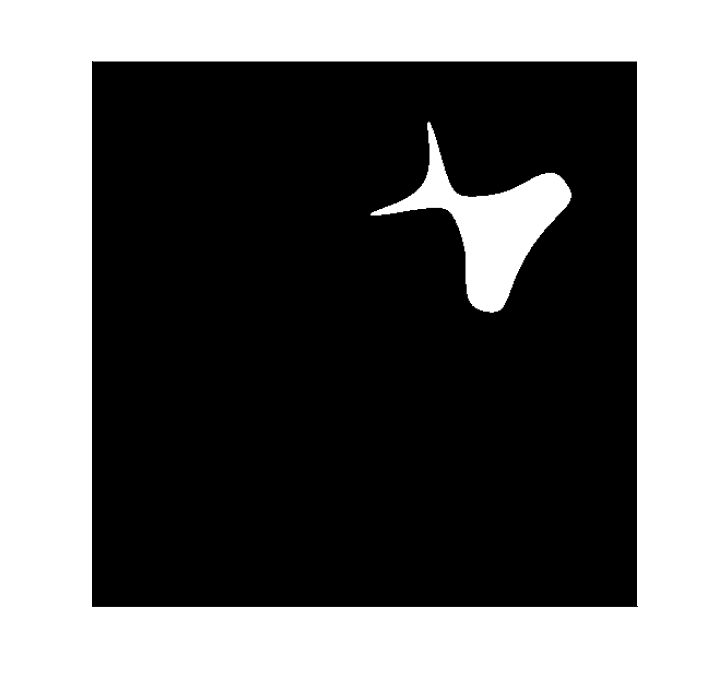} 
 & \includegraphics[scale=0.24]{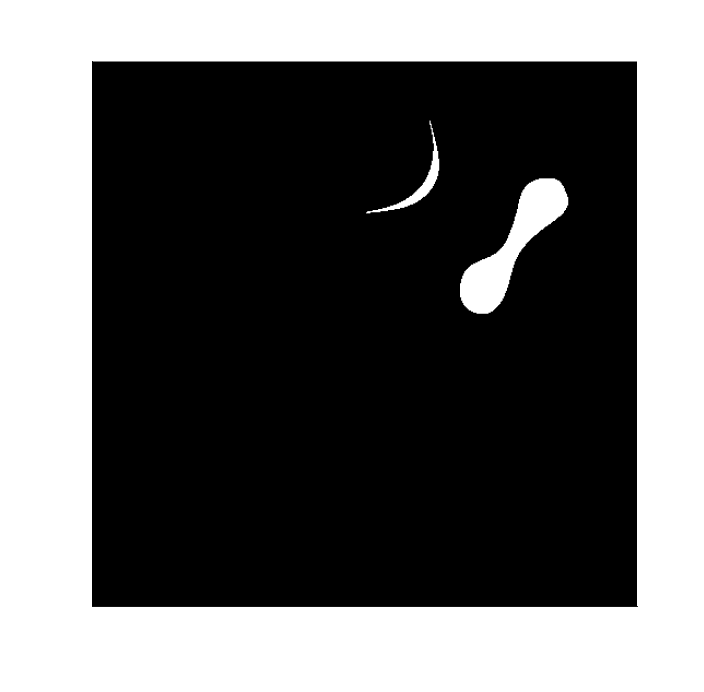} 
 & \includegraphics[scale=0.24]{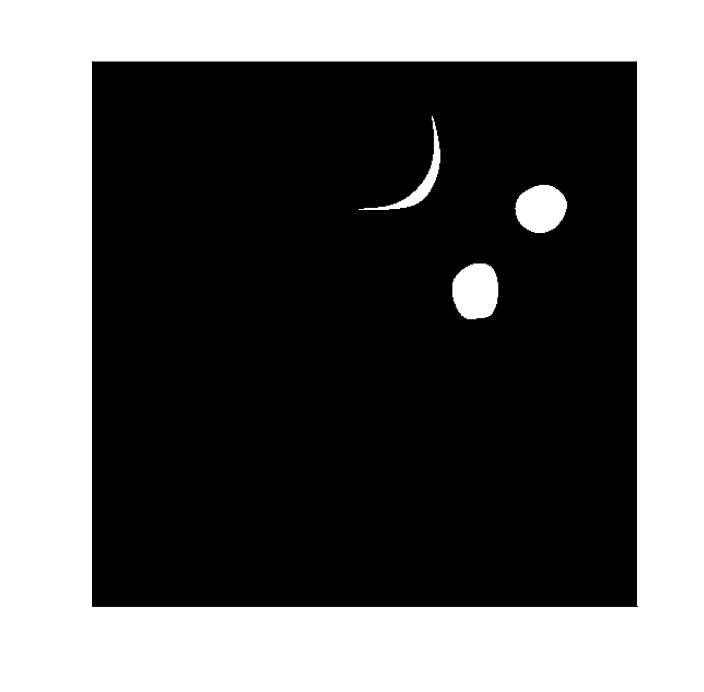} \\
 \raisebox{35pt}{$\lambda=0.25$}
 & \includegraphics[scale=0.24]{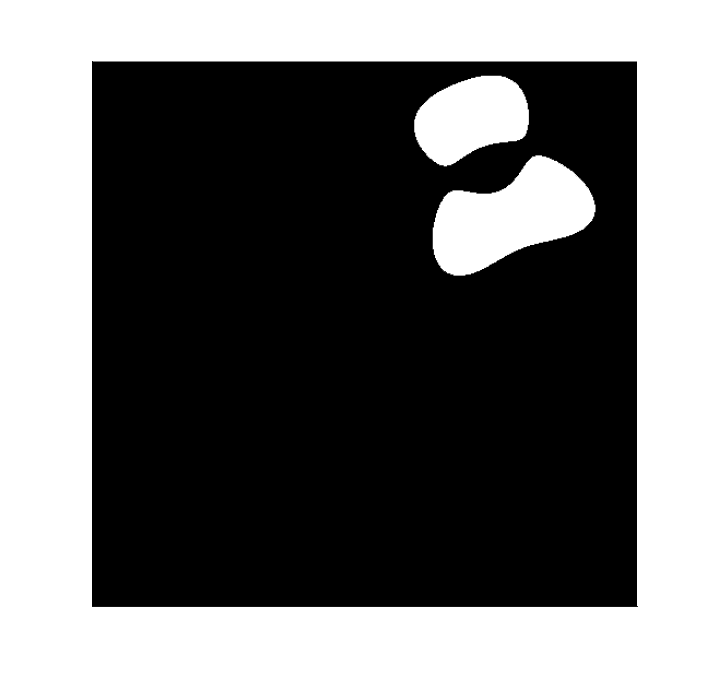}
 & \includegraphics[scale=0.24]{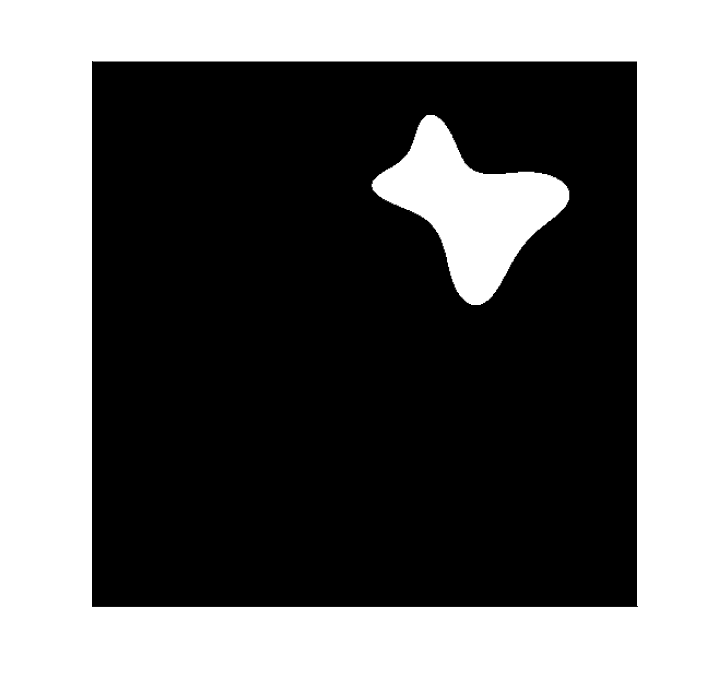} 
 & \includegraphics[scale=0.24]{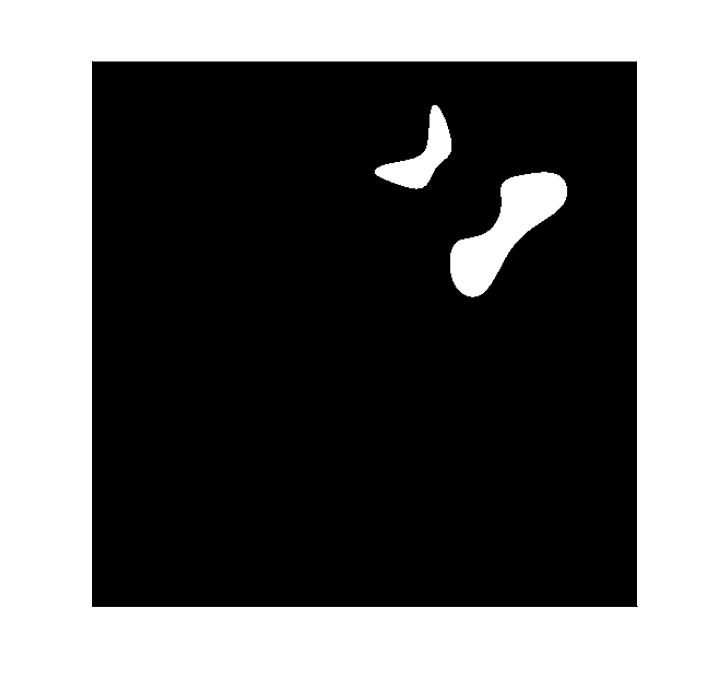} 
 & \includegraphics[scale=0.24]{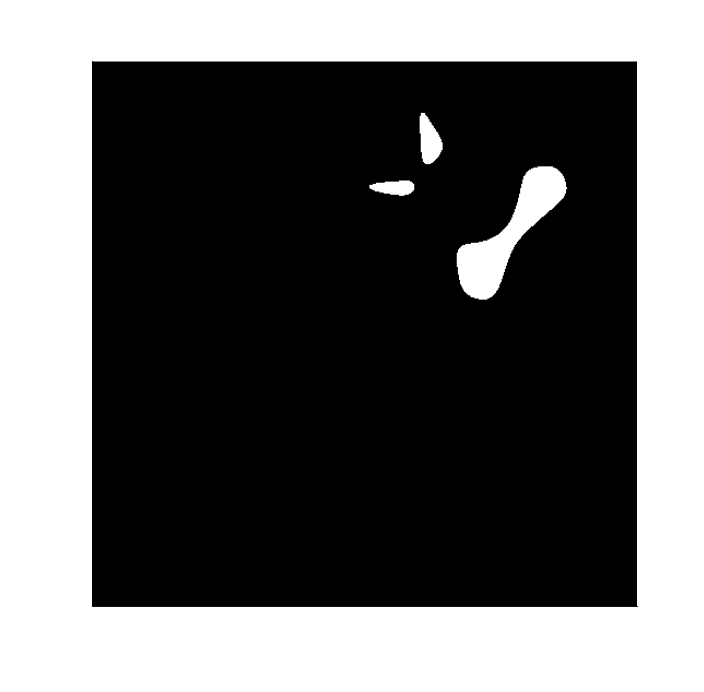} \\
 \raisebox{35pt}{$\lambda=0.5$}
 & \includegraphics[scale=0.24]{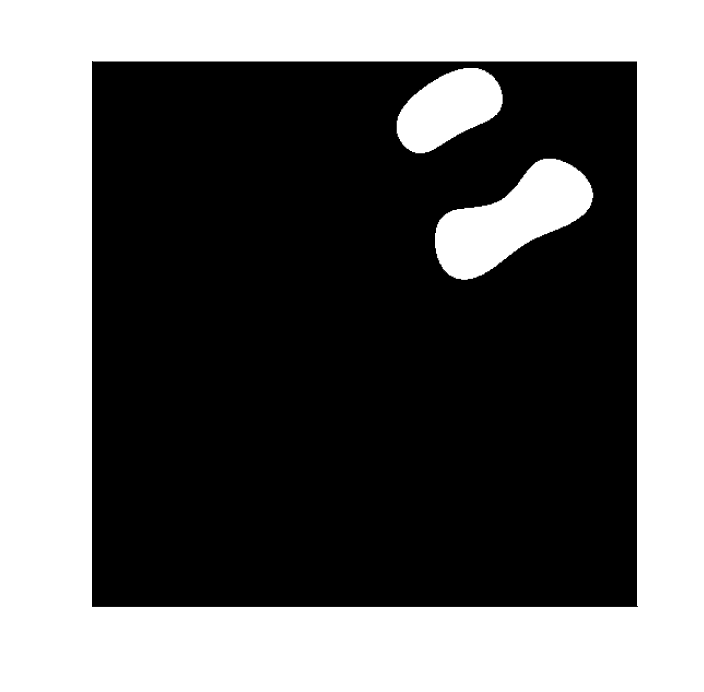}
 & \includegraphics[scale=0.24]{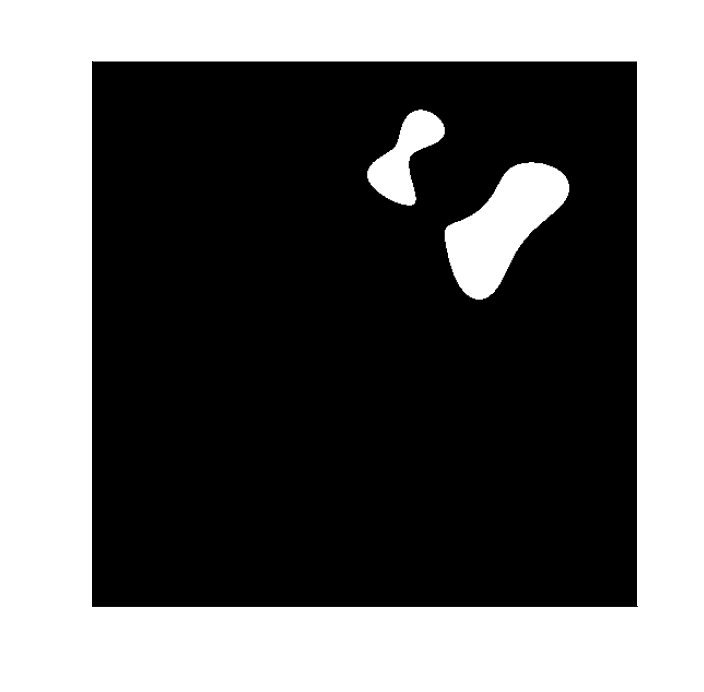} 
 & \includegraphics[scale=0.24]{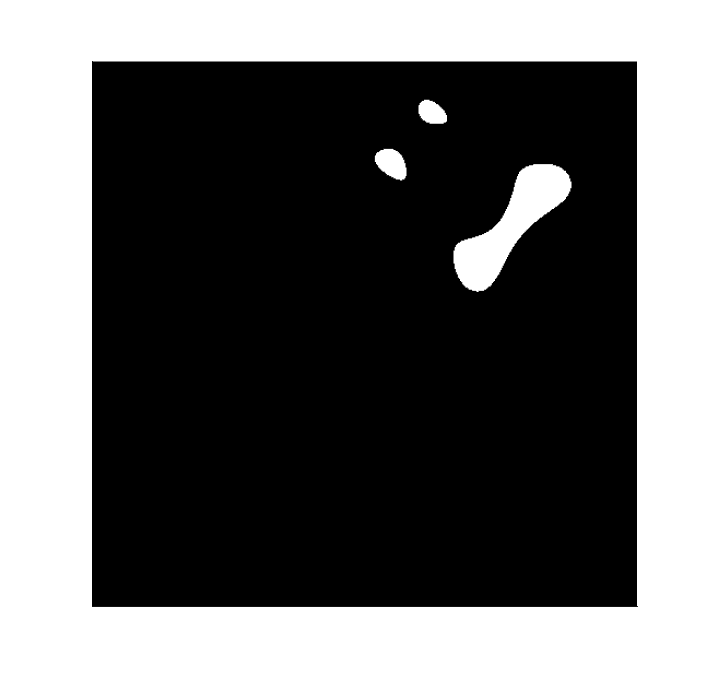} 
 & \includegraphics[scale=0.24]{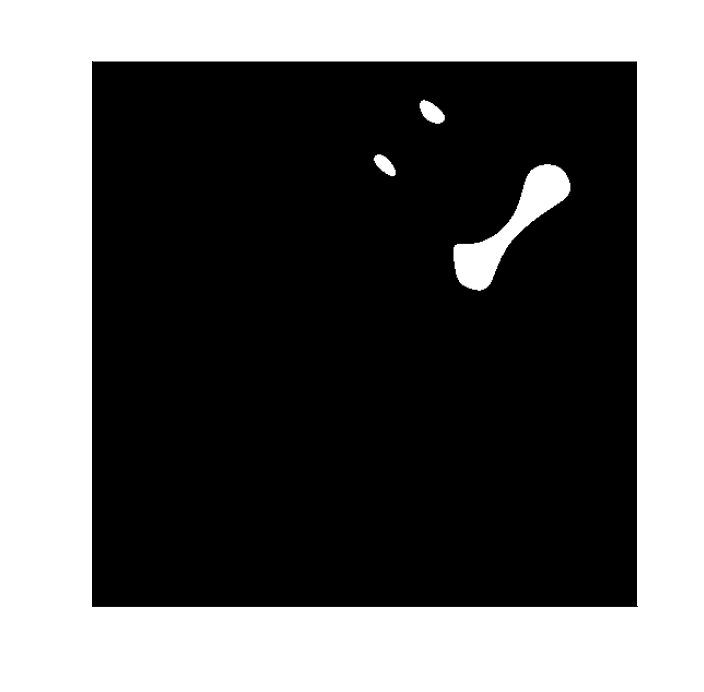} \\
\raisebox{35pt}{$\lambda=0.75$}
 & \includegraphics[scale=0.24]{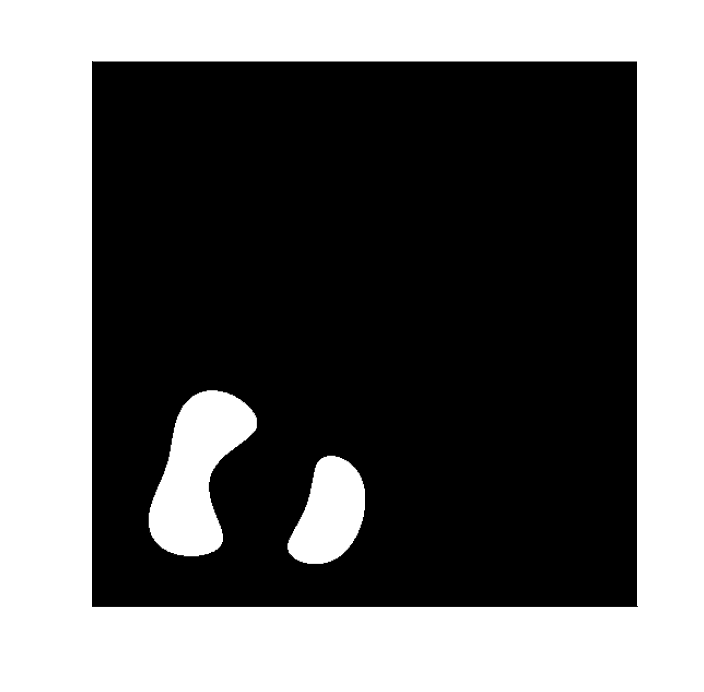}
 & \includegraphics[scale=0.24]{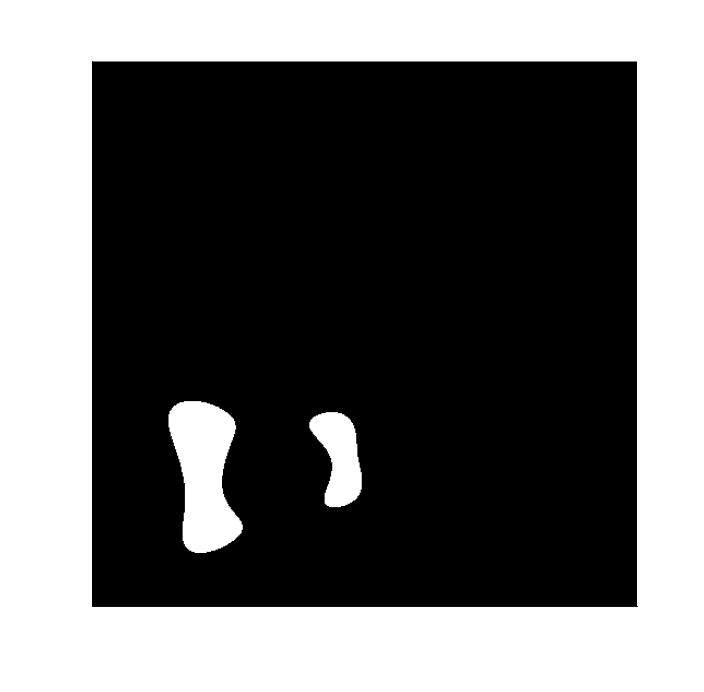} 
 & \includegraphics[scale=0.24]{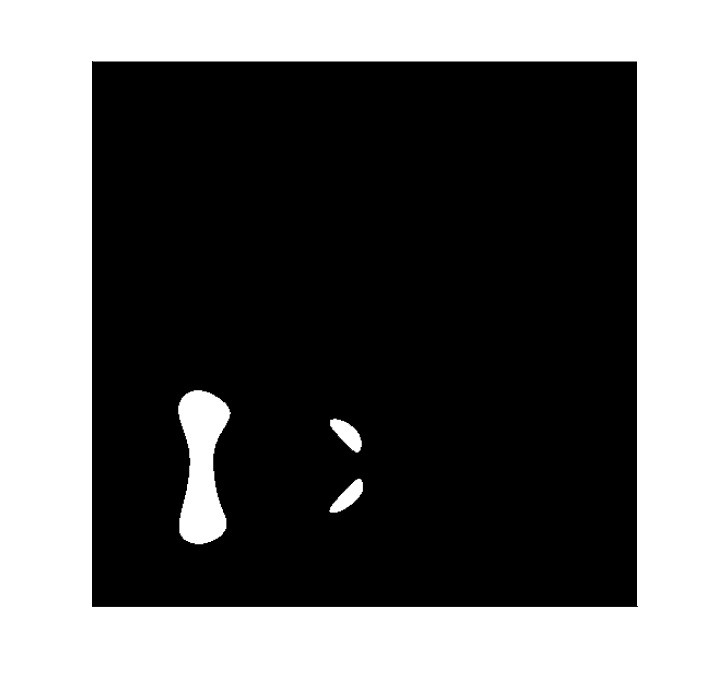} 
 & \includegraphics[scale=0.24]{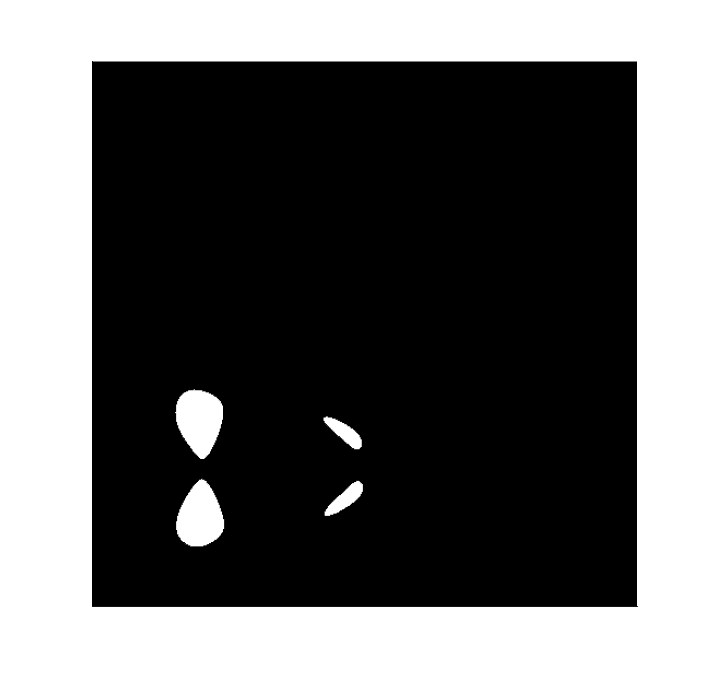} \\
\raisebox{35pt}{$\lambda=1$}
 & \includegraphics[scale=0.24]{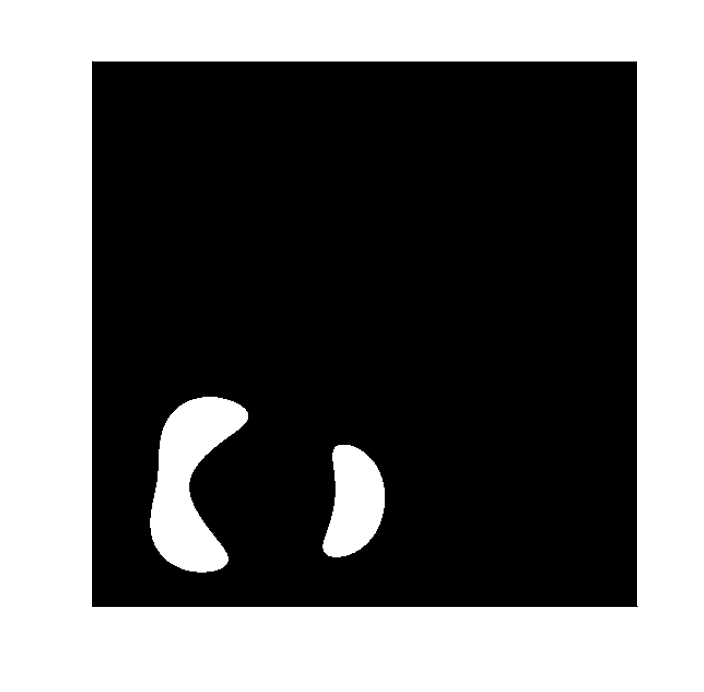}
 & \includegraphics[scale=0.24]{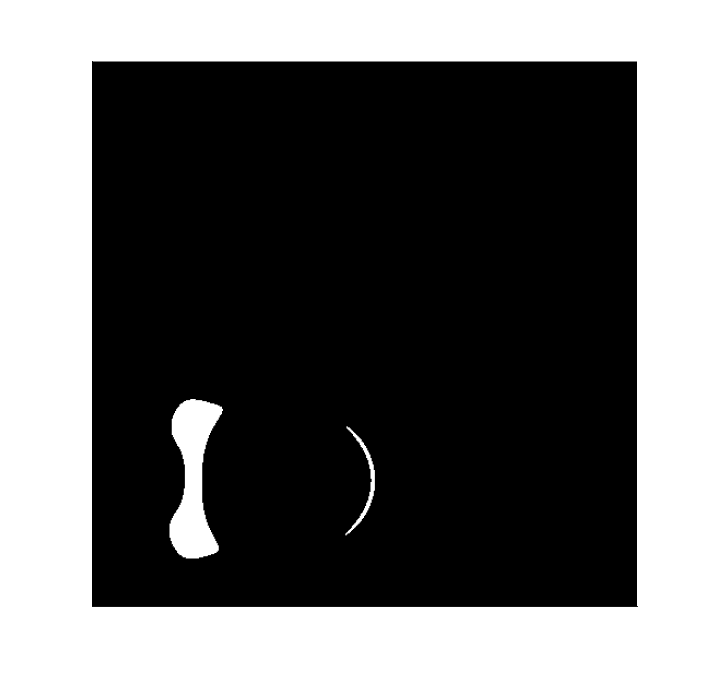} 
 & \includegraphics[scale=0.24]{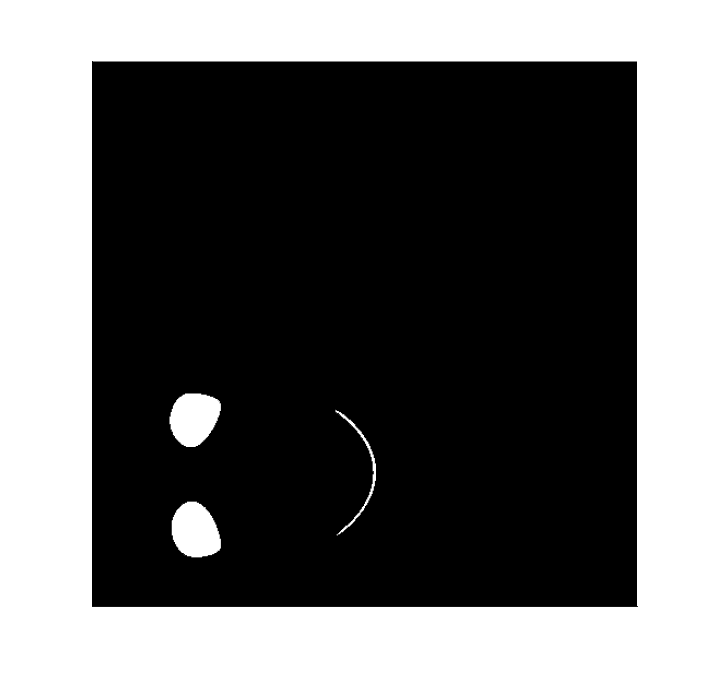} 
 & \includegraphics[scale=0.24]{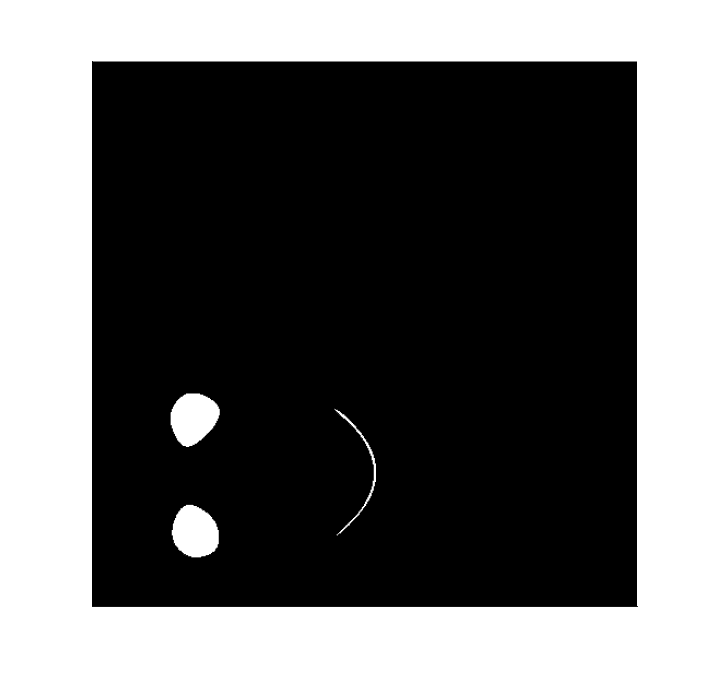} 
   \end{tabular}
 \caption{Gromov-Wasserstein baycenters: estimated support of barycenter $\bary((\mu_1,\mu_2),\lambda)$ for 
 $\lambda \in \{0,0.25,0.5,0.75,1\}$ and different orders of relaxation $r\in \{2,3,4,5\}$.}
 \label{fig:GWbarycenter-support}
\end{figure}

  \begin{figure}
    \centering
    \setlength{\tabcolsep}{0pt}
 \begin{tabular}{ccccc}
 &   $r=2$ & $r=3$ & $r=4$ & $r=5$ \\
\raisebox{35pt}{$\lambda=0$}
 &  \includegraphics[scale=0.24]{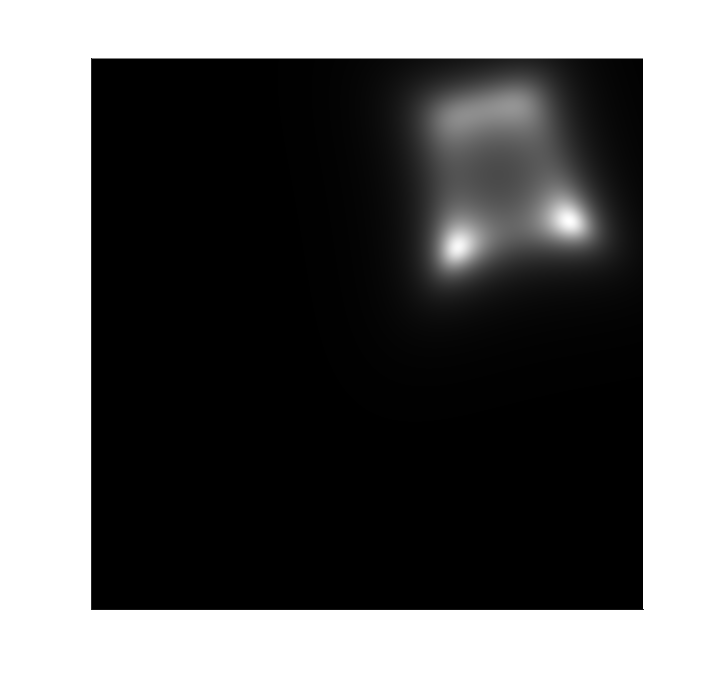}
 & \includegraphics[scale=0.24]{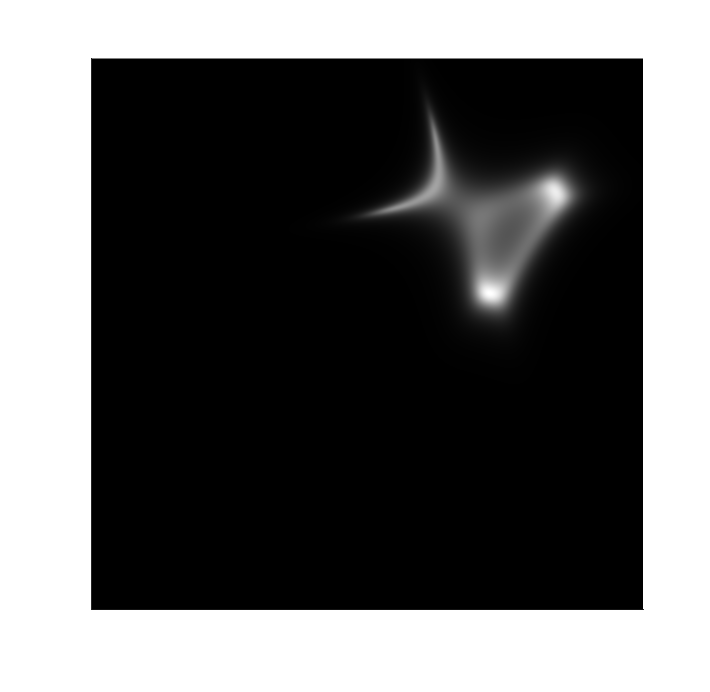} 
 & \includegraphics[scale=0.24]{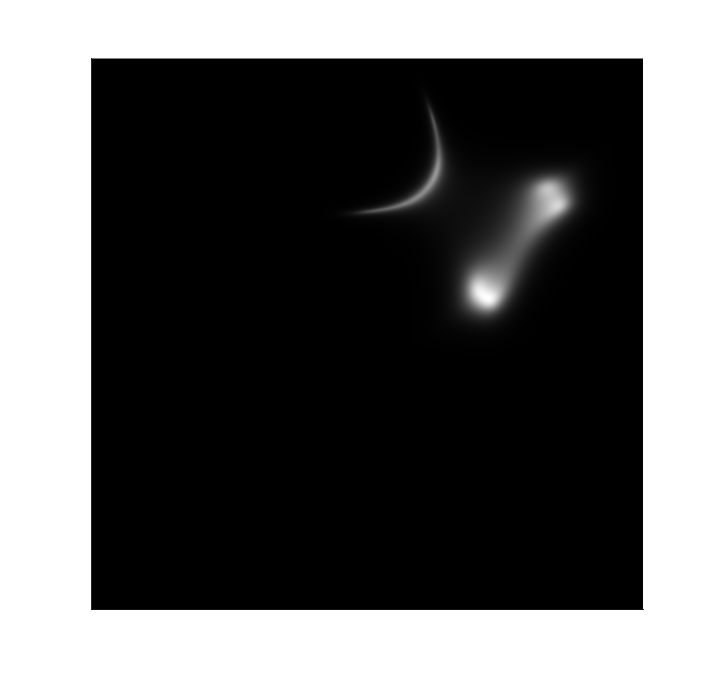} 
 & \includegraphics[scale=0.24]{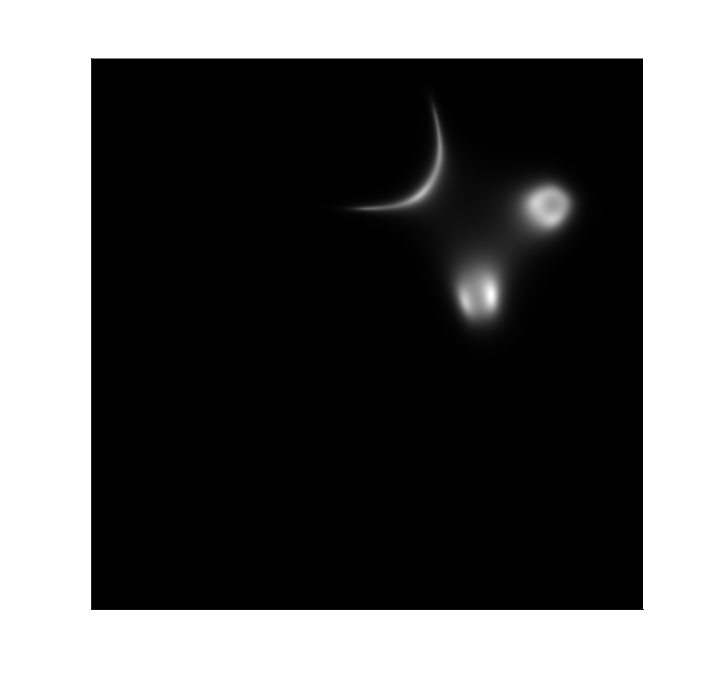} \\
 \raisebox{35pt}{$\lambda=0.25$}
 & \includegraphics[scale=0.24]{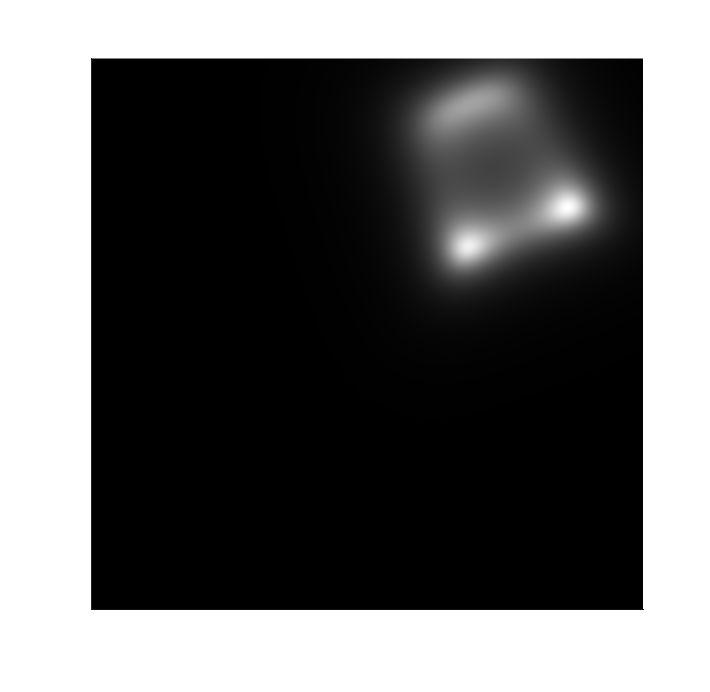}
 & \includegraphics[scale=0.24]{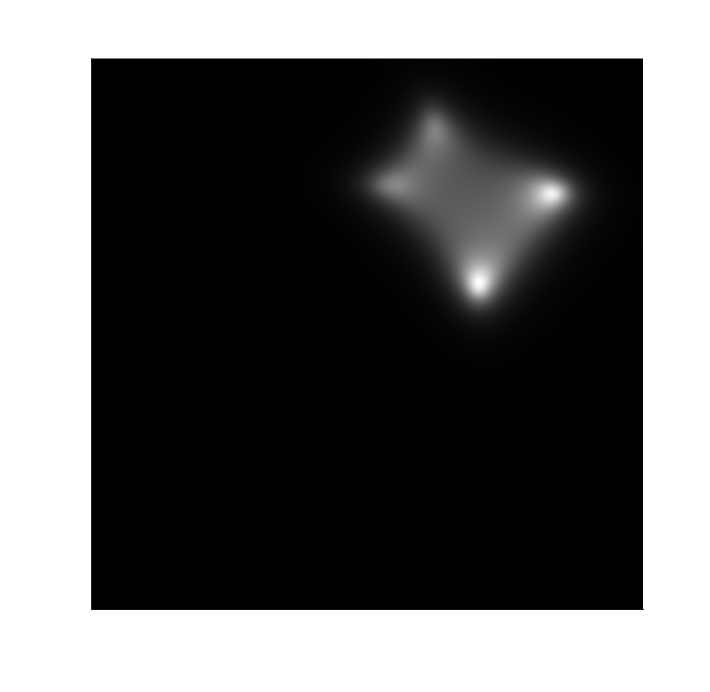} 
 & \includegraphics[scale=0.24]{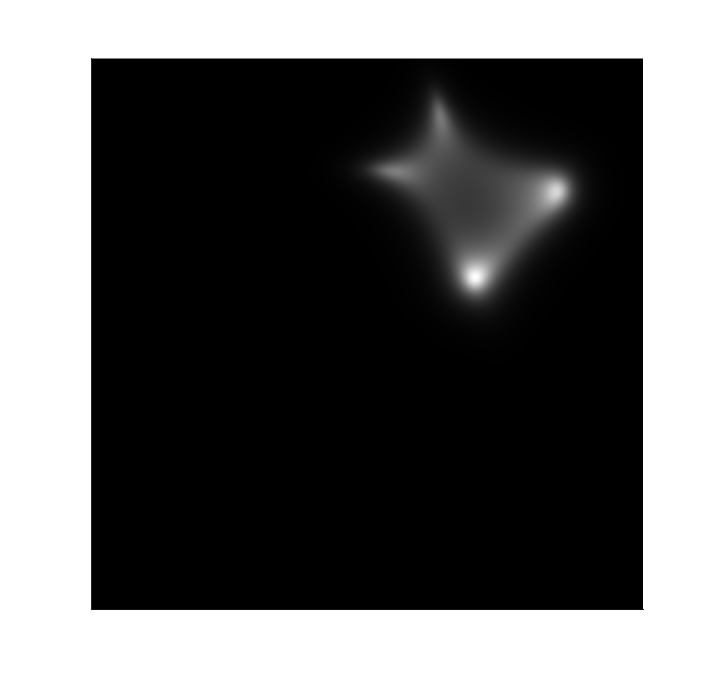} 
 & \includegraphics[scale=0.24]{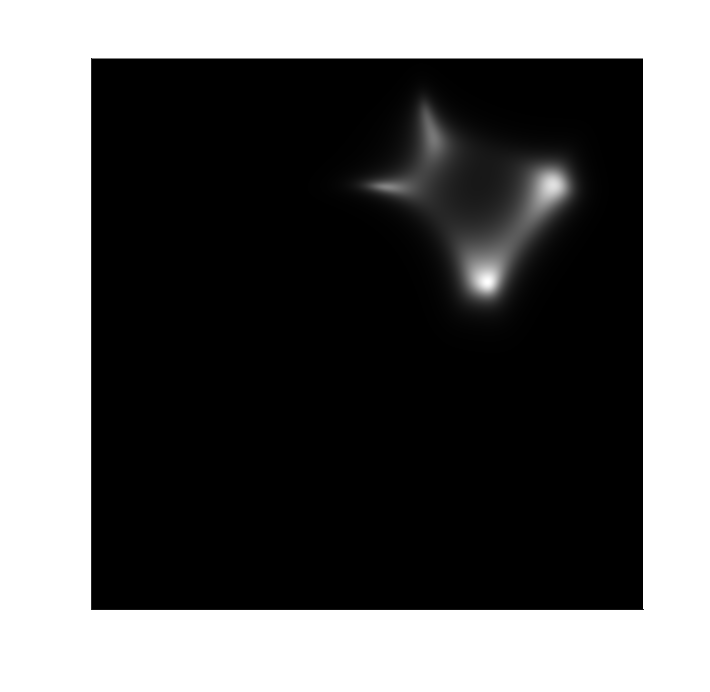} \\
 \raisebox{35pt}{$\lambda=0.5$}
 & \includegraphics[scale=0.24]{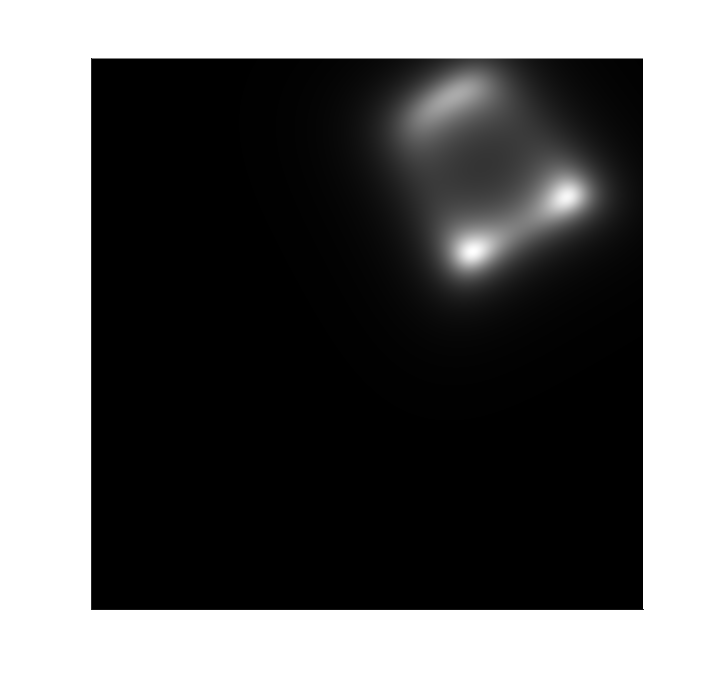}
 & \includegraphics[scale=0.24]{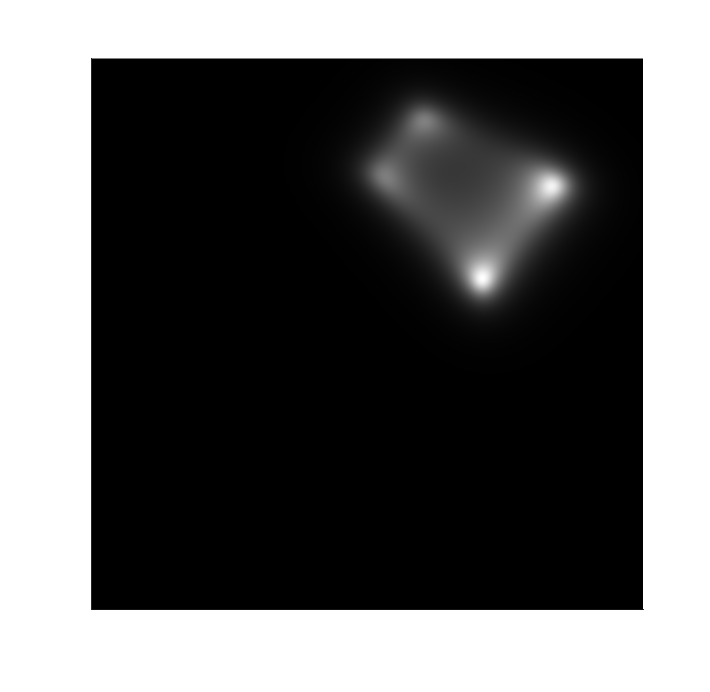} 
 & \includegraphics[scale=0.24]{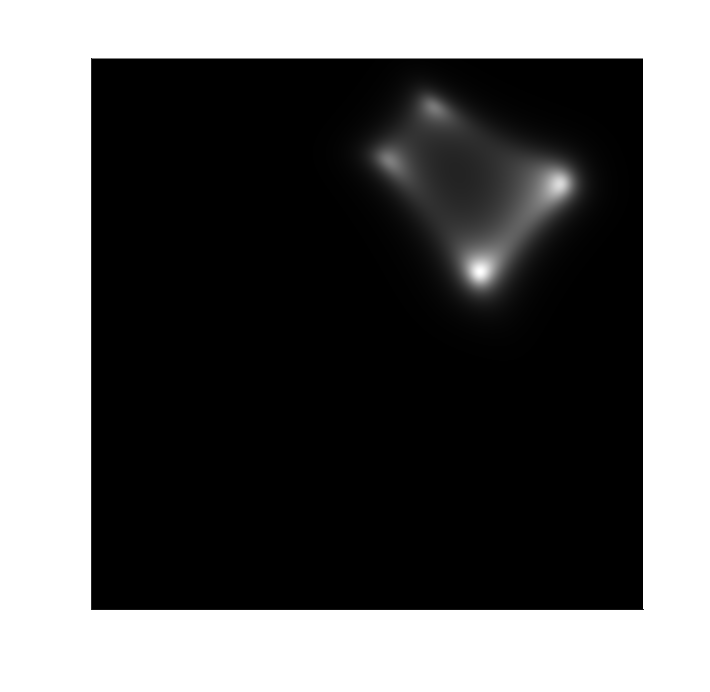} 
 & \includegraphics[scale=0.24]{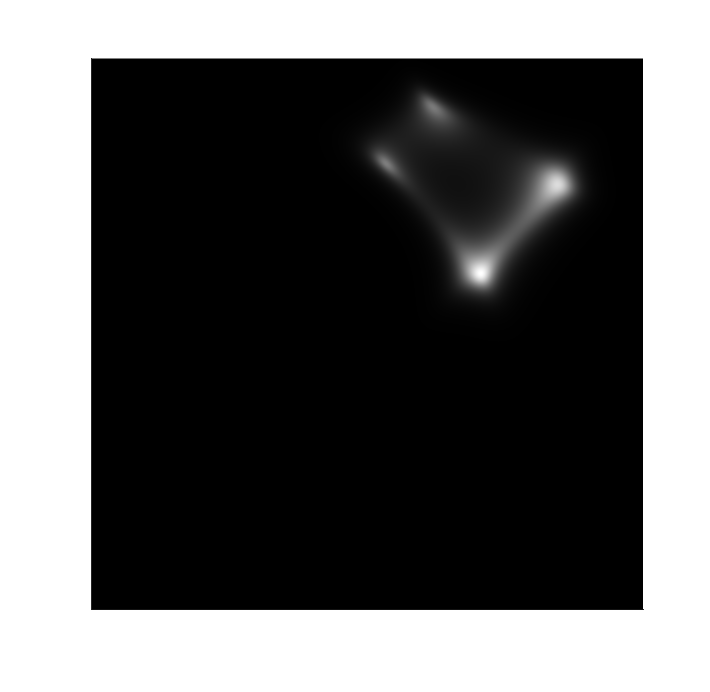} \\
\raisebox{35pt}{$\lambda=0.75$}
 & \includegraphics[scale=0.24]{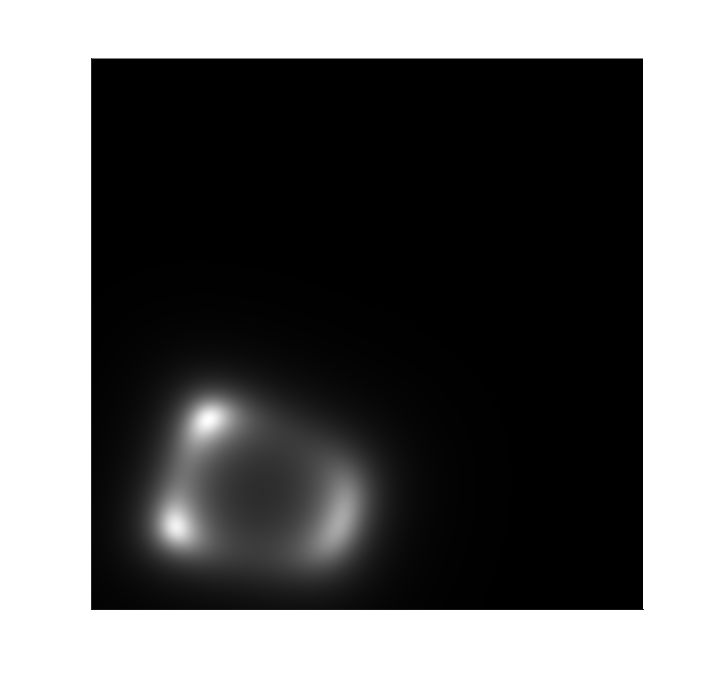}
 & \includegraphics[scale=0.24]{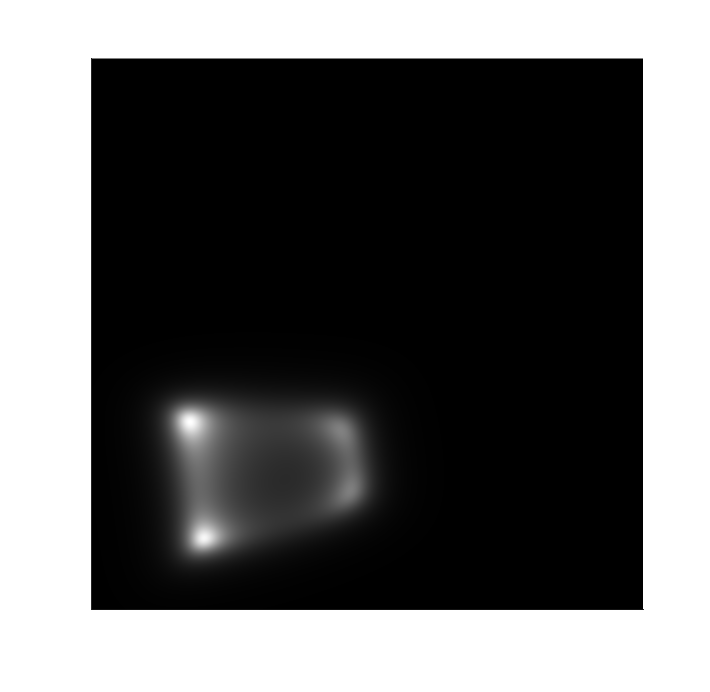} 
 & \includegraphics[scale=0.24]{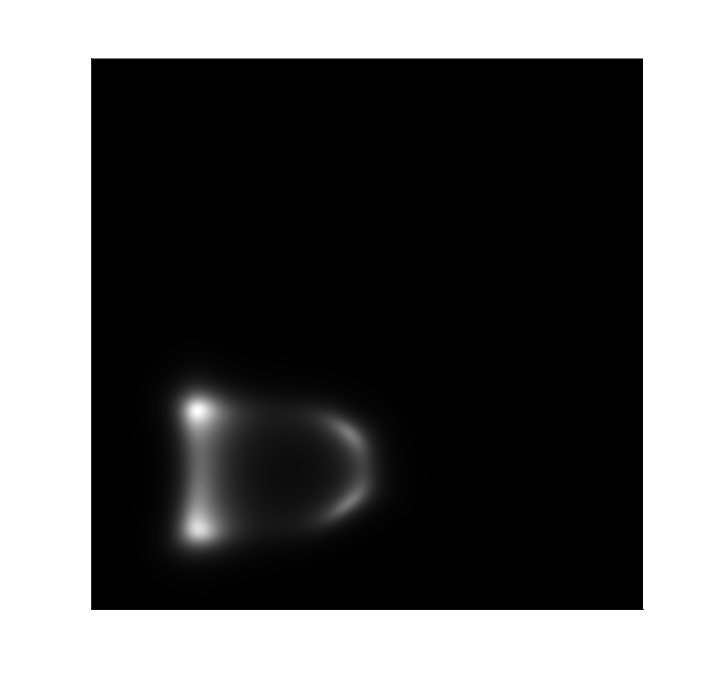} 
 & \includegraphics[scale=0.24]{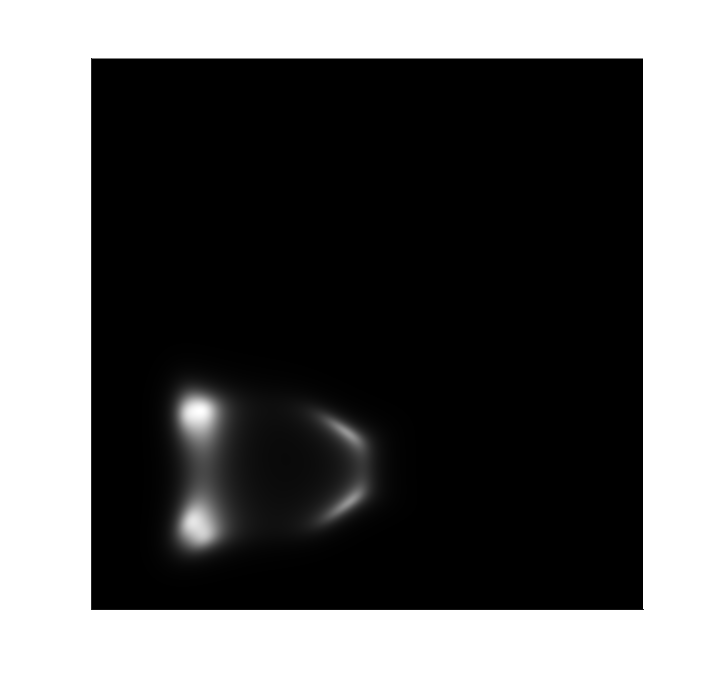} \\
\raisebox{35pt}{$\lambda=1$}
 & \includegraphics[scale=0.24]{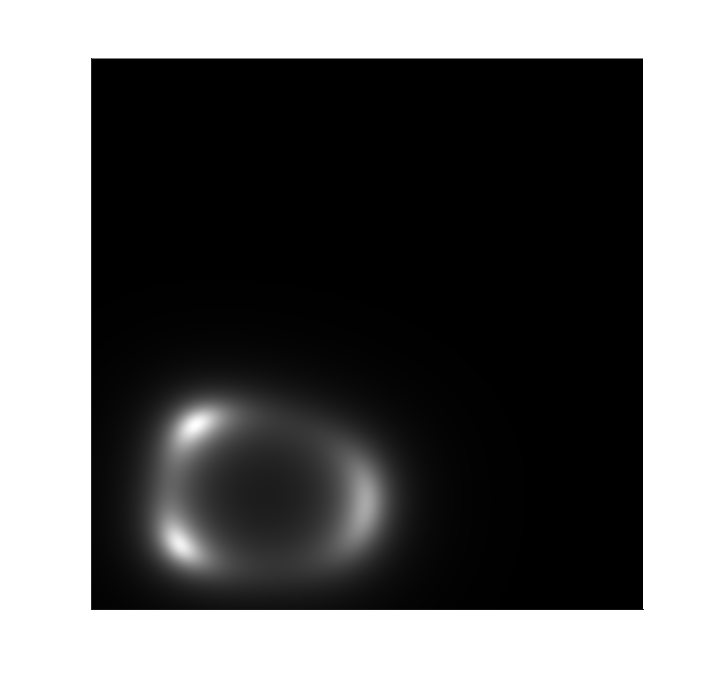}
 & \includegraphics[scale=0.24]{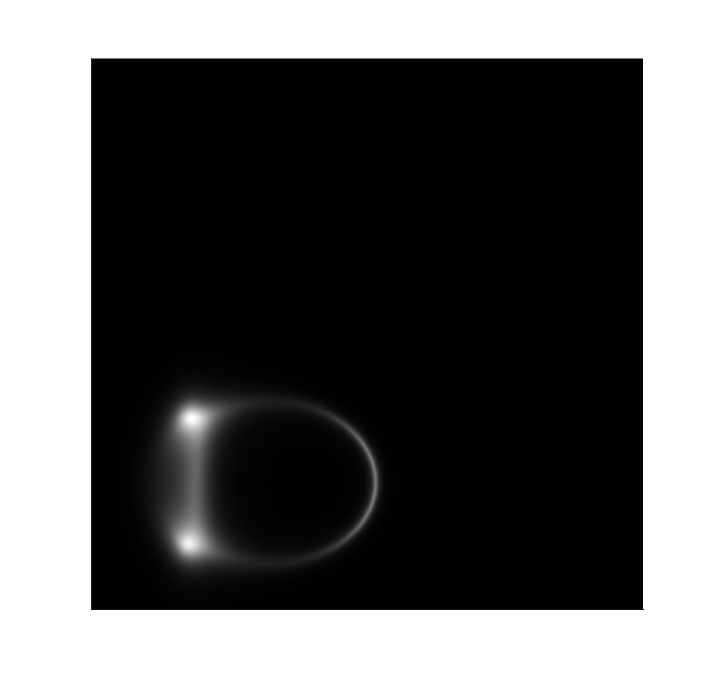} 
 & \includegraphics[scale=0.24]{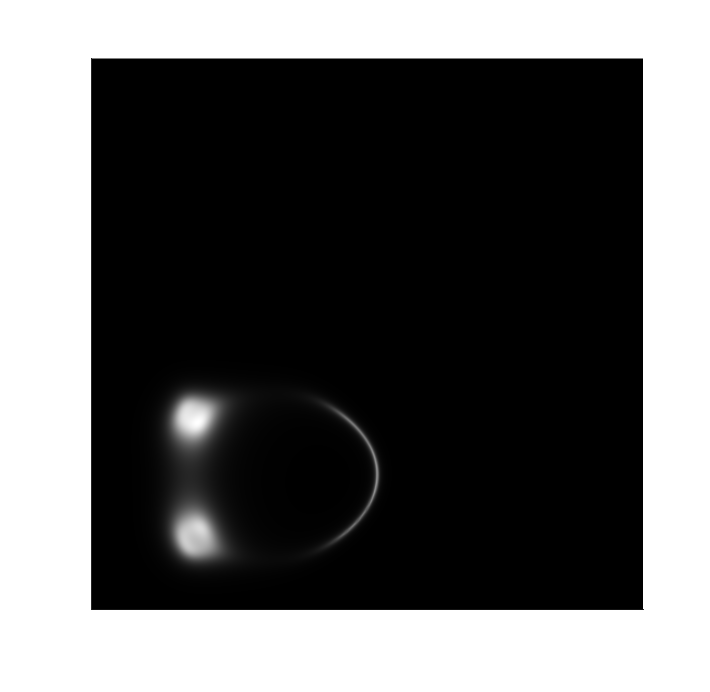} 
 & \includegraphics[scale=0.24]{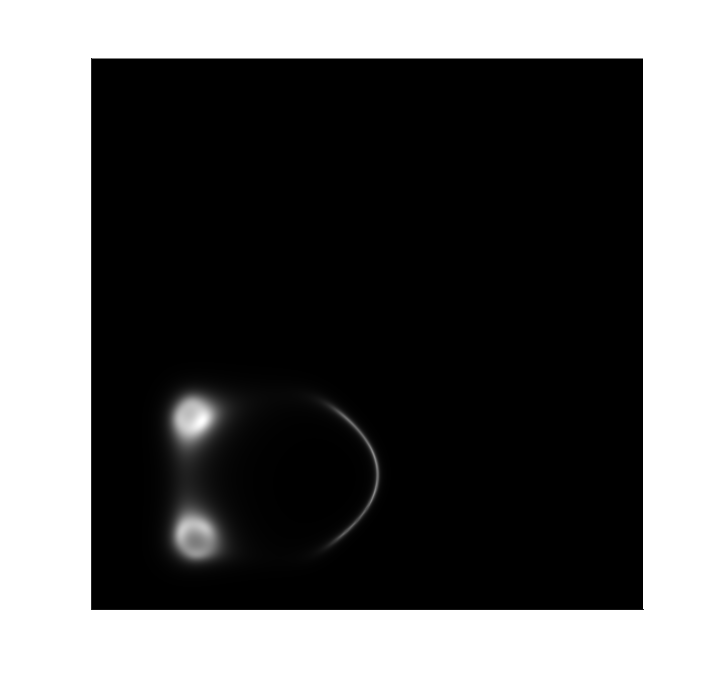} 
   \end{tabular}
 \caption{Gromov-Wasserstein baycenters: estimated Christoffel function of barycenter $\bary((\mu_1,\mu_2),\lambda)$ for 
 $\lambda \in \{0,0.25,0.5,0.75,1\}$ and different orders of relaxation $r\in \{2,3,4,5\}$.}
 \label{fig:GWbarycenter-christoffel}
\end{figure} 


\section{Conclusions}
We have shown the theoretical foundations to compute most common optimal transport problems with a moment-SoS approach. The numerical results reveal that the method gives very good accuracies for the estimation of the value of the loss functions with very low orders.  The support of concentrated measures can efficiently be estimated with relatively low polynomial orders. This feature seems particularly appealing because it could be leveraged to cleverly allocate degrees of freedom in optimal transport solvers that approximate the optimal transport plan. 

\bibliographystyle{plain}
\bibliography{references.bib}
\end{document}